\documentclass[preprint,authoryear,11pt]{elsarticle}
\RequirePackage[T1]{fontenc}
\usepackage{color}
\usepackage[latin1]{inputenc}
\usepackage[T1]{fontenc}
\RequirePackage{amsfonts,amsthm,amsmath}
\RequirePackage[colorlinks,citecolor=blue,urlcolor=blue]{hyperref}
\usepackage{a4wide}
\usepackage{here}
\usepackage{float}
\usepackage{natbib}
\newtheorem{theo}{Theorem}

\newtheorem{lemme}[theo]{Lemma}
\newtheorem{remi}{Remark}

\hypersetup{pdfstartview=FitH}

\def\e{{\mathbb E}}
\def\1{{\mathbb 1}}

\def\cov{\mathop{\rm Cov}\nolimits}    
\def\vec{\mathop{\rm vec}\nolimits}  
\usepackage{graphicx}
\begin{document}
\begin{frontmatter}
\title{Diagnostic checking in multivariate ARMA models with dependent errors using normalized residual autocorrelations\\[0.5cm]
}

\author[Yacouba]{Y. Boubacar Ma{\"\i}nassara}
\address[Yacouba]{Universit\'e Bourgogne Franche-Comt\'e, \\
Laboratoire de math\'{e}matiques de Besan\c{c}on, \\ UMR CNRS 6623, \\
16 route de Gray, \\ 25030 Besan\c{c}on, France.}
\ead{mailto:yacouba.boubacar\_mainassara@univ-fcomte.fr}

\author[Yacouba]{B. Saussereau}
\ead{mailto:bruno.saussereau@univ-fcomte.fr}

%
%
%

%
\begin{abstract}
In this paper we derive the asymptotic distribution of normalized
residual empirical autocovariances and autocorrelations under weak
assumptions on the noise. We propose  new portmanteau statistics for vector autoregressive moving-average (VARMA) models with uncorrelated but non-independent innovations by using a self-normalization approach. We establish the asymptotic distribution of the proposed statistics. This asymptotic distribution is quite different from the usual chi-squared approximation  used under the independent and identically distributed assumption on the noise,  or the weighted sum of independent chi-squared random variables obtained under nonindependent innovations.  A set of Monte Carlo experiments and an application to the daily returns of the CAC40 is presented.
\end{abstract}
\end{frontmatter}
%
\section{Introduction}
In econometric application, the univariate autoregressive moving-average (ARMA) framework is very restrictive. Consequently the class of vector autoregressive moving-average (VARMA) models are commonly used in time series analysis and econometrics. It describes the possible
cross-relationships between the time series and not only the properties of the individual time series (see \cite{L2005}). Moreover,  the time series literature shows growing interest in non-linear models. Roughly speaking, these models are those on which the independence assumption on the noise is relaxed. This motivates the framework of our investigations that we further describe below.

We consider a $d$-dimensional stationary process $X=(X_t)_{t\in\mathbb Z}$
that satisfies a VARMA$(p,q)$ representation of the form
\begin{equation}
\label{VARMA}
A_{00}{X}_t-\sum_{i=1}^pA_{0i}{X}_{t-i}=B_{00}{\epsilon}_t-\sum_{i=1}^qB_{0i}{\epsilon}_{t-i},\quad
\forall t\in \mathbb Z.
\end{equation}
Without loss of generality, one may assume that $A_{0p}$ and $B_{0q}$ are not both equal to the null matrix.
The representation (\ref{VARMA}) is said to be a VARMA$(p,q)$ representation if the noise process ${\epsilon}=({\epsilon}_t)_{t\in\mathbb Z}$ with  ${\epsilon}_t=(\epsilon_{1t},\dots,\epsilon_{dt})'$ is a multivariate \emph{weak white noise}, that is, if it satisfies
\begin{itemize}
\item[\hspace*{1em} {\bf (A0):}]
\hspace*{1em} $\mathbb{E}({{\epsilon}}_t)=0$,
$\mbox{Var}\left({{\epsilon}}_t\right)=\Sigma_0$, and
$\mbox{Cov}\left({{\epsilon}}_t,{{\epsilon}}_{t-h}\right)=0$ for all
$t\in\mathbb{Z}$ and all $h\neq 0$.
\end{itemize}
We also assume that  $\Sigma_0$ is non singular.
It is customary to say that $({X}_t)_{t\in \mathbb Z}$ is a
strong VARMA$(p,q)$ model if $(\epsilon_t)_{t\in \mathbb Z} $ is  an  independent and identically distributed (iid for short) sequence of random vectors with zero mean and non singular variance matrix (\emph{i.e. strong white noise}). 
When $A_{00}=B_{00}=I_d$, where $I_d$ denotes the identity matrix of order $d$, the VARMA$(p,q)$ representation (\ref{VARMA}) is said to be in reduced form. Otherwise, it is said to be structural. The structural forms are mainly used in econometrics to introduce instantaneous relationships between economic variables.  Of course, constraints are necessary for the identifiability of these representations.

The estimation of ARMA models is however much more difficult in the multivariate case than in univariate case. A first difficulty is that non trivial constraints on the parameters must be imposed for the identifiability of the parameters (see \cite{L2005}). Secondly, the implementation of standard estimation methods (for instance the
Gaussian quasi-maximum likelihood estimation) is not obvious because
this requires a constrained high-dimensional optimization (see \cite{L2005} for further details).
These technical difficulties certainly explain why the vector autoregressive (VAR for short) models are much more used than VARMA in applied context. This is also the reason why the statistical analysis of weak ARMA model is mainly limited to the univariate framework (see \cite{fz05}, for a review on weak univariate ARMA models).

After estimating the (V)ARMA process, the next important step in the (V)ARMA modeling consists in checking if the estimated model fits satisfactorily the data. This adequacy checking step validates or invalidates the choice of the orders $p$ and $q$. In VARMA$(p,q)$ models, the choice of $p$ and $q$ is particularly important as the number of parameters quickly increases with $p$ and $q$, which entails statistical difficulties.

In 1970, Box and Pierce  have proposed  (see \cite{bp70}) a goodness-of-fit test, the so-called portmanteau test, for univariate
strong ARMA models. A modification of their test has been proposed by Ljung and Box (see \cite{lb}).  It is nowadays one of the
most popular diagnostic checking tools in ARMA modeling of time series. Both of these tests are based on the residual empirical autocorrelations $\hat\rho(h)$ (see Section \ref{notat} for a precise notation) and they are defined by
\begin{equation}\label{bp}
 Q^{\textsc{bp}}_m=n\sum_{h=1}^m\hat\rho^2(h) \text{ and }   {Q}_m^{\textsc{lb}}=n(n+2)\sum_{h=1}^m\frac{\hat\rho^2(h)}{n-h},
\end{equation}
where $n$ is the length of the series and $m$ is a fixed integer.
The  statistic ${Q}_m^{\textsc{lb}}$ has the same asymptotic distribution as $Q_m^{\mathrm{BP}}$ and has the reputation of doing better for small or medium sized sample (see \cite{lb}).

Being inspired by the univariate portmanteau statistics defined in \eqref{bp},  Chitturi (see \cite{C1974}) and Hosking (see \cite{H1980})
have introduced  the following multivariate versions of the  portmanteau
statistics
\begin{eqnarray}\label{bpmulti}
{{Q}}_m^{\textsc{c}}=n\sum_{h=1}^m\mbox{Tr}\left(\hat\Gamma'(h)\hat\Gamma^{-1}(0)
\hat\Gamma(h)\hat\Gamma^{-1}(0)\right)\text{ and
}{{Q}}_m^{\textsc{h}}=\sum_{h=1}^m\frac{n^2}{(n-h)}\mbox{Tr}\left(\hat\Gamma'(h)\hat\Gamma^{-1}(0)
\hat\Gamma(h)\hat\Gamma^{-1}(0)\right),
 \end{eqnarray}
 where $\hat\Gamma(h)$ is the residual autocovariances matrices function of
 the multivariate process $X$, $\mbox{Tr}(\cdot)$ denotes the trace of a matrix and  $A'$ the transpose of a matrix $A$.

Under the assumption that the  noise sequence is iid, the standard test procedure consists in rejecting the null hypothesis of a (V)ARMA$(p,q)$ model if the statistics (\ref{bp}) or (\ref{bpmulti}) are larger than a certain quantile of a chi-squared distribution.

Henceforth, we deal with some models with uncorrelated but dependent noise process $\epsilon$, the so-called weak noise. For such models, the asymptotic distributions of the statistics defined in (\ref{bp}) or (\ref{bpmulti})  are no longer chi-square distributions but a mixture of chi-squared distributions, weighted by eigenvalues of  the asymptotic covariance matrix of the vector of autocorrelations (see \cite{frz,yac}).
Consequently, in order to obtain the asymptotic distribution of the portmanteau statistics under weak assumptions on the noise, one needs a consistent estimator of the asymptotic covariance matrix of the residual autocorrelations vector.  In the econometric literature the nonparametric kernel estimator, also called heteroscedastic autocorrelation consistent  estimator (see \cite{a_econ,newey}), is widely used to estimate  covariance matrices. However, this causes serious difficulties regarding the choice of the sequence of weights. An alternative  method consists in using a parametric autoregressive  estimate of the spectral density of
a stationary process. This approach, which has been studied in \cite{berk,haan},  is also facing the problem of choosing the truncation  parameter. Indeed, this method is based on an infinite autoregressive representation of the stationary process. So the choice of the order of truncation is crucial and difficult.
The methodology employed in \cite{frz} (an extension is proposed in \cite{yac} for VARMA models) presents these difficulties: it supposes to weight appropriately some empirical fourth-order moments by means of a window and a truncation point.
Recently, Shao generalized in \cite{s2011ET}, the test statistic based on the kernel-based spectral proposed by \cite{H1996}.
However, this test is also confronted to the problem of the choice of the bandwidth parameter.
Zhu and Li \cite{zl2015JE} also proposed a bootstrapped spectral test for checking the adequacy of weak ARMA models, in  which the limiting distribution depends on the unknown data generating process. They used a block-wise random weighting method to bootstrap their critical values which also  need the choice of the block size. When the noise process is observable, Lobato, Nankervis and Savin \cite{lobatoNS2001,lobatoNS2002} address the problem of testing the null hypothesis that a time series is uncorrelated up to some fixed order and propose an extension of the Box-Pierce statistic.

In this work, we  propose an alternative method where we do not estimate an asymptotic covariance matrix. It is based on a self-normalization based approach to construct a new test-statistic which is asymptotically distribution-free under the null hypothesis. The idea comes from Lobato (see \cite{lobato}) and has been already extended by \cite{kl2006,s2010JRSSBa,s2010JRSSBb,shaox} to more general frameworks. See also \cite{s2016} for a review on some recent developments on the inference of time series data using the self-normalized approach.

In our opinion, there are two major contributions in this work. The first one is to show that the Lobato test statistic can be extended to the residuals of weak ARMA and VARMA models.
The second one is to improve the results concerning the statistical analysis of weak ARMA and  VARMA
models by considering the self-normalization approach for the adequacy problem. Notice that the new tests can replace the standard ones when testing strong ARMA and VARMA models.

We briefly give some details on the test statistic that we introduce in this article. We denote by $A\otimes B$ the Kronecker product of two matrices $A$ and $B$.
Our new test-statistics, for VARMA  models, are defined by
\begin{equation}
{{Q}}_{m}^{\textsc{sn}}=
n\,\hat{{\rho}}_m'\left\{I_m\otimes(\hat S_e\otimes \hat S_e)\right\}
\hat C_{md^2}^{-1}\left\{I_m\otimes(\hat S_e\otimes \hat S_e)\right\}\hat{{\rho}}_m,  \label{bms}
\end{equation}
where $\hat C_{md^2}^{-1}$ is a normalization matrix (that is defined later in (\ref{hatcmVARMA})) and $\hat{S_e}$ is a diagonal matrix in which the $i$-th element is the sample estimate of the variance of the $i$-th coordinate of the multivariate noise process. The vector of the first $m$  sample autocorrelations is denoted by $\hat\rho_m$. We prove in Theorem \ref{sn2VARMA} that the asymptotic distribution of $Q_m^{\textsc{sn}}$ is the distribution of a random variable $\mathcal U_{md^2}$ depending on $md^2$ and is independent of all the parameters of the model. It has an explicit expression by means of Brownian bridges but its law is not explicitly known.
Nevertheless it can be easily tabulated by Monte-Carlo experiments
(see Table 1 in \cite{lobato}). We emphasize the fact that in
\cite{frz,yac} the authors have proposed some modified versions of
the Box-Pierce and Ljung-Box statistics that are more difficult to
implement because their critical values have to be computed from the
data.  In our case, the critical values are not
computed from the data since they are tabulated. In some sense,  our
method is finally closer to the standard method in which the critical values are simply
deduced from a $\chi^2$-table.

In Monte Carlo experiments, we illustrate that the proposed test statistics  have reasonable finite sample performance. Under nonindependent errors, it appears that the standard test statistics are generally non reliable, overrejecting or underrejecting severely, while the proposed tests statistics offer satisfactory levels in most cases. Even for independent errors, they seem  preferable to the standard ones, when the number $m$ of autocorrelations is small. Moreover, the error of first kind is well controlled.
Concerning the relative powers of the proposed tests, we also show
that the proposed tests have similar powers than the standard ones when the critical values are adjusted and when the sample size is large.
For all these reasons, we think that the modified versions that we propose in this paper are preferable to the standard ones for diagnosing VARMA models under nonindependent errors.

The article is organised as follows. In the next section, we briefly
present the models that we consider here and summarize the results
on the quasi-maximum likelihood estimator (QMLE) asymptotic distribution obtained by \cite{yac2}. Our methodology and the main results
are given in Section \ref{result}.
Simulation studies and an illustrative application on real data are presented in  Section~\ref{ne} and Section~\ref{ill}. The numerical tables are gathered in Section \ref{table}, after the bibliography. The proofs of the main results are available in the extended online version of this paper.

\section{Parametrization and assumptions}
\label{model}
The structural VARMA$(p,q)$ representation $(\ref{VARMA})$ can be rewritten in a reduced VARMA$(p,q)$ form
if the matrices $A_{00}$ and $B_{00}$ are non singular.
Indeed, premultiplying $(\ref{VARMA})$ by $A_{00}^{-1}$ and introducing
the innovation process $ {e}=({e}_t)_{t\in\mathbb Z}$ with ${e}_t=A_{00}^{-1}B_{00}{\epsilon}_t$, with non singular variance $\Sigma_{e0}=A_{00}^{-1}B_{00}\Sigma_0 B_{00}'A_{00}^{-1 '}$, we obtain the reduced VARMA representation
\begin{equation}
\label{VARMASTANDARD}
{X}_t-\sum_{i=1}^pA_{00}^{-1}A_{0i}{X}_{t-i}={e}_t-\sum_{i=1}^qA_{00}^{-1}B_{0i}B_{00}^{-1}A_{00}{e}_{t-i}.
\end{equation}

The structural form (\ref{VARMA}) allows to handle seasonal models, instantaneous economic relationships, VARMA in the so-called echelon form representation, and many other constrained VARMA representations (see \cite{L2005}, chap. 12). The reduced form (\ref{VARMASTANDARD}) is more practical from a statistical viewpoint as it gives the forecasts of each component of ${X}$ according to the past values of the set of the components.

Let $\left[A_{00}\dots A_{0p} B_{00}\dots B_{0q}\right]$ be the $d\times (p+q+2)d$ matrix of vector antoregressive and moving average
coefficients involved in 
$(\ref{VARMA})$. The matrix $\Sigma_0$ is considered as a nuisance parameter.
The parameter of interest is denoted $\vartheta_0$, where $\vartheta_0$
belongs to the interior of the parameter space ${\Theta}\subset
\mathbb{R}^{k_0}$, and $k_0$ is the number of unknown parameters,
which is typically much smaller than $(p+q+2)d^2$. Without loss of generality, we assume that ${\Theta}$ is compact. The matrices
$A_{00},\dots A_{0p}, B_{00},\dots B_{0q}$ involved in
(\ref{VARMA}) and $\Sigma_0$ are specified by $\vartheta_0$. More
precisely, we write $A_{0i}=A_{i}(\vartheta_0)$ and $B_{0j}=B_{j}(\vartheta_0)$ for $i=0,\dots,p$ and $j=0,\dots,q$, and
$\Sigma_0=\Sigma(\vartheta_0)$. The VARMA model (\ref{VARMA}) can be written more compactly as $A_{\vartheta_0}(L){X}_t=B_{\vartheta_0}(L){e}_t$ where $A_{\vartheta_0}(L)=A_{00}-\sum_{i=1}^pA_{0i}L^i$ and $B_{\vartheta_0}(L)=B_{00}-\sum_{i=1}^qB_{0i}L^i$, where $L$ stands for the backshift operator. Thus for all $\theta\in\Theta$, we have
\begin{align}\label{e(vartheta)}
e_t(\vartheta) = B_{\vartheta}^{-1}A_{\vartheta} X_t = \sum_{i=0}^\infty c_i(\vartheta)X_{t-i}
\end{align}
with $c_0(\vartheta)=1$ and $e_t(\vartheta_0)=e_t$.

For the estimation of ARMA and VARMA models, the commonly used
estimation method is the quasi-maximum likelihood estimation (QMLE for short), which can be also viewed as a nonlinear least squares estimation (LSE). The asymptotic properties of the QMLE of VARMA models are well-known under the restrictive assumption that the noise $\epsilon$ is an iid sequence (see \cite{DufourJouini2014,L2005}). See also \cite{MelardRoySaidi2006} where a description of estimation of structured VARMA models allowing unit roots is proposed.

Given a realization ${X}_1,{X}_2,\dots,{X}_n$ satisfying (\ref{VARMA}), the variable ${e}_t(\vartheta)$ can be approximated, for $0<t\leq n,$ by $\tilde{e}_t(\vartheta)$  defined recursively by
\begin{align}\label{etilde}
\tilde{e}_t(\vartheta)& ={X}_t-\sum_{i=1}^pA_0^{-1}A_i{X}_{t-i}+
\sum_{i=1}^qA_0^{-1}B_iB_0^{-1}A_0\tilde{e}_{t-i}(\vartheta),
\end{align}
where the unknown initial values are set to zero:
$\tilde{e}_{0}(\vartheta)=\dots=\tilde{e}_{1-q}(\vartheta)={X}_0=\dots={X}_{1-p}=0$.
The Gaussian quasi-likelihood is given by
\begin{eqnarray*}
\tilde{\mathrm{L}}_n(\vartheta,\Sigma_e)=\prod_{t=1}^n\frac{1}{(2\pi)^{d/2}\sqrt{\det
\Sigma_e}}\exp\left\{-\frac{1}{2}\tilde{e}_t'(\vartheta)
\Sigma_e^{-1}\tilde{e}_t(\vartheta)\right\},\quad
\Sigma_e=A_0^{-1}B_0\Sigma B_0'A_0^{-1 '}.
\end{eqnarray*}
A QMLE of $(\vartheta,\Sigma_e)$ is a measurable solution
$(\hat{\vartheta}_n,\hat\Sigma_e)$ of
\begin{eqnarray*}(\hat{\vartheta}_n,\hat\Sigma_e)
=\arg\min_{\vartheta,\Sigma_e}\left\{
\log(\det\Sigma_e)+\frac{1}{n}\sum_{t=1}^n\tilde{e}_t(\vartheta)
\Sigma_e^{-1}\tilde{e}'_t(\vartheta)\right\}.
\end{eqnarray*}
In the univariate weak case, Francq and Zako\"{i}an  \cite{fz98} showed the asymptotic normality of the least squares estimator under mixing assumptions. The asymptotic behavior of the QMLE,  of structural VARMA models,  has been studied in a much wider context by \cite{yac2} who proved  consistency and asymptotic normality under weak noise process.
To ensure the consistency and the asymptotic normality of the QMLE, we assume that the parametrization satisfies the following smoothness
conditions as in \cite{yac2}.
\begin{itemize}
\item[\hspace*{1em} {\bf (A1):}]
\hspace*{1em} The functions $\vartheta\mapsto A_{i}(\vartheta)$
$i=0,\dots,p$, $\vartheta\mapsto B_{j}(\vartheta)$ $j=0,\dots,q$ and
$\vartheta\mapsto \Sigma(\vartheta)$ admit continuous third order
derivatives for all $\vartheta\in {\Theta}$.
\end{itemize}
For simplicity, we now write $A_i$, $B_j$ and $\Sigma$ instead of
$A_i(\vartheta)$, $B_j(\vartheta)$ and $\Sigma(\vartheta)$. Let
$A_{\vartheta}(z)=A_0-\sum_{i=1}^pA_iz^i$ and
$B_{\vartheta}(z)=B_0-\sum_{i=1}^qB_iz^i$. We denote by $\det(A)$ the determinant of a matrix $A$ and by $\mbox{vec}(A)$ the vector obtained by stacking the columns of $A$. We assume that ${\Theta}$
corresponds to stable and invertible representations, namely
\begin{itemize}
\item[\hspace*{1em} {\bf (A2):}]
\hspace*{1em} for all $\vartheta\in{\Theta}$, we have $\det
A_{\vartheta}(z)\det B_{\vartheta}(z)\neq 0$ for all $|z|\leq 1$.
\end{itemize}
In the structural VARMA model (\ref{VARMA}), the assumption {\bf (A2)} does not guarantee the identifiability of the parameter. Thus, we make the following  global assumption for all $\vartheta\in{\Theta}$.
\begin{itemize}
\item[\hspace*{1em} {\bf (A3):}]
\hspace*{1em} For all $\vartheta\in{\Theta}$ such that $\vartheta\neq
\vartheta_0$, either the transfer functions
$A_0^{-1}B_0B_{\vartheta}^{-1}(z)A_{\vartheta}(z)\neq A_{00}^{-1}B_{00}B_{\vartheta_0}^{-1}(z)A_{\vartheta_0}(z)$
for some $z\in \mathbb{C}$, or
$A_0^{-1}B_0\Sigma
B_0'A_0^{-1 '}\neq A_{00}^{-1}B_{00}\Sigma_0 B_{00}'A_{00}^{-1 '}.$
\end{itemize}
\begin{itemize}
\item[\hspace*{1em} {\bf (A4):}]
\hspace*{1em}  {The process $e=(e_t)_{t\in\mathbb Z}$ is  ergodic and strictly  stationary}. Moreover, $e_t$ has a positive density on some neighborhood of zero.
\end{itemize}
In the reduced VARMA representation (\ref{VARMASTANDARD}), the condition $A_0^{-1}B_0\Sigma
B_0'A_0^{-1 '}\neq A_{00}^{-1}B_{00}\Sigma_0 B_{00}'A_{00}^{-1 '}$ in {\bf (A3)} can be dropped but may be important for structural VARMA forms. In particular, {\bf (A3)}  is satisfied when we impose that
$A_0=B_0=I_d$, that {\bf (A2)} holds, that the common left divisors of $A_{\vartheta}(L)$ and $B_{\vartheta}(L)$ are unimodular ({\em i.e.} with nonzero constant determinant), and finally that the matrix $[A_{p}:B_{q}]$ is of full rank. In contrast, the echelon form guarantees uniqueness of the VARMA representation and is the most widely identified VARMA representation employed in the literature. The identifiability of VARMA processes has been studied in \cite{H1976} where the author has given several procedures ensuring identifiability.

For the asymptotic normality of the QMLE, additional assumptions are required. It is necessary to assume
that $\vartheta_0$ is not on the boundary of the parameter space
${\Theta}$.
\begin{itemize}
\item[\hspace*{1em} {\bf (A5):}]
\hspace*{1em} We have
$\vartheta_0\in\stackrel{\circ}{{\Theta}}$, where $\stackrel{\circ}{{\Theta}}$ denotes the interior of ${\Theta}$.
\end{itemize}
At last, we use the matrix $M_{\vartheta_0}$ of the coefficients of the
reduced form (\ref{VARMASTANDARD}), where
$$M_{\vartheta_0}=[A_{00}^{-1}A_{01}:\dots:A_{00}^{-1}A_{0p}:A_{00}^{-1}B_{01}
B_{00}^{-1}A_{00}:\dots:A_{00}^{-1}B_{0q}B_{00}^{-1}A_{00}].$$
We need a local identifiability assumption which completes {\bf (A3)} and specifies how the matrix $M_{\vartheta_0}$ depends on the parameter $\vartheta_0$.
\begin{itemize}
\item[\hspace*{1em} {\bf (A6):}] \hspace*{1em} The matrix $\mathbf{M}_{\vartheta_0}$, defined as ${\partial \mbox{vec}(M_{\vartheta})}/{\partial \vartheta'}$
evaluated at $\vartheta_0$, is of full rank $k_0$.
\end{itemize}
We recall that the strong mixing coefficients $\alpha_{e}(h)$ of the stationary process $e$ are defined by
$$\alpha_{e}\left(h\right)=\sup_{A\in\mathcal F^t_{-\infty},B\in\mathcal F_{t+h}^{+\infty}}\left|\mathbb{P}\left(A\cap
B\right)-\mathbb{P}(A)\mathbb{P}(B)\right| ,$$
where $\mathcal F_{-\infty}^t=\sigma (e_u, u\leq t )$ and $\mathcal F_{t+h}^{+\infty}=\sigma (e_u, u\geq t+h )$.
We denote $\|\cdot\|$ the Euclidian norm. We will make an integrability assumption on the moment of the noise and a summability condition on the strong mixing coefficients $(\alpha_{e}(k))_{k\ge 0}$.
\begin{equation}
\label{mixing}
\hspace{-3cm}{\bf (A7):}\hspace*{1.5em}
\text{We have }\mathbb{E}\big\|{e}_t\|^{4+2\nu}<\infty\text{ and }
\sum_{k=0}^{\infty}\left\{\alpha_{{e}}(k)\right\}^{\frac{\nu}{2+\nu}}<\infty\text{ for some } \nu>0.
\end{equation}
Under  the above assumptions,  Boubacar Ma\"{i}nassara and Francq \cite{yac2} showed
that $\hat{\vartheta}_n\to\vartheta_0\;a.s.$ as $n\to\infty$ 
and $\displaystyle \sqrt{n}(\hat{\vartheta}_n-\vartheta_0) \xrightarrow[n\to\infty]{\mathrm{d}} {\cal
N}(0,\mathrm{V}:=J^{-1}IJ^{-1})$ where
\begin{align*}
J& :=J(\vartheta_0,\Sigma_{e0})=2\mathbb{E}
\left\{\frac{\partial{e}_t'(\vartheta_0)}{\partial
\vartheta}\right\}\Sigma_{e0}^{-1}\left\{\frac{\partial
{e}_t(\vartheta_0)}{\partial \vartheta'}\right\}\text{ and } \\
I& :=I(\vartheta_0,\Sigma_{e0})=\sum_{h=-\infty}^{+\infty}\cov({\Upsilon}_t,{\Upsilon}_{t-h})\text{ with} \\
\Upsilon_t& :=\Upsilon_t(\vartheta_0,\Sigma_{e0})=\frac{\partial}{\partial \vartheta}\left\{\log \det \Sigma_{e0}+{e}_t'(\vartheta_0)\Sigma_{e0}^{-1}{e}_t(\vartheta_0)\right\}
=2 \frac{\partial {e}'_{t}(\vartheta_0)}{\partial\vartheta}\Sigma_{e0}^{-1}{e}_{t}(\vartheta_0).
\end{align*}

\section{Main results} \label{result}

In order to state our results,  we need  further notations that will be used both in univariate and multivariate contexts. In all the sequel, we denote  by $\xrightarrow[]{\mathrm{d}}$, respectively by $\xrightarrow[]{\mathbb{P}}$ the convergence in distribution, respectively in probability.  The symbol $\mathrm{o}_{\mathbb P}(1)$ is used for a sequence of random variables that converges to zero in probability.

 Let $(B_K(r))_{r\ge 0}$ be a $K$-dimensional Brownian motion starting from $0$.  For $K\ge 1$, we denote  $\mathcal U_K$ the random variable defined by
\begin{align}
\label{um}
\mathcal U_K&={B}'_{K}(1){V}_{K}^{-1}{B}_{K}(1)
\end{align}
where
\begin{align}
\label{vm}
{V}_{K}=\int_0^1\left ( {B}_{K}(r)-r{B}_{K}(1)\right)\left({B}_{K}(r)-r{B}_{K}(1)\right)'dr.
\end{align}
The critical values for $\mathcal{U}_K$ have been tabulated by Lobato in \cite{lobato}.
The random variable $\mathcal U_K$ will intervene in the asymptotic behaviour of the laws of the statistics proposed in \eqref{bms}.
\subsection{Multivariate ARMA models}\label{notat}

When $p>0$ or $q>0$, we may define $\hat {{e}}_t= \tilde{e}_t(\hat\vartheta_n)$  for $ t\geq1$, the quasi-maximum likelihood residuals as
$$\hat {{e}}_t={X}_t-\sum_{i=1}^pA_{0}^{-1}(\hat\vartheta_n)A_{i}(\hat\vartheta_n)\hat{{X}}_{t-i}
+\sum_{i=1}^qA_{0}^{-1}(\hat\vartheta_n)B_{i}(\hat\vartheta_n)B_{0}^{-1}(\hat\vartheta_n)A_{0}(\hat\vartheta_n)\hat
{{e}}_{t-i},$$ for $t=1,\dots,n$, with $\hat{ {X}}_t=0$ for $ t\leq0$ and
$\hat {{X}}_t={X}_t$ for $ t\geq1$. We denote by
\begin{eqnarray*}
\Gamma_e(h)=\frac{1}{n}\sum_{t=h+1}^{n}{e}_t{e}'_{t-h}\quad \mbox{and} \quad
\hat\Gamma_e(h)=\frac{1}{n}\sum_{t=h+1}^{n}\hat {{e}}_t\hat {{e}}'_{t-h}
\end{eqnarray*}
the multivariate  white noise "empirical" autocovariances and residual
autocovariances. For a fixed integer $m\geq1$, let
$$\Gamma_m =\left(\left\{\mbox{vec}
\Gamma_e(1)\right\}',\dots,\left\{\mbox{vec} \Gamma_e(m)\right\}'
\right)',\;{\hat\Gamma}_m
=\left(\left\{\mbox{vec}\hat\Gamma_e
(1)\right\}',\dots,\left\{\mbox{vec}\hat\Gamma_e(m)\right\}'
\right)'.$$
Denoting the diagonal matrices $S_e$ and $\hat{S_e}$  by
$$S_e=\mbox{Diag}\left(\sigma_e(1),\dots,
\sigma_e(d)\right)\quad \mbox{and} \quad
\hat{S_e}=\mbox{Diag}\left(\hat\sigma_e(1),\dots,
\hat\sigma_e(d)\right),$$ where $\sigma_e^2(i)$ is the variance of
the $i$-th coordinate of ${e}_t$ and $\hat\sigma_e^2(i)$ is its sample
estimate ({\em i.e.} $\sigma^2_e(i)=\mathbb{E} e_{it}^2$ and
$\hat\sigma^2_e(i)=n^{-1}\sum_{t=1}^n\hat e_{it}^2$). The
theoretical and sample autocorrelations at lag $\ell$ are
respectively defined by $R_e(\ell)=S_e^{-1}\Gamma_e(\ell)S_e^{-1}$
and $\hat R_e(\ell)=\hat S_e^{-1}\hat\Gamma_e(\ell)\hat S_e^{-1},$
with ${\Gamma}_e(\ell):=\mathbb{E} {e}_t {e}'_{t-\ell}=0$ for all $\ell\neq0.$ In the sequel, we will also need the vector of the first $m$  sample autocorrelations
$$\hat{{\rho}}_m =\left(\left\{\mbox{vec}\hat R_e
(1)\right\}',\dots,\left\{\mbox{vec}\hat R_e(m)\right\}' \right)'.$$
In \cite{yac}, it is proved that
\begin{eqnarray}\label{GammaO_P}
\vec(\hat S_e^{-1}\hat\Gamma_e(h)\hat S_e^{-1})- \vec(
S_e^{-1}\hat\Gamma_e(h) S_e^{-1})=\mathrm{o}_\mathbb{P}(1).
\end{eqnarray}
Now using the elementary identities  $\vec(ABC)=(C'\otimes A)\vec(B)$
and $(A\otimes B)^{-1}=A^{-1}\otimes B^{-1}$ when A and B are
invertible, it follows from (\ref{GammaO_P}) that
\begin{eqnarray*}\hat{{\rho}}_m
&=&\left(\left\{\mbox{vec}\hat R_e
(1)\right\}',\dots,\left\{\mbox{vec}\hat R_e(m)\right\}'
\right)'
\\&=&\left\{I_m\otimes(\hat S_e\otimes\hat S_e)^{-1}\right\}
\hat\Gamma_m =\left\{I_m\otimes(S_e\otimes S_e)^{-1}\right\}
\hat\Gamma_m+\mathrm{o}_\mathbb{P}(1).
\end{eqnarray*}
Based on the residual empirical autocorrelations $\hat R_e(h)$, the
statistics \eqref{bpmulti} are usually used to test the following null  hypothesis
\begin{itemize}
\item[($\mathbf{H0}$) : ]  $({X}_t)_{t\in \mathbb Z}$ satisfies a VARMA$(p,q)$ representation;
\end{itemize}
against the alternative
\begin{itemize}
\item[($\mathbf{H1}$) :] $({X}_t)_{t\in \mathbb Z}$ does not admit a VARMA$(p,q)$ representation or $({X}_t)_{t\in \mathbb Z}$ satisfies a VARMA$(p',q')$ representation with $p'>p$ or $q'>q$.
\end{itemize}
As mentioned before, our strategy relies on new statistics and we need further notations to explain our purpose. We denote
$\Lambda$ the matrix in $\mathbb R^{md^2\times(k_0+md^2)}$ defined in block
formed by
\begin{equation*}
{\Lambda}=({\Phi}_m{J}^{-1} {\vert} I_{md^2})\text{ and }
{\Phi}_m=\mathbb{E}\left\{\left(\begin{array}{c}
{e}_{t-1}\\\vdots\\{e}_{t-m}\end{array}\right)\otimes \frac{\partial
{e}_{t}}{\partial\vartheta'}\right\}.
\end{equation*}
A Taylor expansion around $\vartheta_0$ gives
\begin{eqnarray*}
\sqrt{n}\, {\hat\Gamma}_m
=\frac{1}{\sqrt{n}}\sum_{t=1}^n{\Lambda}{U}_t+\mathrm{o}_\mathbb{P}(1),\text{ where }{U}_t =  \left (\left (-2\frac{\partial
{e}'_{t}}{\partial\vartheta}\Sigma_{e0}^{-1}{e}_{t}\right )' {,}
\left ({e}_{t-1}\otimes {e}_t \right )' {,}\dots{,}\left ( {e}_{t-m}\otimes {e}_t \right )' \right )'.
\end{eqnarray*}
One refers to the extended online version of this paper for further details about the above expressions. At this stage, we do not rely on the classical method that would consist in estimating the asymptotic covariance matrix of $\Lambda U_t$. We  rather try to apply Lemma 1 in \cite{lobato}. So we need to check that a functional central limit theorem holds for the process $U:=(U_t)_{t\ge 1}$.

Without the first entry in $U_t$, we would be in the context of Lobato and the functional central limit theorem would be clear thanks to the mixing condition on the noise process $\epsilon$. Unfortunately, it is more difficult to deal with the process $U$ itself. In order to prove that
$\frac{1}{\sqrt{n}}\sum_{j=1}^{\lfloor nr\rfloor }\Lambda U_j$ converges on the Skorokhod space to a Brownian motion,
we will employ a cutoff argument on the  representation of $\partial e_t(\vartheta_0)/\partial \vartheta$ as an infinite sum of the past of the noise. This difficulty is overcome in the proof of Theorem \ref{sn1} (see the extended online version of this paper for further details).

Finally, we define the normalization matrix $C_{md^2}\in\mathbb R^{md^2\times md^2}$  by
\begin{equation*}
C_{md^2} =\frac{1}{n^2}\sum_{t=1}^{n}{S}_{t}{S}'_{t} \text{ where }
{S}_t =\sum_{j=1}^{t}\left({\Lambda}{U}_j-\Gamma_m\right).
\end{equation*}
The following theorem states the asymptotic distributions of the sample autocovariances and autocorrelations.
\begin{theo}\label{sn1VARMA}
We assume that $p>0$ or $q>0$. Under Assumptions {\bf (A0)}--{\bf (A7)} and  under the null hypothesis {\bf (H0)}, we have
\begin{align*}
n\, \hat\Gamma_m'C_{md^2}^{-1}\hat\Gamma_m & \xrightarrow[n\to\infty]{\mathrm{d}} \mathcal{U}_{md^2}.
\end{align*}
The sample autocorrelations satisfy
\begin{align*}
n\,\hat{{\rho}}_m'\left\{I_m\otimes(S_e\otimes S_e)\right\}
C_{md^2}^{-1}\left\{I_m\otimes(S_e\otimes S_e)\right\}\hat{{\rho}}_m
& \xrightarrow[n\to\infty]{\mathrm{d}} \mathcal{U}_{md^2}
\end{align*}
\end{theo}
The proof of this result is available in the extended online version of this paper.

Of course, the above theorem is useless for practical purpose,  because it does not involve any observable quantities. This gap will be fixed below (see Theorem \ref{sn2VARMA}). Right now, we make several important comments that are necessary to compare Theorem \ref{sn1VARMA} with the existing results.
\begin{remi}When $p=q=0$, we no longer need to estimate the unknown parameter $\vartheta_0$. Thus a careful reading of the proofs shows that the vector $U_t$ is replaced by
$$ \tilde U_t =\left (\left ({e}_{t-1}\otimes {e}_t \right )' {,}\dots{,}\left ( {e}_{t-m}\otimes {e}_t \right )' \right )'$$ and $\Lambda$ is replaced by the identity matrix. Then we obtain the result of Lobato (see Lemma 1 in \cite{lobato}) and we thus generalized his result to the VARMA model.
\end{remi}
In practice, one has to replace the matrix $C_{md^2}$ and the variance of
the noise $\Sigma_{e0}$  by their empirical or observable counterparts.
For $C_{md^2}$, the idea is to use $\hat{e}_t=\tilde{e}_t(\hat\vartheta_n)$ instead of the unobservable noise $e_t$. The matrices
$J$ and  $\Phi_m$ can be easily estimated by their
empirical counterparts
\begin{equation*}
\hat{{J}}=\frac{2}{n} \sum_{t=1}^n \left\{\frac{\partial
\tilde{e}_t'(\hat\vartheta_n)}{\partial
\vartheta}\hat{\Sigma}_{e0}^{-1}\frac{\partial
\tilde{e}_t(\hat\vartheta_n)}{\partial \vartheta'}\right\}
\text{ and }
\hat{{\Phi}}_m =\frac{1}{n}
\sum_{t=1}^n\left\{\left(\begin{array}{c}
\hat{{e}}_{t-1}\\\vdots\\
\hat{{e}}_{t-m}\end{array}\right)\otimes \frac{\partial
\tilde{e}_{t}(\hat\vartheta_n)}{\partial\vartheta'}\right\}, 
\end{equation*}
where
$\hat{\Sigma}_{e0}=\hat\Gamma_e(0)={n}^{-1}\sum_{t=1}^{n}\hat {{e}}_t\hat {{e}}'_{t}$. Thus we define
\begin{equation*}
 \hat\Lambda =  \Big ( \hat \Phi_m \hat J_n^{-1} {\vert} I_{md^2}  \Big ) \text{ and }
  \hat U_t = \left (\left (-2\frac{\partial
\tilde{e}'_{t}(\hat\vartheta_n)}{\partial\vartheta}\hat{\Sigma}_{e0}^{-1}\tilde{e}_{t}(\hat\vartheta_n) \right )' {,}
\left (  \hat{{e}}_{t-1}\otimes \hat{{e}}_t \right )' {,}\dots{,}
\left ( \hat{{e}}_{t-m}\otimes \hat{{e}}_t \right )' \right )'.
\end{equation*}
Finally we denote the normalization matrix $\hat C_{md^2}\in\mathbb R^{md^2\times md^2}$ as
 \begin{equation}
\label{hatcmVARMA}
\hat C_{md^2} =\frac{1}{n^2}\sum_{t=1}^{n}\hat{{S}}_{t}\hat{{S}}'_{t}
  \text{ where }\hat {{S}}_t =\sum_{j=1}^{t}\left(\hat{{\Lambda}}\hat{{U}}_j-\hat\Gamma_m\right).
\end{equation}
The above quantities are all observable and we are able to state our second
theorem which is the applicable counterpart of Theorem \ref{sn1VARMA}.
\begin{theo}\label{sn2VARMA}
Under the assumptions of Theorem \ref{sn1VARMA}, we have
\begin{align*}
n\, \hat\Gamma_m'\hat C_{md^2}^{-1}\hat\Gamma_m & \xrightarrow[n\to\infty]{\mathrm{d}} \mathcal{U}_{md^2}.
\end{align*}
The sample autocorrelations satisfy
\begin{align*}
 {Q}_m^\textsc{sn}:=n\,\hat{{\rho}}_m'\left\{I_m\otimes(\hat S_e\otimes \hat S_e)\right\}
\hat C_{md^2}^{-1}\left\{I_m\otimes(\hat S_e\otimes \hat S_e)\right\}\hat{{\rho}}_m
& \xrightarrow[n\to\infty]{\mathrm{d}} \mathcal{U}_{md^2}
\end{align*}
\end{theo}
The proof of this result is available in the extended online version.

Based on the above result, we propose a version of Hosking \cite{H1980} statistic when one uses the following one
\begin{equation}\label{bmsbis}
\tilde{ {Q}}_m^\textsc{sn}= n^2\, \hat{{\rho}}_m'\left\{I_m\otimes(\hat S_e\otimes \hat S_e)\right\}
{D}_{n,md^2}\hat C_{md^2}^{-1}\left\{I_m\otimes(\hat S_e\otimes \hat S_e)\right\}\hat{{\rho}}_m,
\end{equation}
where the matrix ${D}_{n,md^2}\in\mathbb R^{md^2\times md^2}$ is diagonal with $(n/(n-1),...,n/(n-m))$ as diagonal terms.
\subsection{Examples}
In order to make our presentation more readable, the results of the previous section are presented in the one dimensional case (\emph{i.e} when $d=1$) which requires less technical notation.
The first  example is on an univariate weak ARMA$(p,q)$ model. In the second example, we give an explicit expression of the normalized matrix $C_{m}$ for one dimensional  AR$(1)$ and MA$(1)$ models.

\subsubsection{Univariate ARMA models}
When $d=1$, in the representation \eqref{VARMASTANDARD} we have $\mathbb{E}(\epsilon_t^2)=\sigma_{e0}^2$ and $a_{00}=b_{00}=1$. So  \eqref{VARMASTANDARD} takes the following form
\begin{equation}
  \label{FZARMA}
X_t-\sum_{i=1}^{p}a_{0i}X_{t-i}=\epsilon_t+\sum_{j=1}^{q}a_{0j}\epsilon_{t-j}.
\end{equation}
Then with our notations, the unknown parameter
$\vartheta_0=(a_{01},\dots{},a_{0p},b_{01},\dots{},b_{0q})'$ is supposed to
belong to the interior  of the
parameter space
\begin{align*}
\Theta & :=  \Big
\{\theta=(\vartheta_1,\dots{},\vartheta_p,\vartheta_{p+1},\dots,\vartheta_{p+q})'
\in{\Bbb R}^{p+q}, \, A_\vartheta(z)=1-\hbox{$\sum_{i=1}^{p}$}\vartheta_iz^i\text{
and } B_\vartheta(z)=1+\hbox{$\sum_{i=p+1}^{p+q}$}\vartheta_iz^i \\
& \hspace{1cm} \text{ have all
their zeros outside the unit disk and have also no common zero}\Big \}\ .
\end{align*}
The theoretical, respectively the sample autocorrelations, at lag $\ell$ take
the simpler forms $R_e(\ell)=\Gamma_e(\ell)/\Gamma_e(0)$, respectively
$\hat R_e(\ell)=\hat \Gamma_e(\ell)/\hat \Gamma_e(0),$
with $\Gamma_e(0):=\sigma_{e0}^2$. The vector of the first $m$  sample autocorrelations is now written as
$\hat\rho_m =(\hat R_e
(1),\dots,\hat R_e(m))'$. The statistics defined in \eqref{bp} are based on the residual empirical autocorrelations $\hat R_e(h)$, and are used to test the null  hypothesis stated above. As mentionned before, we obtain  the block matrix
$\Lambda\in\mathbb R^{m\times(p+q+m)}$ formed by
\begin{align}\label{Gamma}
\Lambda =  \Big ( \Phi_m J^{-1} {\vert} I_m  \Big ), \text{ where } \Phi_m =\e\left\{\left(\begin{array}{c}
e_{t-1}\\\vdots\\e_{t-m}\end{array}\right)\frac{\partial
e_{t}(\vartheta_0)}{\partial\vartheta'}\right\}
\text{ and }
J=  \frac{2}{\sigma_{e0}^2}\mathbb E \left [ \frac{\partial e_t(\vartheta_0)}{\partial\vartheta}\frac{\partial e_t(\vartheta_0)}{\partial\vartheta'}  \right ].
\end{align}
Then one may write
\begin{equation*}
\sqrt{n}\ \hat\Gamma_m  = \frac{1}{\sqrt{n}}\sum_{t=1}^n\ \Lambda U_t+
\mathrm{o}_{\mathbb P}(1) ,\text{ with }
 U_t =
\left (-2\frac{\partial e_t(\vartheta_0)
}{\partial
\vartheta}\frac{1}{\sigma_{e0}^2}e_t,e_te_{t-1},\dots,e_te_{t-m}\right
)'.
\end{equation*}
Finally, the normalization matrix  $C_m\in\mathbb R^{m\times m}$ is now defined by
\begin{equation*}
C_{m} =\frac{1}{n^2}\sum_{t=1}^{n}S_{t}S'_{t} \text{ where } S_t
=\sum_{j=1}^{t}\left(\Lambda U_j-\Gamma_m\right).
\end{equation*}
We are now able to state the following theorem, which is clearly the univariate version of Theorem~\ref{sn1VARMA}.
\begin{theo}\label{sn1}
Assume that $p>0$ or $q>0$.
Under the assumptions of Theorem~\ref{sn1VARMA}, we have
\begin{align}\label{conv1}
n\, \hat\Gamma_m'C_m^{-1}\hat\Gamma_m & \xrightarrow[n\to\infty]{\mathrm{d}} \mathcal{U}_m.
\end{align}
The sample autocorrelations satisfy
\begin{align}\label{conv2}
n\,\sigma_{e0}^4 \,\hat\rho_m'C_m^{-1}\hat\rho_m & \xrightarrow[n\to\infty]{\mathrm{d}} \mathcal{U}_m.
\end{align}
\end{theo}
As mentioned in the VARMA case, Theorem \ref{sn1} has also to be completed.
The matrices $J$, $\Phi_m$ and the scalar $\sigma_{e0}^2$  can be easily estimated by their empirical counterparts:
\begin{equation*}
\hat{J}_n=\frac{1}{\hat\sigma_{e0}^2}\frac{2}{n}\sum_{t=1}^n\frac{\partial \tilde{e}_t(\hat{\vartheta}_n)
}{\partial \vartheta}\frac{\partial \tilde{e}_t(\hat{\vartheta}_n)
}{\partial \vartheta'}{,}\ \hat\Phi_m  =\frac{1}{n}
\sum_{t=1}^n\left\{\left(\hat e_{t-1},\dots,\hat
e_{t-m}\right)'\frac{\partial
 \tilde{e}_{t}(\hat\vartheta_n)}{\partial\vartheta'}\right\}\text{ and }\hat\sigma_{e0}^2=\frac{1}{n}\sum_{t=1}^n\tilde{e}^2_t(\hat{\vartheta}_n).
\end{equation*}
Thus we define
\begin{equation*}
 \hat\Lambda =  \Big ( \hat \Phi_m \hat J_n^{-1} {\vert} I_m  \Big ) \text{ and }
  \hat U_t = \left (-2\frac{\partial e_t(\hat\vartheta_n) }{\partial \vartheta}\frac{1}{\hat\sigma_{e0}^2}\hat e_t,\hat e_t\hat e_{t-1},\dots,\hat e_t\hat e_{t-m}\right )'.
\end{equation*}
Finally we denote the normalization matrix $\hat C_m\in\mathbb R^{m\times m}$ by
\begin{equation*}
\hat C_{m} =\frac{1}{n^2}\sum_{t=1}^{n}\hat S_{t}\hat S'_{t}  \text{  where  }
\hat S_t =\sum_{j=1}^{t}\left(\hat\Lambda\hat U_j-\hat \Gamma_m\right).
\end{equation*}
The above quantities are all observable and the following result is the applicable counterpart of Theorem \ref{sn1}.
\begin{theo}\label{sn2}
Assume that $p>0$ or $q>0$. Under Assumptions of Theorem \ref{sn1}, we have
\begin{align*}
n\, \hat\Gamma_m'\hat C_m^{-1}\hat\Gamma_m & \xrightarrow[n\to\infty]{\mathrm{d}} \mathcal{U}_m.
\end{align*}
The sample autocorrelations satisfy
\begin{align*}
Q_m^\textsc{sn}=n\,\hat\sigma_{e0}^4 \,\hat\rho_m'\hat C_m^{-1}\hat\rho_m & \xrightarrow[n\to\infty]{\mathrm{d}} \mathcal{U}_m.
\end{align*}
\end{theo}
Based on the above result, we propose a modified version of the Ljung-Box statistic when one uses the statistic
$$\tilde Q_m^\textsc{sn}= n\, \hat\sigma^4_{e0}  \,\hat\rho_m'  D^{1/2}_{n,m} \hat C_m^{-1}D^{1/2}_{n,m}\hat\rho_m $$
where the matrix $D_{n,m}\in\mathbb R^{m\times m}$ is diagonal with $((n+2)/(n-1),...,(n+2)/(n-m))$ as diagonal terms.
%
\subsubsection{Explicit form of the matrix $C_m$ in AR$(1)$ or MA$(1)$ case}
The following example gives an explicit form of the matrix $C_m$ in the univariate AR$(1)$ or MA$(1)$ model.

For instance, consider the AR$(1)$ case, with $a_{00}=b_{00}=1$, $a_0=a_{01}$ and $\sigma^2=\sigma_{e0}^2$.
It is classical  that the univariate noise derivatives can be
represented as ${\partial \epsilon_t}/{\partial \vartheta}=\sum_{i=1}^\infty \lambda_i\epsilon_{t-i}$ where  $\lambda_i=-a_0^{i-1}$. Then we deduce that
$J=2/(1-a_0^{2})$ and the expression
$\Phi_m=-\sigma^2  ( 1 \  a_0 \  \cdots \  a_0^{m-1} )'$.
Since $U_t= (
            \begin{array}{cccc}
              -2\sigma^{-2}\sum_{i=1}^\infty \lambda_i\epsilon_{t}\epsilon_{t-i} & \epsilon_{t}\epsilon_{t-1} & \cdots & \epsilon_{t}\epsilon_{t-m} \\
            \end{array}
          )'$ and $ \Lambda=  ( \Phi_m J^{-1} {\vert} I_m  )$,  we obtain
\begin{align*}
\Lambda U_t & =
(1-a_0^{2})\sum_{i=1}^\infty \lambda_i\epsilon_{t}\epsilon_{t-i}( 1\ a_0\ \cdots \ a_0^{m-1})'+ (\epsilon_{t}\epsilon_{t-1}\ \cdots \ \epsilon_{t}\epsilon_{t-m})' \qquad\text{and} \\%
 S_t& =\sum_{j=1}^t
\left[(1-a_0^{2})\sum_{i=1}^\infty \lambda_i\epsilon_{j}\epsilon_{j-i}
( 1\ a_0\ \cdots \ a_0^{m-1})'+ (\epsilon_{j}\epsilon_{j-1}\ \cdots \ \epsilon_{j}\epsilon_{j-m})' - (\Gamma_e(1)\ \cdots \
\Gamma_e(m))' \right].
\end{align*}
For simplicity, we take $m=1$ and we obtain the following simple expression for the normalization matrix
$$C_1=\frac{1}{n^2}\sum_{t=1}^nS_t^2=\frac{1}{n^2}\sum_{t=1}^n\left\{\sum_{j=1}^t
\left(-(1-a_0^{2})\sum_{i=1}^\infty
a_0^{i-1}\epsilon_{j}\epsilon_{j-i}+\epsilon_{j}\epsilon_{j-1}-\Gamma_e(1)\right)\right\}^2.$$
When $p=0$ and $q=1$  the same result holds with $a_0$ replaced by $b_0=b_{01}$.

\section{Numerical illustrations}\label{ne}
In this section, by means of Monte Carlo experiments, we investigate the finite sample
properties of the modified version of the portmanteau tests that we introduced in this work.
The numerical illustrations of this section are made  with the open source
statistical software R (see R Development Core Team, 2015) or (see http://cran.r-project.org/). The tables are gathered in Section \ref{table}.

\subsection{Simulated models}
First of all, we introduce the models that we simulate.
For illustrative purpose, we also consider the standard portmanteau test and the modified
portmanteau test proposed by \cite{frz} (resp. by \cite{yac})
in the univariate ARMA case (resp. in VARMA case).

We indicate the conventions that we adopt in the discussion and in the tables:
\begin{itemize}
\item $\mathrm{{LB}_{\textsc{frz}}}$ and $\mathrm{BP_{\textsc{frz}}}$ refer to LB and BP tests using  ${Q}_m^\textsc{lb}$ and $Q_m^\textsc{bp}$ as in \cite{frz}
\item $\mathrm{{LB}_{\textsc{s}}}$ and $\mathrm{BP_{\textsc{s}}}$ refer to  LB and BP tests using the statistics \eqref{bp} and \eqref{bpmulti}.
\item $\mathrm{{{LB}}_{\textsc{sn}}}$ and  $\mathrm{{{BP}}_{\textsc{sn}}}$ refer to modified test using the statistics \eqref{bms} and \eqref{bmsbis}
\item $\mathrm{{{LB}}_{\textsc{bm}}}$ and $\mathrm{{{BP}}_{\textsc{bm}}}$ refer to  LB and BP tests with ${{Q}}_m^\textsc{h}$ and ${{Q}}_m^\textsc{c}$ as in \cite{yac}
\end{itemize}
Compared to the modified
portmanteau test proposed by \cite{frz} (resp. by \cite{yac}),
we use a vector autoregressive (VAR) spectral estimator approach to estimate the asymptotic
covariance matrix of a vector autocorrelations residuals.
The implementation of this method requires a choice of the VAR order $r$.
In the strong (V)ARMA cases we fixed $r=1$. By contrast, in the weak (V)ARMA cases
the VAR order $r$  is set as $r=1,\dots,5$ and is automatically
selected by Akaike Information Criterion (AIC) using the
function {\tt VARselect()} of the {\bf vars} R package.

The  $p$-values of the modified portmanteau tests, introduced by \cite{frz,yac}, are
computed using the Imhof  algorithm (see \cite{i1961}) and by using the function
{\tt imhof()}  of the R package {\bf CompQuadForm}.

We notice that in the tables, the numerical results using the Ljung-Box tests are very close to those of the Box-Pierce tests. Nevertheless, they are still presented here  for the sake of completeness. Other tests  could have been considered as well but our aim is not to make an exhaustive comparative study.

\subsubsection{Univariate ARMA case}
\label{univARMA}
To generate the strong and the weak ARMA models, we consider the
following ARMA$(1,1)$ model
\begin{eqnarray}
\label{ARMA01MonteCarlo}
X_{t}=a_0X_{t-1}+\epsilon_{t}+b_0\epsilon_{t-1},
\end{eqnarray}
with $\vartheta_0=(a_0,b_0)'=(0.95,-0.6)'$ and the innovation process  $\epsilon$ follows a strong or weak white noise.  Different weak ARMA$(1,1)$ models are simulated with various examples of weak white noises.


The  generalized autoregressive conditional heteroscedastic (GARCH) models is an important example of weak white noises in the univariate case (see \cite{FZ2010}).
So we first assume that in (\ref{ARMA01MonteCarlo}) the
innovation process $\epsilon$ is the following GARCH$(1,1)$ model
defined by
\begin{equation} \label{bruitARCH}
\left\{\begin{array}{l}\epsilon_{t}=\sigma_t\eta_{t}\\
\sigma_t^2=1+\alpha_1\epsilon_{t-1}^2+\beta_1\sigma_{t-1}^2
\end{array}\right.
\end{equation}
where  $(\eta_t)_{t\ge 1}$ is a sequence of iid standard Gaussian random variables. 

We propose three other sets of experiments with innovation processes $\epsilon$ in (\ref{ARMA01MonteCarlo}) defined by \begin{align}
\label{PT}
\epsilon_{t}& =\eta_{t}\eta_{t-1}, \\
\label{PTcarre}
\epsilon_{t}& =\eta_{t}^2\eta_{t-1}, \\
\label{RT}
\epsilon_{t}& =\eta_{t}(|\eta_{t-1}|+1)^{-1}.
\end{align}
The example \eqref{PT} was proposed in \cite{rt1996} and \eqref{PTcarre} and  \eqref{RT} are extensions of other types of noise process in \cite{rt1996}.
To generate the strong ARMA, we assume that in
(\ref{ARMA01MonteCarlo}) the innovation process follows (\ref{bruitARCH}) with $(\alpha_1{,}\beta_1)=(0{,}0)$.
Contrary  to \eqref{PT} and \eqref{RT}, the noise defined in \eqref{PTcarre} is not a martingale difference sequence for which the limit theory is more classical.
\subsubsection{Multivariate ARMA case}\label{multiARMA}
This section is a direct extension of the univariate section \ref{univARMA}.
Now we repeat the same experiment on different weak and strong  VARMA$(1,1)$ models.
In order to ensure the uniqueness of a VARMA representation, we consider the following bivariate VARMA$(1,1)$ model  considered in \cite{reinsel97} (chap. 3, example 3.2 p. 81)
\begin{eqnarray}
\label{VARMA11MonteCarlo}
\left(\begin{array}{c}X_{1,t}\\X_{2,t}\end{array}\right)
&=&\left(\begin{array}{cc}a_{11,1}&a_{12,1}\\a_{21,1}&a_{22,1}\end{array}\right)
\left(\begin{array}{c}X_{1,t-1}\\X_{2,t-1}\end{array}\right)
+\left(\begin{array}{c}\epsilon_{1,t}\\
\epsilon_{2,t}\end{array}\right)
- \left(\begin{array}{cc}b_{11,1}&b_{12,1}\\b_{21,1}&b_{22,1}\end{array}\right)
\left(\begin{array}{c}\epsilon_{1,t-1}\\
\epsilon_{2,t-1}\end{array}\right)\qquad
\end{eqnarray}
with $\vartheta_0=(a_{11,1},a_{21,1},a_{12,1},a_{22,1},b_{11,1},b_{21,1},
b_{12,1},b_{22,1})'=(1.2,0.6,-0.5,0.3,-0.6,0.3,0.3,0.6)'$ and $\epsilon_t=(\epsilon_{1,t},\epsilon_{2,t})'$ that follows a strong or weak white noise.
Note that $$\det\left\{I_2-z_1\left(\begin{array}{cc}1.2&-0.5\\0.6&0.3\end{array}\right)\right\}=0 \text{ and }\det\left\{I_2-z_2\left(\begin{array}{cc}-0.6&0.3\\0.3&0.6\end{array}\right)\right\}=0$$  for $z_1=1.136\pm 0.473i$ (and hence $|z_1|=1.23$) and  for $z_2=\pm (0.45)^{-1/2}$ (and hence $|z_2|=1.49$). Hence Model (\ref{VARMA11MonteCarlo}) can be viewed as a bivariate VARMA$(1,1)$ model  in echelon form ($\text{ARMA}_E(1,1)$) considered in \cite{L2005} (chap. 12, definition 12.2 p. 453).

Let $\eta=((\eta_{1,t},\eta_{2,t})')_{t\ge 1}$ be an iid sequence of random variables such that
\begin{equation*}
\left(\begin{array}{c}\eta_{1,t}\\\eta_{2,t}\end{array}\right) \  \overset{\text{law}}{=} \ {\cal
N}(0,I_2).
\end{equation*}
We first consider the strong VARMA case by assuming that  the innovation process $\epsilon=((\epsilon_{1,t},\epsilon_{2,t})')_{t\ge 1}$ in
(\ref{VARMA11MonteCarlo}) is defined by an iid sequence such that
\begin{equation} \label{bruitfort}
\left(\begin{array}{c}\epsilon_{1,t}\\\epsilon_{2,t}\end{array}\right) \  \overset{\text{law}}{=} \ {\cal
N}(0,I_2).
\end{equation}
The GARCH models have numerous extensions to the multivariate framework (see
\cite{blr2006} for a
review). Jeantheau (see \cite{j1998}) has proposed a simple extension of
the multivariate GARCH$(p,q)$ with conditional constant correlation.
For simplicity we consider the following bivariate ARCH$(1)$ model
proposed in \cite{j1998} and defined by
\begin{eqnarray}\label{ARCH1}
\left(\begin{array}{c}\epsilon_{1,t}\\\epsilon_{2,t}\end{array}\right)
&=&\left(\begin{array}{cc}h_{11,t}&0\\0&h_{22,t}\end{array}\right)
\left(\begin{array}{c}\eta_{1,t}\\\eta_{2,t}\end{array}\right)
\end{eqnarray}
where
\begin{eqnarray*}
\left(\begin{array}{c}h_{11,t}^{2}\\h_{22,t}^{2}\end{array}\right)=
\left(\begin{array}{c}0.3\\0.2\end{array}\right)+
\left(\begin{array}{cc}0.45&0.00\\0.40&0.25\end{array}\right)
\left(\begin{array}{c}\epsilon_{1,t-1}^{2}\\
\epsilon_{2,t-1}^{2}\end{array}\right).
\end{eqnarray*}
First we assume that in (\ref{VARMA11MonteCarlo}) the
innovation process $\epsilon$ is an ARCH$(1)$ model defined in (\ref{ARCH1}).
In three other sets of experiments, we assume that in (\ref{VARMA11MonteCarlo})
the noise process $\epsilon$ is defined by
\begin{align}
\label{multiPT}
\left(\begin{array}{c}\epsilon_{1,t}\\\epsilon_{2,t}\end{array}\right) & =
\left(\begin{array}{c}\eta_{1,t}\eta_{2,t-1}\eta_{1,t-2}\\\eta_{2,t}\eta_{1,t-1}\eta_{2,t-2}\end{array}\right), \\ \label{multiPTcarre}
\left(\begin{array}{c}\epsilon_{1,t}\\\epsilon_{2,t}\end{array}\right) & =
\left(\begin{array}{c}\eta_{1,t}^2\eta_{2,t-1}\eta_{1,t-2}\\\eta_{2,t}^2\eta_{1,t-1}\eta_{2,t-2}\end{array}\right), \\
 \label{multiRT}
\left(\begin{array}{c}\epsilon_{1,t}\\\epsilon_{2,t}\end{array}\right) & =
\left(\begin{array}{c}\eta_{1,t}(|\eta_{1,t-1}|+1)^{-1}\\\eta_{2,t}(|\eta_{2,t-1}|+1)^{-1}\end{array}\right).
\end{align}
These noises are direct extensions of the weak noises defined in
\cite{rt1996} in the univariate case.

\subsection{Empirical size}

We first simulate  $N=1,000$ independent trajectories of size $n=10,000$ of models (\ref{ARMA01MonteCarlo}) and (\ref{VARMA11MonteCarlo}) (the same series of random numbers is used to generate the noises for the different cases). The same series is  partitioned as three series of sizes $n=500$, $n=2,000$ and $n=10,000$.
For each of these $N$ replications,  we use the quasi-maximum likelihood estimation method to estimate the coefficient $\vartheta_0$ and we apply portmanteau tests to the residuals for different values of $m$, where $m$ is the number of autocorrelations used in the portmanteau test statistic. For instance, $m\in\{1,\dots,5\}$ in the VARMA case and $m\in\{1,2,3,6,12\}$ in the univariate ARMA case.
The nominal asymptotic level of the tests is $\alpha=5\%$.

\subsubsection{Strong ARMA and VARMA models case}
We consider the strong ARMA model (\ref{ARMA01MonteCarlo})--(\ref{bruitARCH}), with $(\alpha_1{,}\beta_1)=(0{,}0)$ and of strong VARMA model (\ref{VARMA11MonteCarlo})--(\ref{bruitfort}).

For the standard  Box-Pierce test, the model is therefore rejected when the statistic $Q_m^{\textsc{bp}}$ or $Q_m^{\textsc{lb}}$ (resp. $Q_m^{\textsc{c}}$ or $Q_m^{\textsc{h}}$) is larger than $\chi_{(m-2)}^2(0.95)$ (resp. than $\chi_{(4m-8)}^2(0.95)$) in a univariate ARMA$(1,1)$ case (resp. in a VARMA$(1,1)$ case).
We know that the asymptotic level of this test  is indeed $\alpha=5\%$ when $\vartheta_0=(0,0)'$ (resp. $\vartheta_0=0\in\mathbb R^8$)  in a univariate ARMA$(1,1)$ case (resp. in a bivariate VARMA$(1,1)$ case).
Note however that, even when the noise is strong,  the asymptotic level is not exactly $\alpha=5\%$ when $\vartheta_0\neq(0,0)'$ (resp. $\vartheta_0\neq0\in\mathbb R^8$)  in a univariate ARMA case (resp. in a bivariate VARMA case).

For the proposed modified test $\mathrm{BP}_\textsc{sn}$ or $\mathrm{LB}_\textsc{sn}$, the model is rejected when the statistic $
Q_m^\textsc{sn}$ or $\tilde{Q}_m^\textsc{sn}$ is larger than $\mathcal{U}_m(0.95)$ in the univariate ARMA case and than $\mathcal{U}_{4m}(0.95)$ in the bivariate VARMA case, where the critical values $\mathcal{U}_K(0.95)$ (for $K=1,\dots,20$) are tabulated  in Lobato (see Table 1 in \cite{lobato}).

Table \ref{tab1ARMA}  (resp. Table \ref{tab1VARMA}) displays the relative rejection frequencies of the null hypothesis $H_0$ that the  data generating process follows an ARMA$(1,1)$ (resp. a bivariate VARMA$(1,1)$),
over the $N$ independent replications. When one uses the statistic $Q_m^\textsc{bp}$ or $Q_m^\textsc{lb}$, the observed relative rejection frequency of the standard Box-Pierce or Ljung-Box test is very far from the nominal level $\alpha=5\%$ when the number $m$ of autocorrelations  used in the statistic is small. This observation is in accordance with the results in the literature on the standard (V)ARMA models. The theory  that the $\chi_{(m-2)}^2$ (resp. the $\chi_{(4m-8)}^2(0.95)$) approximation
 is better for larger $m$ is confirmed. In contrast, our modified test
better controls the error of first kind in the univariate and multivariate ARMA model, even when $m$ is small.
Note that the  tests based on $\mathrm{{LB}}_\textsc{bm}$, ${\mathrm{BP}}_\textsc{bm}$, $\mathrm{BP}_\textsc{frz}$ or $\mathrm{LB}_\textsc{frz}$ control well the error of first kind in the univariate ARMA case (see Model I of Table \ref{tab1ARMA}) and also in the bivariate VARMA case (see Model I of Table \ref{tab1VARMA}). Note that for $m\leq2$, the empirical size is not available (n.a.) for the standard Box-Pierce or Ljung-Box tests because they are not applicable when $m\leq 2$. For $n=200$ (the results are not reported here but are presented in the on-line extended version) our test still accurately controls the type I error when $m\leq5$ in the multivariate case and for $m\leq 12$ in the univariate case.

From these examples we draw the conclusion that  the proposed modified version are preferable to the standard ones in these univariate and multivariate strong ARMA models.

\subsubsection{Weak ARMA and VARMA models case}

We repeat the same experiments on weak (V)ARMA models (see the different models that are proposed in the tables).
We first consider the univariate  weak ARMA$(1,1)$ models.
As expected, Tables \ref{tab1ARMA} and  \ref{tab1ARMAsuite}  show that the
standard  $\mathrm{{LB}}_\textsc{s}$ or $\mathrm{{BP}}_\textsc{s}$ test poorly performs in assessing the adequacy of
these weak ARMA models (see Models II, III, IV and V in the tables). In view of the observed relative rejection frequencies, the standard test rejects very often the true ARMA$(1,1)$ and all the relative rejection frequencies are very far from the nominal $\alpha=5\%$. Our modified test and the  tests based on $\mathrm{BP}_\textsc{frz}$ or $\mathrm{LB}_\textsc{frz}$
control well the error of first kind for these weak ARMA models except for  Model  III when $n=500$ and $m=12$ (see Table \ref{tab1ARMAsuite}).

When $n=200$ (the results are not reported here, see the on-line extended version for precisions) our test still accurately controls the type I error when $m\leq 12$ for Model V, but the results are less satisfactory for the other cases.

Now, we consider the multivariate  weak ARMA$(1,1)$ models.
As expected, Tables \ref{tab1VARMA} and \ref{tab1VARMAsuite} show that the
standard tests $\mathrm{{LB}}_\textsc{s}$ and $\mathrm{{BP}}_\textsc{s}$ poorly perform in assessing the adequacy of all these weak VARMA models.
Table \ref{tab1VARMAsuite} shows that the error of first kind is
well controlled by all the tests $\mathrm{{LB}}_\textsc{sn}$, ${\mathrm{BP}}_\textsc{sn}$, $\mathrm{BP}_\textsc{bm}$ and $\mathrm{LB}_\textsc{bm}$ in the particular case of Model V.
For the models II, III and IV (see Table \ref{tab1VARMA} and \ref{tab1VARMAsuite}) the error of first kind is also
well controlled by our modified test when $n=10,000$. When $n\le 2,000$ the results are less satisfactory except for Model II when $n=2,000$.
When $n=500$, the results are acceptable for $m=1$. This is not surprising since the number of parameters in the model is eight and since the dependence structure is complex.

For $n=200$ (the results are not reported here) our test still accurately controls the type I error for Model V, but the results are not satisfactory in the other weak VARMA cases.

As a conclusion, the tests based on $Q_m^\textsc{sn}$  and $\tilde Q_m^\textsc{sn}$ can be used safely for small and large $m$ for both ARMA and VARMA models.

\subsection{Empirical power}\label{emp-power}
In this part we present the results on the univariate ARMA models. The results on the VARMA models are not presented here because the numerical results and
their interpretations are very similar to those we will now discuss in the univariate case. Nevertheless, some tables are proposed in the extended version of this work.

In this section we repeat the same experiments as in Section \ref{univARMA} to examine the power of the tests for the null hypothesis of an ARMA$(1,1)$ against the following ARMA$(2,1)$ alternative defined by
\begin{eqnarray}
\label{lutkepohl}
X_{t}=X_{t-1}-0.2X_{t-2}+\epsilon_{t}+0.8\epsilon_{t-1},
\end{eqnarray}
where the innovation process  $\epsilon$ follows a strong or weak white noise introduced in Section \ref{univARMA}.

For each of these $N$ replications we fit an ARMA$(1,1)$ model
and perform standard and modified tests based on $m=1,2,3$, $6$ and $12$ residual autocorrelations.

For these particular strong and weak ARMA models, we have seen that the actual level of the standard version is generally very different from  the $5\%$ nominal level (see Tables \ref{tab1ARMA} and \ref{tab1ARMAsuite}). In order to compare the powers  of the three tests on an equal basis,  we  use (as in  \cite{H1996}) the empirical critical values based on $N=1000$ replications. This
power is usually referred as size  adjusted power.  Tables \ref{tab2ARMA} and  \ref{tab2ARMAsuite} display the relative
rejection frequencies of over the $N$ independent replications for the univariate ARMA models.

Thanks to these examples, we notice that the standard and modified versions  tests have very similar powers when $n\geq2,000$. In contrast,
when  $n=500$ our tests are clearly less powerful than the standard ones. But we repeat that our tests are usually closer from $5\%$ nominal level.

Note that, the empirical critical values strongly depend on the type of weak ARMA which is generated under the null hypothesis. Therefore, this method consisting
in adjusting the critical values only works for very specific hypotheses.  We also think that this method may be distorted by a too small number of replication that yield to an inaccurate estimation of the empirical critical value. However, in terms of a day-to-day application, we usually wouldn't do this. In that case, what's of interest is the "raw power" of the test, that is to say the rate at which it rejects false null hypotheses when the asymptotic ("wrong") critical value is used. The study of the ''raw-power'' is presented in the extended version of our work.

\section{Illustrative example}\label{ill}
We now consider an application to the daily $\log$ returns (also simply called the returns) of the CAC40\footnote{
The CAC 40 index is a linear combination of a selection of 40 shares on the Paris Stock Exchange (CAC
stands for <<Cotations Assist\'{e}es en Continu>>). }.
The observations cover the period from March 1, 1990 to July 26, 2010.
The length of the series is $n=5154$. The data can be downloaded
from the website Yahoo Finance: http://fr.finance.yahoo.com/.
Figure~\ref{figcac40} plots the closing prices and returns of the CAC40 index from March 1, 1990 to July 26, 2010.
\begin{figure}[h!]
\centering
\includegraphics[width=0.9\textwidth]{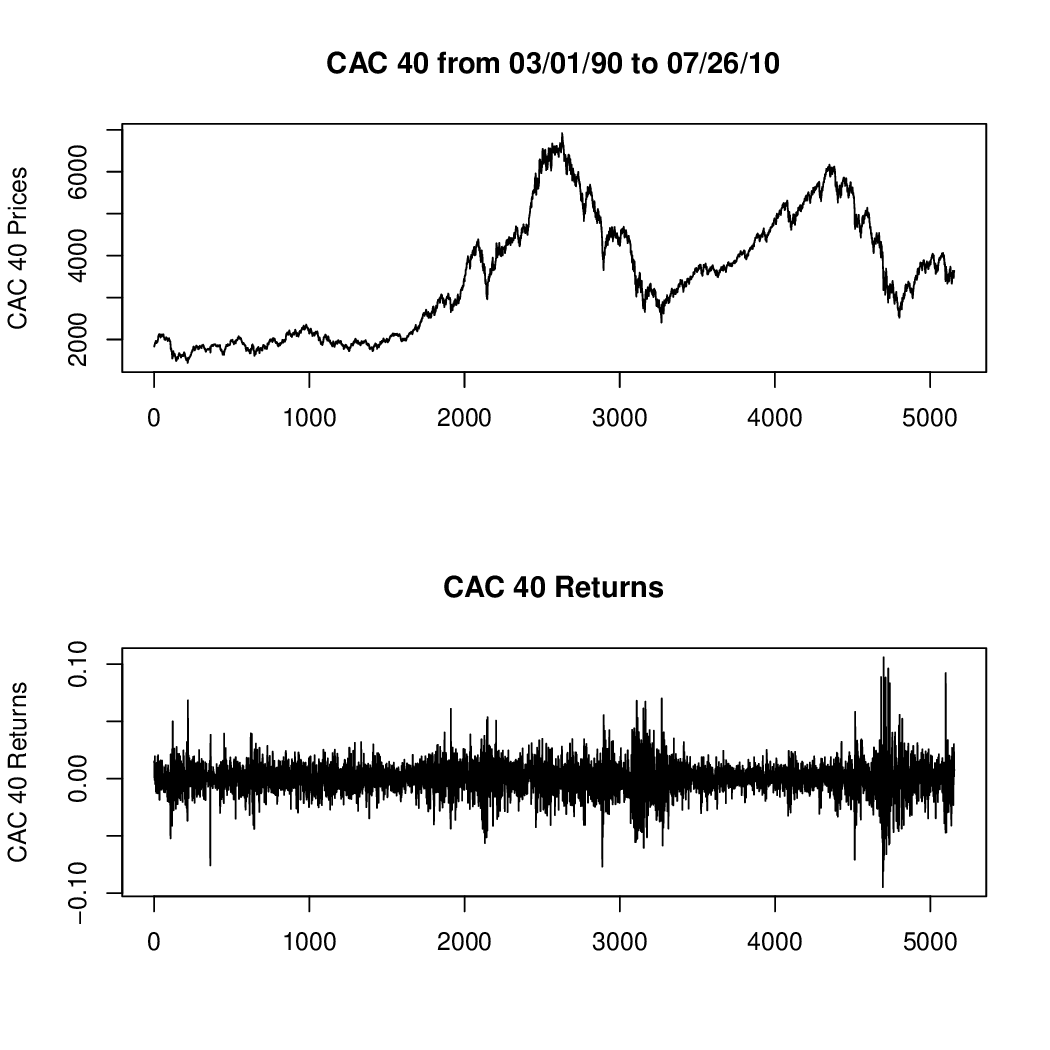}
\caption{\label{figcac40} {\footnotesize Closing prices and returns of the CAC40 index from March 1, 1990 to July 26, 2010 (5154 observations). }}
\end{figure}
It shows that the CAC index series are generally close to a random
walk without intercept and that the returns are generally compatible with the second-order stationarity
assumption.
In financial econometrics, the returns are often assumed to be martingale increments
(though they are not generally independent
sequences). Moreover it is commonly accepted that the squares of the returns have second-order moments close to those of an ARMA(1,1)
(which is compatible with a GARCH(1,1) model for the returns).
We will test these hypotheses by fitting weak ARMA models on the returns and on their squares.

First, we apply portmanteau tests
for checking the hypothesis that the CAC40 returns constitute a white noise. Table \ref{cac40BB} displays the statistics of the standard and modified tests.

Since the $p$-values of the standard test are very small, the white noise hypothesis is rejected at the nominal level $\alpha=1\%$.
This is not surprising because the standard tests requires the iid assumption and it is well known that the strong white noise model is not adequate for these series (one can think about the
so-called volatility clustering phenomena). In contrast, the  white noise hypothesis is not rejected by the modified tests since the statistic  is not larger than the critical values (see Table 1 in \cite{lobato}). 
To summarize, the outputs of Table~\ref{cac40BB} are in accordance with the
common belief that these series are not strong white noises, but could be weak white noises. This is also in accordance with other works devoted to the analysis of stock-market returns \cite{lobatoNS2001}.

Then we focus on the dynamics of the squared returns and we fit an ARMA$(1,1)$ model to the squares of the CAC40 returns.
Denoting by $(X_t)_{t\ge 1}$ the mean corrected
series of the squared returns, we obtain the following model
\begin{eqnarray*}
X_{t}=0.97942X_{t-1}+\epsilon_{t}-0.89094\epsilon_{t-1},
\text{ where }\mbox{Var}(\epsilon_{t})=23.5302\times10^{-8}.
\end{eqnarray*}
Table \ref{cac40ARMA11} displays the statistics of the standard and modified LB and BP tests.
From Table \ref{cac40ARMA11} we draw the same conclusion on the squares of the previous daily returns: the strong ARMA$(1,1)$ model is
rejected, but a weak ARMA$(1,1)$ model is not rejected. Note that the first and second-order structures that we found for the CAC40 returns, namely
a weak white noise for the returns and a weak ARMA$(1,1)$ model for the squares of the returns, are
compatible with a GARCH$(1,1)$ model.

We emphasize the fact that the assumption of second-order stationarity  can be considered
for this series. Indeed the estimated autoregressive coefficient is $0.97942<1$ and the estimated standard deviation is 0.00569. The procedure described in Proposition 8.1 p. 186 in \cite{FZ2010} yields a  $p$-value equals to 0,00799 for the test of the second-order stationarity assumption.

We mention that another illustrative example with the Standard \& Poor's 500 index is proposed in the extended version of this work.
\begin{center}
{\bf Acknowledgements}
\end{center}
The authors thank the referees for the very careful reading of the paper.

The research of Y. Boubacar Ma\"{i}nassara was supported
by a BQR (Bonus Qualit\'{e} Recherche) of the Universit\'{e} de Franche-Comt\'{e}.

\section{Tables}\label{table}

\newpage

\begin{table}[hbt]
 \caption{\small{Empirical size (in \%) of the modified and standard versions
 of the LB and BP tests in the case of ARMA$(1,1)$. 
The nominal asymptotic level of the tests is $\alpha=5\%$.
The number of replications is $N=1000$. }}
\begin{center}
\begin{tabular}{ccc ccc ccc}
\hline\hline \\
Model& Length $n$ & Lag $m$ & $\mathrm{{LB}}_{\textsc{sn}}$&$\mathrm{BP}_{\textsc{sn}}$&$\mathrm{{LB}}_{\textsc{frz}}$&$\mathrm{BP}_{\textsc{frz}}$&$\mathrm{{LB}}_{\textsc{s}}$&$\mathrm{BP}_{\textsc{s}}$
\vspace*{0.2cm}\\\hline
&& $1$&4.6 &4.5 &5.3& 5.3&n.a. &n.a.\\
&& $2$&3.9  & 3.9  & 5.4  & 5.4&n.a. &n.a.\\
I &$n=500$& $3$&3.6 & 3.6&  4.6&  4.4 &14.5 &14.3\\
 &&$6$&3.9& 3.7& 4.7& 4.6& 7.2& 7.2\\
 &&$12$&3.4& 3.2& 4.7& 4.1& 5.6& 5.2\\
  \cline{2-9}
&& $1$&4.6& 4.6& 5.1& 5.1&n.a. &n.a.\\
&& $2$&5.1  & 5.1  & 5.7  & 5.7&n.a. &n.a.\\
I &$n=2,000$& $3$&4.0 & 3.9 & 4.8 & 4.8& 13.8& 13.8\\
 &&$6$&3.5& 3.5& 5.0& 4.9& 7.2& 7.1\\
 &&$12$&5.1& 5.0 &5.3 &5.3& 5.9& 5.8\\
  \cline{2-9}
&& $1$&7.0& 7.0 &5.6& 5.6 &n.a. &n.a.\\
&& $2$&5.8  & 5.8  & 6.9  & 6.9&n.a. &n.a.\\
I &$n=10,000$& $3$&5.4 & 5.4 & 5.9 & 5.9& 16.0& 16.0\\
 &&$6$&5.9& 5.9& 5.6& 5.6& 8.6& 8.6\\
 &&$12$&4.6& 4.4& 6.2& 6.2& 7.3& 7.3\\

\hline
&& $1$& 3.2& 3.1& 4.7& 4.7&n.a. &n.a.\\
&& $2$&3.6 &  3.6  & 5.1  & 5.1&n.a. &n.a.\\
II &$n=500$& $3$&3.0 & 3.0 & 3.6 & 3.5& 25.6& 25.3\\
 &&$6$&1.8 & 1.8 & 2.4&  2.2& 17.0& 16.5\\
 &&$12$&4.8 & 4.7 & 6.4 & 6.0& 17.1& 16.0\\
 \cline{2-9}
&& $1$&4.7& 4.6& 4.2& 4.1&n.a. &n.a.\\
&& $2$&4.7  & 4.7  & 5.0  & 5.0&n.a. &n.a.\\
II &$n=2,000$& $3$&4.4 & 4.4 & 4.3 & 4.3 &26.8 &26.6\\
 &&$6$& 3.2 & 3.2 & 4.2&  4.1& 19.1& 19.0\\
 &&$12$&7.3 & 7.3 & 4.6 & 4.6 &19.9& 19.7\\
  \cline{2-9}
&& $1$&5.8& 5.8& 4.1& 4.1&n.a. &n.a.\\
&& $2$&4.6  & 4.6  & 5.0  & 5.0&n.a. &n.a.\\
II &$n=10,000$& $3$&5.5 & 5.5 & 5.1 & 5.1& 25.0& 25.0\\
 &&$6$& 4.1 & 4.1 & 4.8&  4.7& 19.5& 19.5\\
 &&$12$&6.9 & 6.9 & 3.7 & 3.7& 22.5& 22.5\\
  \hline
&& $1$& 2.9& 2.9& 3.9& 3.7&n.a. &n.a.\\
&& $2$&3.6  & 3.6  & 3.5  & 3.4&n.a. &n.a.\\
III &$n=500$& $3$&3.2 & 3.2 & 3.4 & 3.4& 20.9& 20.7\\
 &&$6$&2.7 & 2.7 & 2.3 & 2.2 &10.8 & 9.8\\
 &&$12$&1.1& 1.1& 2.5& 1.9& 7.4& 6.9\\

 \cline{2-9}
&& $1$&4.9 &4.8& 4.2& 4.2&n.a. &n.a.\\
&& $2$&4.5  & 4.5  & 4.5  & 4.5&n.a. &n.a.\\
III &$n=2,000$& $3$&5.5 & 5.5 & 4.3 & 4.3 &21.3& 21.0\\
 &&$6$&4.7 & 4.7 & 3.8 & 3.8& 11.4& 11.3\\
 &&$12$&4.4& 4.4& 3.2& 3.0& 8.8& 8.7\\
  \cline{2-9}
&& $1$&4.5& 4.5& 5.1& 5.1&n.a. &n.a.\\
&& $2$&4.7  & 4.7  & 4.9  & 4.9&n.a. &n.a.\\
III &$n=10,000$& $3$&4.4 & 4.4 & 4.8 & 4.8& 22.3& 22.3
\\
 &&$6$&5.0 & 5.0&  5.4 & 5.4 &11.9 &11.9\\
 &&$12$&3.9& 4.0& 4.9& 4.8& 8.0& 8.0\\
\hline\hline
\\
\multicolumn{9}{l}{I: Strong ARMA$(1,1)$ model
(\ref{ARMA01MonteCarlo})-(\ref{bruitARCH}) with $(\alpha_1,\beta_1)=(0,0).$}\\
\multicolumn{9}{l}{II: Weak ARMA$(1,1)$ model
(\ref{ARMA01MonteCarlo})-(\ref{bruitARCH}) with $(\alpha_1,\beta_1)=(0.1,0.85)$.}\\
\multicolumn{9}{l}{III: Weak ARMA$(1,1)$ model
(\ref{ARMA01MonteCarlo})-(\ref{PT}).}\\
\end{tabular}
\end{center}
\label{tab1ARMA}
\end{table}

%

\begin{table}[hbt]
 \caption{\small{Empirical size (in \%) of the modified and standard versions  of the LB and BP tests in the case of ARMA$(1,1)$. 
The nominal asymptotic level of the tests is $\alpha=5\%$.
The number of replications is $N=1000$. }}
\begin{center}
\begin{tabular}{ccc ccc ccc}
\hline\hline \\
Model& Length $n$ & Lag $m$ & $\mathrm{{LB}}_{\textsc{sn}}$&$\mathrm{BP}_{\textsc{sn}}$&$\mathrm{{LB}}_{\textsc{frz}}$&$\mathrm{BP}_{\textsc{frz}}$&$\mathrm{{LB}}_{\textsc{s}}$&$\mathrm{BP}_{\textsc{s}}$
\vspace*{0.2cm}\\\hline
&& $1$&3.3& 3.1& 3.2& 3.2&n.a. &n.a.\\
&& $2$&3.5  & 3.5  & 4.5 &  4.2&n.a. &n.a.\\
IV &$n=500$& $3$&3.0 & 2.7 & 2.9 & 2.9& 23.5 &23.4\\
 &&$6$&0.9 & 0.8&  1.7&  1.4& 13.5& 13.3\\
 &&$12$&1.1& 1.0& 2.7& 2.7& 8.8& 8.7\\
  \cline{2-9}
&& $1$&5.0& 5.0& 5.7& 5.7&n.a. &n.a.\\
&& $2$&5.2 &  5.2  & 5.6  & 5.6&n.a. &n.a.\\
IV &$n=2,000$& $3$&4.1 & 4.1 & 3.8&  3.8& 28.1& 28.0\\
 &&$6$&3.9  &3.9 & 2.5 & 2.5& 14.1& 14.1\\
 &&$12$&1.9 & 1.8 & 2.1&  2.1& 10.5& 10.1\\
  \cline{2-9}
&& $1$&6.1& 6.1& 5.5& 5.5&n.a. &n.a.\\
&& $2$&5.9  & 5.9  & 5.3  & 5.3&n.a. &n.a.\\
IV &$n=10,000$& $3$&5.6 & 5.6 & 5.0 & 5.0& 29.2 &29.2\\
 &&$6$&6.3&  6.3&  4.2&  4.2& 16.8 &16.8\\
 &&$12$&3.5 & 3.5  &3.8 & 3.8& 13.4 &13.4\\

\hline
&& $1$&4.4& 4.4& 5.4& 5.5&n.a. &n.a.\\
&& $2$&6.3   &6.2   &5.8  & 5.8&n.a. &n.a.\\
V &$n=500$& $3$&5.5 & 5.1 & 5.6 & 5.5& 11.6& 11.6\\
 &&$6$&4.6& 4.5& 5.2& 5.1& 7.6& 7.2\\
 &&$12$&4.8& 4.6& 5.1 &4.6 &6.1 &5.7\\
 \cline{2-9}
&& $1$&5.6& 5.6& 4.6& 4.6&n.a. &n.a.\\
&& $2$&5.2  & 5.2  & 5.1  & 5.1&n.a. &n.a.\\
V &$n=2,000$& $3$&4.3 & 4.3 & 4.8&  4.7& 12.6& 12.6\\
 &&$6$&4.3 &4.2& 5.2& 5.1& 7.0& 7.0
\\
 &&$12$&4.2& 4.1& 4.5& 4.5& 5.0& 5.0\\
  \cline{2-9}
&& $1$&6.0& 6.0& 5.4& 5.4&n.a. &n.a.\\
&& $2$&6.4  & 6.4   &6.1  & 6.1 &n.a. &n.a.\\
V &$n=10,000$& $3$&6.2 & 6.2 & 5.6&  5.6& 13.7& 13.7\\
 &&$6$& 5.0& 5.0& 5.4& 5.4& 7.7& 7.7\\
 &&$12$&5.1& 5.2& 4.5& 4.5& 5.4& 5.4\\
\hline\hline \\
\multicolumn{9}{l}{IV: Weak ARMA$(1,1)$ model
(\ref{ARMA01MonteCarlo})-(\ref{PTcarre}).}\\
\multicolumn{9}{l}{V: Weak ARMA$(1,1)$ model
(\ref{ARMA01MonteCarlo})-(\ref{RT}).}\\
\end{tabular}
\end{center}
\label{tab1ARMAsuite}
\end{table}

\begin{table}[hbt]
 \caption{\small{Empirical size  adjusted power (in \%) of the modified and standard versions
 of the LB and BP tests at the 5\% nominal level  in the case of ARMA$(2,1)$ model.
The number of replications is $N=1000$. }}
\begin{center}
\begin{tabular}{ccc ccc ccc}
\hline\hline \\
Model& Length $n$ & Lag $m$ & $\mathrm{{LB}}_{\textsc{sn}}$&$\mathrm{BP}_{\textsc{sn}}$&$\mathrm{{LB}}_{\textsc{frz}}$&$\mathrm{BP}_{\textsc{frz}}$&$\mathrm{{LB}}_{\textsc{s}}$&$\mathrm{BP}_{\textsc{s}}$
\vspace*{0.2cm}\\\hline
&& $1$&80.8 & 80.8 &100.0 &100.0&n.a. &n.a.\\
&& $2$&66.3 & 66.5 &100.0 &100.0&n.a. &n.a.\\
I &$n=500$& $3$&64.9 & 65.0& 100.0& 100.0 & 85.9 & 85.9\\
 &&$6$&55.3 & 55.4& 100.0& 100.0 & 78.5 & 78.7\\
 &&$12$&42.3 & 42.7& 100.0& 100.0 & 71.0 & 71.3\\
 \cline{2-9}
&& $1$&99.7 & 99.7 &100.0 &100.0&n.a. &n.a.\\
&& $2$& 97.5 & 97.5 &100.0& 100.0&n.a. &n.a.\\
I &$n=2,000$& $3$&97.2 & 97.2 &100.0 &100.0 &100.0 &100.0\\
 &&$6$&97.6 & 97.6& 100.0 &100.0& 100.0& 100.0\\
 &&$12$&95.2 & 95.3& 100.0 &100.0 &100.0& 100.0\\
  \cline{2-9}
&& $1$&100.0& 100.0 &100.0 &100.0&n.a. &n.a.\\
&& $2$&100.0& 100.0 &100.0& 100.0 &n.a. &n.a.\\
I &$n=10,000$& $3$&100.0& 100.0 &100.0 &100.0& 100.0& 100.0\\
 &&$6$&100.0& 100.0 &100.0 &100.0& 100.0& 100.0\\
 &&$12$&100.0& 100.0 &100.0 &100.0& 100.0& 100.0\\

\hline
&& $1$&69.5 & 69.5 & 100.0& 100.0 &n.a. &n.a.\\
&& $2$&54.4 & 54.4 & 99.9 & 99.9&n.a. &n.a.\\
II &$n=500$& $3$&50.0& 50.2& 99.9& 99.9& 71.3& 71.3\\
 &&$6$&38.9 & 38.9& 100.0 &100.0 & 61.1 & 61.1\\
 &&$12$&14.6 & 14.8 &100.0 &100.0 & 50.5 & 51.0
\\
 \cline{2-9}
&& $1$&98.0 & 98.0&100.0 &100.0&n.a. &n.a.\\
&& $2$& 92.3  &92.3 &100.0 &100.0&n.a. &n.a.\\
II &$n=2,000$& $3$&92.6 & 92.6& 100.0& 100.0& 100.0& 100.0\\
 &&$6$&92.7 & 92.6 &100.0& 100.0 & 99.8 & 99.8\\
 &&$12$&50.6 & 50.8& 100.0& 100.0 & 99.8&  99.8\\
  \cline{2-9}
&& $1$&100.0 &100.0&100.0 &100.0&n.a. &n.a.\\
&& $2$&100.0 &100.0& 100.0& 100.0&n.a. &n.a.\\
II &$n=10,000$& $3$&99.9 & 99.9 & 100.0& 100.0& 100.0& 100.0\\
 &&$6$&100.0 &100.0& 100.0 &100.0 &100.0& 100.0\\
 &&$12$&100.0 &100.0& 100.0& 100.0& 100.0& 100.0\\
  \hline
&& $1$&60.9 & 60.9& 100.0& 100.0&n.a. &n.a.\\
&& $2$&47.8 & 47.8 & 99.7 & 99.7&n.a. &n.a.\\
III &$n=500$& $3$&40.4& 40.5& 99.6& 99.6& 71.1& 71.0\\
 &&$6$&30.4& 30.3& 99.9& 99.9& 68.4& 68.6\\
 &&$12$&20.2& 19.3& 99.8& 99.8& 64.8& 65.1\\
 \cline{2-9}
&& $1$&95.3 & 95.3& 100.0& 100.0&n.a. &n.a.\\
&& $2$&91.1 & 91.1 &100.0 &100.0&n.a. &n.a.\\
III &$n=2,000$& $3$&85.0 & 84.9 & 100.0 &100.0 & 99.9&  99.9
\\
 &&$6$&80.5 & 80.5& 100.0 &100.0  &99.9 & 99.9\\
 &&$12$&73.0 & 74.0& 100.0& 100.0 & 99.8&  99.8\\
  \cline{2-9}
&& $1$&100.0& 100.0 & 100.0 &100.0&n.a. &n.a.\\
&& $2$&100.0& 100.0 &100.0& 100.0&n.a. &n.a.\\
III &$n=10,000$& $3$&99.6 & 99.6& 100.0 &100.0 &100.0 &100.0\\
 &&$6$&99.7 & 99.7& 100.0 &100.0 &100.0& 100.0\\
 &&$12$&100.0& 100.0& 100.0 &100.0 &100.0 &100.0\\
\hline\hline \\
\multicolumn{9}{l}{I: Strong ARMA$(2,1)$ model
(\ref{lutkepohl})-(\ref{bruitARCH}) with $(\alpha_1,\beta_1)=(0,0)$.}\\
\multicolumn{9}{l}{II: Weak ARMA$(2,1)$ model
(\ref{lutkepohl})-(\ref{bruitARCH}) with $(\alpha_1,\beta_1)=(0.1,0.85)$.}\\
\multicolumn{9}{l}{III: Weak ARMA$(2,1)$ model
(\ref{lutkepohl})-(\ref{PT}).}\\
\end{tabular}
\end{center}
\label{tab2ARMA}
\end{table}

\begin{table}[hbt]
 \caption{\small{Empirical size  adjusted power (in \%) of the modified and standard versions
 of the LB and BP tests at the 5\% nominal level in the case of ARMA$(2,1)$ model.
The number of replications is $N=1000$. }}
\begin{center}
\begin{tabular}{ccc ccc ccc}
\hline\hline \\
Model& Length $n$ & Lag $m$ & $\mathrm{{LB}}_{\textsc{sn}}$&$\mathrm{BP}_{\textsc{sn}}$&$\mathrm{{LB}}_{\textsc{frz}}$&$\mathrm{BP}_{\textsc{frz}}$&$\mathrm{{LB}}_{\textsc{s}}$&$\mathrm{BP}_{\textsc{s}}$
\vspace*{0.2cm}\\\hline
&& $1$&50.0 &50.0& 99.7 &99.7&n.a. &n.a.\\
&& $2$&34.4 & 34.4 & 99.8 & 99.8&n.a. &n.a.\\
IV &$n=500$& $3$&30.6& 30.3& 99.9& 99.9& 64.8& 64.8\\
 &&$6$&27.9& 28.2& 99.9 &99.9 &63.0& 63.0\\
 &&$12$&17.1 & 17.8& 100.0 &100.0 & 58.1 & 58.1\\
 \cline{2-9}
&& $1$&84.9&  84.9 &100.0 &100.0&n.a. &n.a.\\
&& $2$&75.6 & 75.6 &100.0 &100.0&n.a. &n.a.\\
IV &$n=2,000$& $3$&74.2 & 74.2& 100.0& 100.0 & 99.3&  99.3
\\
 &&$6$&65.3 & 65.4& 100.0& 100.0 & 99.1 & 99.1\\
 &&$12$&63.1 & 63.1& 100.0& 100.0 & 99.2 & 99.2\\
  \cline{2-9}
&& $1$&100.0 &100.0&100.0 &100.0&n.a. &n.a.\\
&& $2$&99.4 & 99.4& 100.0 &100.0 &n.a. &n.a.\\
IV &$n=10,000$& $3$&99.0 & 99.0& 100.0& 100.0& 100.0 &100.0\\
 &&$6$&98.3 & 98.3 &100.0 &100.0 &100.0 &100.0\\
 &&$12$&99.4 & 99.4 &100.0 &100.0 &100.0 &100.0\\

\hline
&& $1$&85.6 & 85.6 &100.0 &100.0&n.a. &n.a.\\
&& $2$&72.1 & 72.1&100.0 &100.0&n.a. &n.a.\\
V &$n=500$& $3$&70.1 & 70.3& 100.0 &100.0  &88.5 & 88.5\\
 &&$6$&62.6 & 62.6 &100.0& 100.0 & 80.3 & 80.5\\
 &&$12$&48.5 & 49.2 &100.0& 100.0 & 71.0&  71.1\\
 \cline{2-9}
&& $1$&100.0 &100.0&100.0 &100.0&n.a. &n.a.\\
&& $2$&99.4 & 99.4&100.0 &100.0&n.a. &n.a.\\
V &$n=2,000$& $3$&99.4 & 99.4 &100.0 &100.0 &100.0 &100.0\\
 &&$6$&98.8 & 98.8& 100.0 &100.0 &100.0& 100.0\\
 &&$12$&99.3 & 99.3&100.0& 100.0 &100.0 &100.0\\
  \cline{2-9}
&& $1$&100.0& 100.0& 100.0& 100.0&n.a. &n.a.\\
&& $2$&100.0& 100.0& 100.0& 100.0&n.a. &n.a.\\
V &$n=10,000$& $3$&100.0& 100.0& 100.0& 100.0& 100.0& 100.0\\
 &&$6$&100.0& 100.0& 100.0& 100.0&100.0 &100.0\\
 &&$12$&100.0& 100.0& 100.0& 100.0&100.0 &100.0\\
\hline\hline \\
\multicolumn{9}{l}{IV: Weak ARMA$(2,1)$ model
(\ref{lutkepohl})-(\ref{PTcarre}).}\\
\multicolumn{9}{l}{V: Weak ARMA$(2,1)$ model
(\ref{lutkepohl})-(\ref{RT}).}\\
\end{tabular}
\end{center}
\label{tab2ARMAsuite}
\end{table}

\begin{table}[hbt]
 \caption{\small{Empirical size (in \%) of the modified  and standard  versions
 of the LB and BP tests in the case of VARMA$(1,1)$. 
The nominal asymptotic level of the tests is $\alpha=5\%$.
The number of replications is $N=1000$. }}
\begin{center}
\begin{tabular}{ccc ccc ccc}
\hline\hline \\
Model& Length $n$ & Lag $m$ & $\mathrm{{{LB}}}_{\textsc{sn}}$&$\mathrm{{{BP}}}_{\textsc{sn}}$&$\mathrm{{{LB}}}_{\textsc{bm}}$&
$\mathrm{{{BP}}}_{\textsc{bm}}$&$\mathrm{{{LB}}}_{\textsc{s}}$&$\mathrm{{{BP}}}_{\textsc{s}}$
\vspace*{0.2cm}\\\hline
&& $1$&4.8& 4.9 &4.2 &4.0&n.a. &n.a.\\
&& $2$&5.2  & 5.2  & 3.4  & 3.3&n.a. &n.a.\\
I &$n=500$& $3$&4.0 & 4.1 & 4.4 & 4.2& 20.6& 19.9
\\
 &&$4$&3.5 & 3.6 & 4.6 & 4.3& 11.6 &11.5\\
 &&$5$&3.5& 3.7& 3.9& 3.8 &8.2 &7.8\\
  \cline{2-9}
&& $1$&4.6& 4.6& 2.9& 2.9&n.a. &n.a.\\
&& $2$&4.8  & 4.8 &  3.2 &  3.1&n.a. &n.a.\\
I &$n=2,000$& $3$&4.3  &4.3 & 3.9 & 3.9 &19.6 &19.4\\
 &&$4$&5.1 & 5.1 & 3.5 & 3.4& 10.5 &10.4\\
 &&$5$&4.8& 4.8 &4.3& 4.3 &8.3 &8.1\\
  \cline{2-9}
&& $1$&5.7 &5.7 &3.7& 3.7& n.a. &n.a.\\
&& $2$&5.4  & 5.4  & 4.7  & 4.7&n.a. &n.a.\\
I &$n=10,000$& $3$&5.3 & 5.3 & 4.4 & 4.4& 19.0& 19.0\\
 &&$4$&5.7 & 5.7 & 4.5 & 4.5 &10.6& 10.6\\
 &&$5$&6.1& 6.1 &4.4& 4.3& 8.5& 8.5\\

\hline
&& $1$&3.9& 3.9& 5.3& 5.1& n.a. &n.a.\\
&& $2$&2.9  & 2.9  & 5.5  & 5.2&  n.a. &n.a.\\
II &$n=500$& $3$&2.1 & 2.1 & 2.9 & 2.7& 34.9 &34.0\\
 &&$4$&1.1 & 1.2  &2.8 & 2.7& 24.5& 23.7\\
 &&$5$&0.9 & 0.9 & 2.3 & 2.0 &18.0& 17.7\\
 \cline{2-9}
&& $1$&3.9& 3.9& 6.1 &5.9& n.a. &n.a.\\
&& $2$&5.0  & 5.0  & 7.0  & 6.9 & n.a. &n.a.\\
II &$n=2,000$& $3$&3.3 & 3.3 & 5.4 & 5.3& 43.9 &43.8
\\
 &&$4$&3.1 & 3.1  &5.5 & 5.4& 32.1& 32.0\\
 &&$5$&3.0 & 3.0 & 4.9&  4.8& 26.1& 25.9\\
  \cline{2-9}
&& $1$&4.4& 4.4& 5.3& 5.3 &  n.a. &n.a.\\
&& $2$&4.9  & 4.9  & 6.1  & 6.1 & n.a. &n.a.\\
II &$n=10,000$& $3$&3.9 & 3.9 & 5.6 & 5.6& 50.0& 49.9\\
 &&$4$&4.2 & 4.2 & 5.4 & 5.4& 37.2& 37.1\\
 &&$5$&3.3 & 3.3 & 5.6 & 5.6& 28.4& 28.3\\
  \hline
&& $1$&3.4& 3.4& 6.6& 6.5 & n.a. &n.a.\\
&& $2$&1.6  & 1.6  & 4.6 &  4.5& n.a. &n.a.\\
III &$n=500$& $3$&0.3 & 0.3&  2.1 & 2.1& 50.7& 50.1\\
 &&$4$& 0.1 & 0.1 & 1.3 & 1.3& 37.7 &36.9\\
 &&$5$&0.0 & 0.0 & 1.1 & 1.1& 29.7 &28.9\\

 \cline{2-9}
&& $1$&4.8 &4.8 &6.3 &6.3 & n.a. &n.a.\\
&& $2$& 3.7  & 3.7 &  5.4 &  5.4& n.a. &n.a.\\
III &$n=2,000$& $3$&2.7 & 2.7 & 5.0 & 5.0& 56.5& 56.4\\
 &&$4$&2.5&  2.5 & 3.6 & 3.5& 41.8& 41.3\\
 &&$5$&1.6 & 1.6 & 2.4 & 2.4 &34.4 &34.3\\
  \cline{2-9}
&& $1$&5.1& 5.1& 6.2& 6.2 & n.a. &n.a.\\
&& $2$&5.2  & 5.2  & 6.0  & 6.0& n.a. &n.a.\\
III &$n=10,000$& $3$&4.0 & 4.0 & 5.8 & 5.8& 58.6& 58.5
\\
 &&$4$&4.3 & 4.3 & 5.6 & 5.6 &44.9& 44.9\\
 &&$5$&4.2 & 4.2 & 5.7 & 5.6& 38.2 &38.2\\
\hline\hline \\
\multicolumn{9}{l}{I: Strong VARMA$(1,1)$ model
(\ref{VARMA11MonteCarlo})-(\ref{bruitfort}).}\\
\multicolumn{9}{l}{II: Weak VARMA$(1,1)$ model
(\ref{VARMA11MonteCarlo})-(\ref{ARCH1}).}\\
\multicolumn{9}{l}{III: Weak VARMA$(1,1)$ model
(\ref{VARMA11MonteCarlo})-(\ref{multiPT}).}\\
\end{tabular}
\end{center}
\label{tab1VARMA}
\end{table}

\begin{table}[hbt]
 \caption{\small{Empirical size (in \%) of the modified  and standard  versions
 of the LB and BP tests in the case of VARMA$(1,1)$. 
The nominal asymptotic level of the tests is $\alpha=5\%$.
The number of replications is $N=1000$. }}
\begin{center}
\begin{tabular}{ccc ccc ccc}
\hline\hline \\
Model& Length $n$ & Lag $m$ & $\mathrm{{{LB}}}_{\textsc{sn}}$&$\mathrm{{{BP}}}_{\textsc{sn}}$&$\mathrm{{{LB}}}_{\textsc{bm}}$&
$\mathrm{{{BP}}}_{\textsc{bm}}$&$\mathrm{{{LB}}}_{\textsc{s}}$&$\mathrm{{{BP}}}_{\textsc{s}}$
\vspace*{0.2cm}\\\hline
&& $1$&2.4& 2.4& 9.5& 9.5&n.a.&n.a.\\
&& $2$&0.7 &  0.8  & 6.5  & 6.5&n.a.&n.a.\\
IV &$n=500$& $3$&0.0 & 0.0 & 5.2 & 5.2& 52.0& 51.5\\
 &&$4$&0.2 & 0.2 & 5.1 & 5.1& 40.5& 40.1\\
 &&$5$&0.2 & 0.2 & 4.5 & 4.5& 37.3& 36.4\\
  \cline{2-9}
&& $1$&3.1& 3.1& 6.0& 6.0&n.a.&n.a.\\
&& $2$&2.5  & 2.5  & 4.8  & 4.8 &n.a.&n.a.\\
IV &$n=2,000$& $3$&2.5 & 2.5&  3.2 & 3.2& 61.8 &61.5\\
 &&$4$&2.0 & 2.0 & 3.0&  3.0& 50.2& 50.0\\
 &&$5$&0.4&  0.4&  1.8 & 1.7& 43.9& 43.8\\
  \cline{2-9}
&& $1$&4.7& 4.7& 4.1& 4.1&n.a.&n.a.\\
&& $2$&4.4 &  4.4 &  5.6  & 5.6&n.a.&n.a.\\
IV &$n=10,000$& $3$&3.8 & 3.8 & 5.2 & 5.2& 70.4& 70.4\\
 &&$4$&3.5 & 3.5 & 4.1&  4.1& 57.3& 57.3\\
 &&$5$&3.1 & 3.1 & 4.2&  4.2& 50.3& 50.3\\

\hline
&& $1$&4.0& 4.0 &3.6 &3.6&n.a.&n.a.\\
&& $2$&4.5  & 4.6  & 3.8 &  3.8 &n.a.&n.a.\\
V &$n=500$& $3$&3.9&  4.1&  4.2&  4.2& 15.1& 14.7
\\
 &&$4$&3.6& 3.8& 4.4& 4.0& 9.6& 9.5\\
 &&$5$&3.1& 3.2& 3.7& 3.4& 7.8& 7.0\\
 \cline{2-9}
&& $1$&5.5& 5.5& 3.7& 3.7 &n.a.&n.a.\\
&& $2$&5.4   &5.5   &4.4  & 4.4&n.a.&n.a.\\
V &$n=2,000$& $3$&4.5 & 4.6 & 3.8&  3.8& 16.6 &16.4\\
 &&$4$&4.4& 4.4& 4.0& 4.0& 9.5& 9.3\\
 &&$5$&4.3& 4.3& 4.3& 4.3& 7.9& 7.8\\
  \cline{2-9}
&& $1$&4.3& 4.3& 3.6& 3.6&n.a.&n.a.\\
&& $2$&4.2  & 4.2  & 3.2  & 3.2&n.a.&n.a.\\
V &$n=10,000$& $3$&4.7 & 4.7 & 3.9 & 3.9& 15.3& 15.3\\
 &&$4$&4.4& 4.4& 4.1& 4.1 &9.1& 9.1\\
 &&$5$&4.8& 4.8& 3.9& 3.9& 6.8& 6.8\\
\hline\hline \\
 \multicolumn{9}{l}{IV: Weak VARMA$(1,1)$ model
(\ref{VARMA11MonteCarlo})-(\ref{multiPTcarre}).}\\
\multicolumn{9}{l}{V: Weak VARMA$(1,1)$ model
(\ref{VARMA11MonteCarlo})-(\ref{multiRT}).}\\
\end{tabular}
\end{center}
\label{tab1VARMAsuite}
\end{table}

\begin{table}[hbt]
 \caption{\small{Modified and standard versions
 of portmanteau tests to check the null hypothesis that the CAC40 returns is a white noise. }}
\begin{center}
{\small
\begin{tabular}{c ccc cccc c}
\hline\hline
Lag $m$& &2&3&4 &5&10&18& 24\\
\hline$\hat\rho(m)$&& -0.02829& -0.05308 & 0.04064& -0.05296&0.00860&-0.02182&0.00466
\\
$\mathrm{LB}_{\textsc{sn}}$&&37.1438&65.4815&141.899&183.391&435.224&669.439&880.159
\\
$\mathrm{BP}_{\textsc{sn}}$&&37.1138&65.4176&141.727&183.144&434.646&668.402&878.556
\\
\cline{3-9}  $\mathrm{LB}_{\textsc{frz}}$&&4.88063&19.4172&27.9413&42.4158&52.7845&61.2431&67.2210
\\
$\mathrm{BP}_{\textsc{frz}}$&&4.87699&19.3994&27.9134&42.3685&52.7171&61.1467&67.0991
\\
\cline{3-9}
$\mbox{p}_{\textsc{sn}}^{\textsc{lb}}$&&0.23931&0.27453&0.18218&0.22384&0.43726&0.90622&0.98502
\\
$\mbox{p}_{\textsc{sn}}^{\textsc{bp}}$&&0.23951&0.27493&0.18250&0.22440&0.43814&0.90674&0.98519
\\
\cline{3-9}  $\mbox{p}_{\textsc{frz}}^{\textsc{lb}}$&&0.29758&0.03480&0.03837&0.00911&0.02085&0.17452&0.24341
\\
$\mbox{p}_{\textsc{frz}}^{\textsc{bp}}$&&0.29784&0.03491&0.03850&0.00916&0.02100&0.17544&0.24482
\\
\cline{3-9}  $\mbox{p}_{\textsc{S}}^{\textsc{lb}}$&&0.08713&0.00022&0.00000&0.00000&0.00000&0.00000&0.00000
\\
$\mbox{p}_{\textsc{S}}^{\textsc{bp}}$&& 0.08729&0.00022&0.00000&0.00000&0.00000&0.00000&0.00000
\\
\hline\hline
\end{tabular}
}
\end{center}
\label{cac40BB}
\end{table}

\begin{table}[hbt]
 \caption{\small{Modified and standard versions
 of portmanteau tests to check the null hypothesis that the CAC40 squared returns follow an ARMA$(1,1)$ model. }}
\begin{center}
{\small
\begin{tabular}{c ccc cccc c}
\hline\hline
Lag $m$& &1&2&3 &4&5&6& 7\\
\hline$\hat\rho(m)$&& -0.04724 & 0.010467&  0.01275 &-0.00252&  0.08579& -0.03532 &-0.02772
\\
$\mathrm{LB}_{\textsc{sn}}$&&8.96411&17.2907&21.0192&20.9689&21.0344&21.8014& 25.0933
\\
$\mathrm{BP}_{\textsc{sn}}$&& 8.95890&17.2786&21.0026&20.9526&21.0172&21.7755&25.0477
\\
\cline{3-9}
$\mathrm{LB}_{\textsc{frz}}$&&11.5095&12.0746&12.9140&12.9468&50.9362&57.3777&61.3462
\\
$\mathrm{BP}_{\textsc{frz}}$&&11.5028&12.0674&12.9061&12.9388&50.8767&57.3081&61.2697
\\
\cline{3-9}
$\mbox{p}_{\textsc{sn}}^{\textsc{lb}}$&&0.30050&0.45977&0.66164&0.84375&0.93811&0.97700&0.98987
\\
$\mbox{p}_{\textsc{sn}}^{\textsc{bp}}$&&0.30061&0.45998&0.66183&0.84391&0.93822&0.97709&0.98991
\\
\cline{3-9}  $\mbox{p}_{\textsc{frz}}^{\textsc{lb}}$&&0.11777&0.18293&0.34192&0.48658&0.29325&0.36110&0.38848
\\
$\mbox{p}_{\textsc{frz}}^{\textsc{bp}}$&&0.11789&0.18310&0.34218&0.48687&0.29368&0.36159&0.38899
\\
\cline{3-9}  $\mbox{p}_{\textsc{S}}^{\textsc{lb}}$&&n.a.&n.a.&0.00033&0.00154&0.00000&0.00000&0.00000
\\
$\mbox{p}_{\textsc{S}}^{\textsc{bp}}$&&n.a.&n.a.&0.00033&0.00155&0.00000&0.00000&0.00000
\\
\hline
Lag $m$& &8&9&10 &12&18&20& 24\\\hline
$\hat\rho(m)$&&-0.04099&0.02048&0.04568&0.03584&0.03929&-0.04526&-0.04526
\\
$\mathrm{LB}_{\textsc{sn}}$&&27.7828&27.5084&67.8011&93.7494&167.937&155.600& 364.860
\\
$\mathrm{BP}_{\textsc{sn}}$&&27.7576&27.4791&67.5936&93.1298&169.383&154.868&363.935
\\
\cline{3-9}
$\mathrm{LB}_{\textsc{frz}}$&&70.0230&72.1898&82.9700&89.6424&117.933&128.546&143.244
\\
$\mathrm{BP}_{\textsc{frz}}$&& 69.9297&72.0919&82.8470&89.5013&117.697&128.265&142.890
\\
\cline{3-9}
$\mbox{p}_{\textsc{sn}}^{\textsc{lb}}$&&0.99599&0.99896&0.99275&0.99673&0.99988&0.99999&0.99995
\\
$\mbox{p}_{\textsc{sn}}^{\textsc{bp}}$&&0.99600&0.99897&0.99278&0.99679&0.99987&0.99999&0.99995
\\
\cline{3-9}  $\mbox{p}_{\textsc{frz}}^{\textsc{lb}}$&&0.33766&0.40190&0.33819&0.39101&0.48717&0.45510&0.45367
\\
$\mbox{p}_{\textsc{frz}}^{\textsc{bp}}$&& 0.33821&0.40250&0.33892&0.39191&0.48829&0.45635&0.45503
\\
\cline{3-9}  $\mbox{p}_{\textsc{S}}^{\textsc{lb}}$&&0.00000&0.00000&0.00000&0.00000&0.00000&0.00000&0.00000
\\
$\mbox{p}_{\textsc{s}}^{\textsc{bp}}$&&0.00000&0.00000&0.00000&0.00000&0.00000&0.00000&0.00000
\\
\hline\hline
\end{tabular}
}
\end{center}
\label{cac40ARMA11}
\end{table}

\clearpage
\appendix
\begin{center}
{\bf
Complementary results that are not submitted for publication \\
They will be put into an on-line extended version}
\end{center}
\section{Proofs}\label{proof}
 We recall that the  Skorokhod space $\mathbb D^k [0{,}1]$ is the set of $\mathbb R^k-$valued functions defined on $[0{,}1]$ which are right continuous and have left limits. It is endowed with the Skorokhod topology and the weak convergence on $\mathbb D^k [0{,}1]$ is mentioned by $\xrightarrow[]{\mathbb{D}^k}$. We finally denote by $\lfloor a \rfloor$ the integer part of the real $a$.

In order to make our presentation more readable, we  restrict ourselves to the one dimensional case that requires less technical notation (hence $k_0=p+q$ in the following).
The proofs of Theorems~\ref{sn1VARMA} and \ref{sn2VARMA} are the multivariate adaptation of the proofs of Theorems~\ref{sn1} and \ref{sn2} that are detailed below.

First, we shall need some technical results which are essentially contained in \cite{frz,fz98,fz00}. They are necessary to understand the proofs, but are not essential to give the main ideas of the self-normalization approach. This is the reason why these facts are presented here.
\subsection{Reminder on technical issues on quasi likelihood method for ARMA models}\label{reminder}
We recall that, given a realization $X_1,...,X_n$ of length $n$, the noise $e_t(\vartheta)$ (see Equation \eqref{e(vartheta)}) is approximated by $\tilde e(\vartheta)$ which is defined in \eqref{etilde}.

The starting point in the asymptotic analysis is the property that $e_t(\vartheta)-\tilde e_t(\vartheta)$ converges uniformly to $0$ (almost-surely) as $t$ goes to infinity. Similar properties also hold for the derivatives with respect to $\vartheta$ of $e_t(\vartheta)-\tilde e_t(\vartheta)$. We sum up the properties that we shall need in the sequel. We refer to the appendix of \cite{fz00} (see also \cite{fz98}) for a more detailed treatment.

For any $\vartheta\in\Theta\subset \mathbb R^{k_0}$ and any $(l,m)\in\{ 1,...,k_0\}^2$, there exists absolutely summable and deterministic sequences  $(c_i(\vartheta))_{i\ge 0}$,  $(c_{i,l}(\vartheta))_{i\ge 1}$  and $(c_{i,l,m}(\vartheta))_{i\ge 1}$ such that, almost surely,
\begin{align}
\label{c1}
e_t(\vartheta)& =\sum_{i=0}^{\infty}c_i(\vartheta)X_{t-i}\ ,\  \frac{\partial e_t(\vartheta)}{\partial \vartheta_l} =\sum_{i=1}^{\infty}c_{i,l}(\vartheta)X_{t-i} \  \text{and}\   \frac{\partial^2 e_t(\vartheta)}{\partial \vartheta_l\partial\vartheta_m} =\sum_{i=2}^{\infty}c_{i,l,m}(\vartheta)X_{t-i} \\
\label{c2}
\tilde e_t(\vartheta)& =\sum_{i=0}^{t-1}c_i(\vartheta)X_{t-i}\ ,\  \frac{\partial \tilde e_t(\vartheta)}{\partial \vartheta_l} =\sum_{i=1}^{t-1}c_{i,l}(\vartheta)X_{t-i} \ \text{and}\   \frac{\partial^2 \tilde e_t(\vartheta)}{\partial \vartheta_l\partial\vartheta_m} =\sum_{i=2}^{t-1}c_{i,l,m}(\vartheta)X_{t-i}.
\end{align}
We strength the fact that $c_0(\vartheta)=1$ in the above identities. A useful property of the above three sequences that they are asymptotically exponentially small. Indeed there exists $\rho\in ]0{,}1[$ and a positive constant $K$ such that, for all $i\ge 1$,
\begin{align*}
\sup_{\vartheta\in \Theta} \Big ( |c_i(\vartheta)| +|c_{i,l}(\vartheta)| + |c_{i,l,m}(\vartheta)|\Big ) & \le K\, \rho^i \ .
\end{align*}
From \eqref{FZARMA}, this implies that there exists some other absolutely summable and deterministic sequences  $(d_i(\vartheta))_{i\ge 0}$,  $(d_{i,l}(\vartheta))_{i\ge 1}$  and $(d_{i,l,m}(\vartheta))_{i\ge 1}$ such that, almost surely,
\begin{align}
\label{ce1}
e_t(\vartheta)& =\sum_{i=0}^{\infty}d_i(\vartheta)e_{t-i}\ ,\  \frac{\partial e_t(\vartheta)}{\partial \vartheta_l} =\sum_{i=1}^{\infty}d_{i,l}(\vartheta)e_{t-i} \  \text{and}\   \frac{\partial^2e_t(\vartheta)}{\partial \vartheta_l\partial\vartheta_m} =\sum_{i=2}^{\infty}d_{i,l,m}(\vartheta)e_{t-i} \\
 \tilde e_t(\vartheta)& =\sum_{i=0}^{t-1}d_i(\vartheta)\tilde e_{t-i}\ ,\  \frac{\partial \tilde e_t(\vartheta)}{\partial \vartheta_l} =\sum_{i=1}^{t-1}d_{i,l}(\vartheta)\tilde e_{t-i} \ \text{and}\   \frac{\partial^2 \tilde e_t(\vartheta)}{\partial \vartheta_l\partial\vartheta_m} =\sum_{i=2}^{t-1}d_{i,l,m}(\vartheta)\tilde e_{t-i}.\nonumber
\end{align}
We have $d_0(\vartheta)=1$ and the three above sequences also satisfy
\begin{align}\label{rho2}
\sup_{\vartheta\in \Theta} \Big ( |d_i(\vartheta)| +|d_{i,l}(\vartheta)| + |d_{i,l,m}(\vartheta)|\Big ) & \le K\, \rho^i \ .
\end{align}
Finally, from the above estimates, we are able to deduce that for any $(l,m)\in\{ 1,...,k_0\}^2$
\begin{align}
\label{app}
\sup_{\vartheta\in\Theta} \left | e_t(\vartheta)-\tilde e_t(\vartheta) \right | & \xrightarrow[t\to\infty]{\text{a.s.}} 0 \ , \\
\label{app2}
\rho^t \sup_{\vartheta\in\Theta} \left | e_t(\vartheta) \right |&  \xrightarrow[t\to\infty]{\text{a.s.}} 0 \ .
\end{align}
Analogous estimates to \eqref{app} and \eqref{app2} are satisfied for first and second order derivatives of $e_t$ and $\tilde e_t$.

\subsection{Proof of Theorem \ref{sn1}}
The proof is divided in several steps.   
\subsubsection{Taylor's expansion of $\hat\Gamma_m$}
The aim of this step is to prove that
\begin{align}\label{tayl}
\sqrt{n}\ \hat\Gamma_m & = \frac{1}{\sqrt{n}}\sum_{t=1}^n\ \Lambda U_t + \mathrm{o}_{\mathbb P}(1).
\end{align}
Let $\sigma_{e0}^2$ the common variance of the univariate  noise process.
In this case, the Gaussian $\log$-quasi-likelihood can be written as  $$\tilde{\mathrm{\ell}}_n(\vartheta,\sigma_e^2)=\frac{-2}{n}\log \tilde{\mathrm{L}}_n(\vartheta,\sigma_e^2)=\log(2\pi)+\log\sigma_{e}^2
+\frac{2}{n}\sum_{t=1}^n\frac{\tilde{e}_t^2(\vartheta)}{\sigma_e^2}.$$
The least squares estimator (LSE in short) of $\vartheta$ is the value of that minimizes the last quantity as a function of $\vartheta$. Thus the LSE $\hat{\vartheta}_n$ satisfies almost surely
\begin{align}
Q_n(\hat{\vartheta}_n) = \min_{\vartheta\in \Theta} Q_n(\vartheta), \text{\ with\ }
Q_n(\vartheta)=\frac{1}{n}\sum_{t=1}^n \tilde{e}_t^2(\vartheta). \label{qn}
\end{align}
We denote $O_n(\vartheta)=\frac{1}{n}\sum_{t=1}^n e_t^2(\vartheta)$. By \eqref{app} and \eqref{app2}, the sequences $\sqrt{n}\frac{\partial}{\partial \vartheta} Q_n(\vartheta_0) $ and $\sqrt{n}\frac{\partial}{\partial \vartheta} O_n(\vartheta_0) $ have the same asymptotic distribution. More precisely we have
\begin{align}\label{conv_partial}
\sqrt{n}  \left (\frac{\partial}{\partial\vartheta}Q_n(\vartheta_0)-\frac{\partial}{\partial\vartheta}O_n(\vartheta_0)\right)  & = \mathrm{o}_\mathbb P (1)\ .
\end{align}
Then we use the Taylor expansion of the derivative of $Q_n$ which is defined in
\eqref{qn}. We shall need that   $\partial Q_n(\hat\vartheta_n)/\partial \vartheta=0$, which is true because $\hat\vartheta_n$ minimizes the function $\vartheta\mapsto Q_n(\vartheta)$.   Following the proof
of Theorem 2 in \cite{fz98}, we obtain that
\begin{align*}
\nonumber 0&=\sqrt{n}\,\frac{\partial
Q_n(\hat\vartheta_n)}{\partial\vartheta} =\sqrt{n}\,\frac{\partial
Q_n(\vartheta_0)}{\partial \vartheta}+\frac{\partial^2
Q_n(\vartheta^\sharp_{n,i,j})}{\partial \vartheta_i\partial
\vartheta_j}\sqrt{n}\left(\hat{\vartheta}_n-\vartheta_0\right),
\end{align*}
where the $\vartheta^\sharp_{n,i,j}$'s are between $\hat\vartheta_n$ and $\vartheta_0$, and consequently
\begin{align}
\label{taylorintheo2}0 &=\sqrt{n}\, \frac{\partial
O_n(\vartheta_0)}{\partial \vartheta}+\frac{\partial^2
O_n(\vartheta_0)}{\partial \vartheta\partial
\vartheta'}\sqrt{n}\left(\hat{\vartheta}_n-\vartheta_0\right)+\mathrm{o}_\mathbb P(1).
\end{align}
With $Y_t =-2e_t(\vartheta_0)\frac{\partial e_t(\vartheta_0) }{\partial \vartheta}$, we deduce from (\ref{taylorintheo2}) that
\begin{eqnarray*}
\hat{\vartheta}_n-\vartheta_0=-J^{-1}\frac{\partial O_n(\vartheta_0)}{\partial
\vartheta}+\mathrm{o}_\mathbb P(1)=J ^{-1}\frac{1}{n}\sum_{t=1}^n Y_t+\mathrm{o}_\mathbb P(1).
\end{eqnarray*}
Thanks to the results recalled in Section \ref{reminder}, we may use the arguments of Lemma A3 in \cite{fz00}. We have for $h\in\{1,\dots,m\}$
\begin{align}
\label{dl2}
\hat\Gamma_e(h)& = \Gamma_e(h) + \phi_h \, (\hat\vartheta_n-\vartheta_0)  +\mathrm{o}_\mathbb{P}(1),
\end{align}
where $ \phi_h:=\mathbb E \left ( \frac{\partial e_t(\vartheta_0)}{\partial \vartheta'}\, e_{t-h}(\vartheta_0)\right )\in\mathbb R^{p+q}$.
Thus we may rewrite \eqref{dl2} as
\begin{align*}
\hat\Gamma_e(h)
& = \frac{1}{n}\sum_{t=h+1}^{n}e_t\, e_{t-h} -(\phi_h \, J^{-1})\frac{1}n   \sum_{t=1}^n Y_t  +\mathrm{o}_\mathbb{P}(1).
\end{align*}
Now, we come back to the vector $\hat\Gamma_m=(\hat\Gamma_e(1),\dots,\hat\Gamma_e(m))'$. We remark that the matrix $\Phi_m$ defined by \eqref{Gamma} is the  matrix in $\mathbb R^{m\times(p+q)}$ whose line $h$ is $\phi_h$. So we have
\begin{align}
\label{dl4}
\sqrt{n}\ \hat\Gamma_m & = \frac{1}{\sqrt{n}}\sum_{t=1}^n\ \Lambda \left (   Y_t'{,} e_te_{t-1} {,}\dots{,} e_te_{t-m}
 \right )'+ \mathrm{o}_{\mathbb P}(1) .
\end{align}
Therefore the Taylor expansion \eqref{tayl} of $\hat\Gamma_m$ is proved. This ends our first step.

Now, it is clear that the asymptotic behaviour of $\hat\Gamma_m$ is related to the limit distribution of $U_t=\left ( Y_t' {,} e_te_{t-1} {,}\dots{,} e_te_{t-m} \right )'$. The next step deals with the asymptotic distribution of $\Lambda U_t$.

\subsubsection{Functional central limit theorem for $(\Lambda U_t)_{t\ge 1}$}\label{hh}
Our purpose is to prove that there exists a lower triangular matrix $\Psi$, with nonnegative diagonal entries, such that
\begin{equation}
\label{fclt}
\frac{1}{\sqrt{n}}\sum_{j=1}^{\lfloor nr\rfloor }\Lambda U_j
\xrightarrow[n\to\infty]{\mathbb D^m}
\Psi B_{m}(r)
\end{equation}
where $(B_m(r))_{r\ge 0}$ is a $m$-dimensional standard Brownian motion.

By \eqref{ce1} one rewrites $U_t$ as
\begin{align}\label{ut}
U_t & =  \left ( -2\sum_{i=1}^{\infty} d_{i,1}(\vartheta_0)e_t e_{t-i}\ , \ \cdots \ ,\ -2\sum_{i=1}^{\infty} d_{i,k_0}(\vartheta_0)e_t e_{t-i}\  {,}\ e_te_{t-1} \ {,}\ \dots{,} \ e_te_{t-m}
 \right )'
\end{align}
and thus it has zero expectation with values in $\mathbb R^{k_0+m}$. In order to apply the function central limit theorem for strongly mixing process,
we need to identify the asymptotic covariance matrix in the classical central limit theorem for the sequence $(U_t)_{t\ge 1}$.
It is proved in \cite{fz98} that
\begin{align}\label{convvv}
\frac{1}{\sqrt{n}}\sum_{t=1}^{n}U_t \xrightarrow[n\to\infty]{\mathrm{d}}\mathcal{N}(0,\Xi)
\end{align}
where
\begin{equation}
\label{spectraldensityatzero}
\Xi:=2\pi f_U(0)=\sum_{h=-\infty}^{+\infty}\cov(U_t,U_{t-h})=\sum_{h=-\infty}^{+\infty}\e(U_tU'_{t-h}),
\end{equation}
where $f_U(0)$ is the spectral density of the
stationary process $(U_t)_{t\in\mathbb Z}$ evaluated at frequency 0 (see
for example \cite{broc-d}). The main issue is the existence of the sum of the right-hand side of
(\ref{spectraldensityatzero}). For that sake, one has to introduce for any integer $k$, the random variables
\begin{align*}
U_t^k =\left ( -2\sum_{i=1}^{k} d_{i,1}(\vartheta_0)e_t e_{t-i}\ , \ \cdots \ ,\ -2\sum_{i=1}^{k} d_{i,k_0}(\vartheta_0)e_t e_{t-i}\  {,}\ e_te_{t-1} \ {,}\ \dots{,} \ e_te_{t-m}
 \right )' .
\end{align*}
Since $U^k$ depends on a finite number of values of the noise-process $e$, it also satisfies a mixing property of the form \eqref{mixing}. Based on the Davydov inequality  (see \cite{davy}), the arguments developed in the Lemma A.1 in \cite{frz} (see also \cite{fz98}) imply that
\begin{align}\label{convvvk}
\frac{1}{\sqrt{n}}\sum_{t=1}^{n}U^k_t \xrightarrow[n\to\infty]{\mathrm{d}}\mathcal{N}(0,\Xi_k)
\end{align}
where
\begin{equation*}
\Xi_k:=2\pi f_{U^k}
(0)=\sum_{h=-\infty}^{+\infty}\cov(U^k_t,U^k_{t-h})=
\sum_{h=-\infty}^{+\infty}\e (U_t^k {{U}^{k}_{t-h}}' ) .
\end{equation*}
Using this truncation procedure and the Davydov inequality (see \cite{davy}), the arguments developed in the Lemma A.1 in \cite{frz} (see also \cite{fz98}) imply that \eqref{convvv} holds. Moreover we have that $\lim_{k\to\infty}\Xi_k =\Xi$.

Since the matrix $\Xi$ is positive definite, it can be factorized as $\Xi=\Upsilon\Upsilon'$ where the $(k_0+m)\times (k_0+m)$ lower triangular matrix $\Upsilon$ has nonnegative diagonal entries. Therefore, we have
$$\frac{1}{\sqrt{n}}\sum_{t=1}^{n}\Lambda U_t \xrightarrow[n\to\infty]{\mathrm{d}}\mathcal{N}(0,\Lambda\Xi\Lambda'),$$
and the new variance matrix can also been factorized as $\Lambda\Xi\Lambda'= (\Lambda\Upsilon)(\Lambda\Upsilon)':=\Psi\Psi'$.  Thus,
${n}^{-1/2}\sum_{t=1}^{n}\Psi^{-1}\Lambda U_t \xrightarrow[n\to\infty]{\mathrm{d}} \mathcal{N}\left(0,I_{m}\right)$ where $I_{m}$ is the identity matrix of order $m$.
The above arguments also apply to matrix $\Xi_k$ with some matrix $\Psi_k$ which is defined analogously as $\Psi$. Consequently,
$$\frac{1}{\sqrt{n}}\sum_{t=1}^{n}\Lambda U_t^k \xrightarrow[n\to\infty]{\mathrm{d}}\mathcal{N}(0,\Lambda\Xi_k\Lambda'), $$
and we also have ${n}^{-1/2}\sum_{t=1}^{n}\Psi_k^{-1}\Lambda U^k_t \xrightarrow[n\to\infty]{\mathrm{d}} \mathcal{N}\left(0,I_{m}\right)$

Now we are able to apply the functional central limit theorem (see \cite{herr}). We have for any $r\in (0{,}1)$,
\begin{align*}
\frac{1}{\sqrt{n}}\sum_{j=1}^{\lfloor nr\rfloor }\Psi_k^{-1}\Lambda U_j^k
\xrightarrow[n\to\infty]{\mathbb D^m}
B_{m}(r) .
\end{align*}
We write $$\Psi^{-1}\Lambda U_j^k = \Big (\Psi^{-1}-\Psi_k^{-1}\Big)\Lambda U_j^k  + \Lambda \Psi_k^{-1}\Lambda U_j^k$$
and we obtain that
\begin{align*}
\frac{1}{\sqrt{n}}\sum_{j=1}^{\lfloor nr\rfloor }\Psi^{-1}\Lambda U_j^k
\xrightarrow[n\to\infty]{\mathbb D^m}
B_{m}(r) .
\end{align*}
In order to conclude that \eqref{fclt} is true, it remains to observe that, uniformly with respect to $n$,
\begin{align}
\label{ll}
Z^k_n(r):=\frac{1}{\sqrt{n}}\sum_{j=1}^{\lfloor nr\rfloor }\Psi^{-1}\Lambda V_j^k
\xrightarrow[k\to\infty]{\mathbb D^m}
0 ,
\end{align}
where
$$ V_t^k = \left ( -2\sum_{i=k+1}^{\infty} d_{i,1}(\vartheta_0)e_t e_{t-i}\ , \ \cdots \ ,\ -2\sum_{i=k+1}^{\infty} d_{i,k_0}(\vartheta_0)e_t e_{t-i}\  {,}\ e_te_{t-1} \ {,}\ \dots{,} \ e_te_{t-m}
 \right )' \ .$$
By Lemma 4 in \cite{fz98},
$$
\sup_{n} \mathrm{Var} \left ( \frac{1}{\sqrt{n}} \sum_{j=1}^n V_j^k \right ) \xrightarrow[k\to\infty]{} 0
$$
and since $\lfloor nr\rfloor\le n$,
$$\sup_{0\le r\le 1}\sup_n \left (  | Z_n^k(r) | \right ) \xrightarrow[k\to\infty]{} 0 . $$
Thus \eqref{ll} is true and the proof of \eqref{fclt} is achieved.
\subsubsection{Limit theorem}\label{justif}
We prove Theorem \ref{sn1}. We follow the arguments developed in Sections 2 and 3 in \cite{lobato}. The main difference is that we shall work with the  sequence $(\Lambda U_t)_{t\ge 1}$ instead of the sequence $\big ( ( e_te_{t-1},\dots,e_te_{t-m})'\big )_{t\ge 1}$.
The previous step ensures us that  Assumption 1 in \cite{lobato} is satisfied for the sequence $(\Lambda U_t)_{t\ge 1}$.
Since $C_{m}=(1/n^2)\sum_{t=1}^{n}S_{t}S'_{t} $, the continuous mapping theorem on the Skorokhod space implies that
\begin{align}
\label{conv}
C_m   \xrightarrow[n\to\infty]{\mathrm{d}} \Psi{V}_{m}\Psi',
\end{align}
where the random variable $V_m$ is defined in \eqref{vm}.
We assume for the moment that the matrix $C_m$ is invertible (this will be stated and proved in Lemma \ref{inversible} in Subsection \ref{cmmm} at the end of this appendix).

Since by \eqref{dl4}, $\sqrt{n} \hat\Gamma_m = n^{-1/2}\sum_{t=1}^n\Lambda U_t+\mathrm{o}_\mathbb P(1)$, we use (\ref{fclt}) and (\ref{conv}) in order to obtain
\begin{align*}
n\hat\Gamma_m'C_m^{-1}\hat\Gamma_m & = \frac{1}{n}\sum_{t=1}^{n}\left((\Lambda U_t)' C_{m}^{-1}(\Lambda U_t)\right) \\
 &\xrightarrow[n\to\infty]{\mathrm{d}} \left(\Psi{B}_{m}(1)\right)'\left(\Psi{V}_{m}\Psi'\right)^{-1}\left(\Psi{B}_{m}(1)\right)
={B}'_{m}(1){V}_{m}^{-1}{B}_{m}(1),
\end{align*}
and we recognize the random variable $\mathcal U_m$ defined in \eqref{um}. Consequently we have proved \eqref{conv1}. The property \eqref{conv2} is straightforward since $\hat\rho(h)=\hat\Gamma_e(h)/\sigma_{e0}^2$ for $h=1,\dots,m$. The proof of Theorem \ref{sn1} is then complete.
\subsection{Proof of Theorem \ref{sn2}}
We write $\hat C_m=C_m+ \Upsilon_n$ where $ \Upsilon_n  ={n}^{-2} \sum_{t=1}^n \big (S_tS_t' - \hat S_t\hat S_t' \big)$.
There are three kinds of entries in the matrix $\Upsilon_n$. The first one is a sum composed of
$$\upsilon_t^{k,k'} = e_t^2(\vartheta_0)e_{t-k}(\vartheta_0)e_{t-k'}(\vartheta_0)  - \tilde{e}_t^2(\hat\vartheta_n)\tilde{e}_{t-k}(\hat\vartheta_n)\tilde{e}_{t-k'}(\hat\vartheta_n)  $$
for $(k,k')\in\{1,\dots,m\}^2$.
By \eqref{rho2} and the consistency of $\hat\vartheta_n$, we have $\upsilon_t^{k,k'} = \mathrm{o}_{\mathbb P}(1)$ almost surely.
The two last  kinds of entries of $\Upsilon_n$ come from the following quantities for $i,j\in\{1,...,k_0\}$ and $k\in\{1,...,m\}$
\begin{align*}
\tilde\upsilon_t^{k,i} & = e_t^2(\vartheta_0)e_{t-k}(\vartheta_0) \frac{\partial e_{t}(\vartheta_0)}{\partial \vartheta_i}-  \tilde{e}_t^2(\hat\vartheta_n)\tilde{e}_{t-k}(\hat\vartheta_n)\frac{\partial \tilde{e}_{t}(\hat\vartheta_n)}{\partial \vartheta_i},
\\
\bar\upsilon^{i,j}_t & = e^2_{t}(\vartheta_0)\frac{\partial e_{t}(\vartheta_0)}{\partial \vartheta_i} \frac{\partial e_{t}(\vartheta_0)}{\partial \vartheta_j}-  \tilde{e}^2_{t}(\hat\vartheta_n)\frac{\partial \tilde{e}_{t}(\hat\vartheta_n)}{\partial \vartheta_i}
\frac{\partial \tilde{e}_{t}(\hat\vartheta_n)}{\partial \vartheta_j}
\end{align*}
and they also satisfy $\tilde\upsilon_t^{k,i} + \bar{\upsilon}_t^{i,j}= \mathrm{o}_{\mathbb P}(1)$ almost surely.
Consequently, $\Upsilon_n= \mathrm{o}_{\mathbb P}(1)$ almost surely as $n$ goes to infinity.
Thus one may find a matrix $\Upsilon^\ast_n$, that tends to the null matrix almost surely, such  that
\begin{align*}
n\, \hat\Gamma_m'\hat C_m^{-1}\hat\Gamma_m &  = n\, \hat\Gamma_m' (C_m + \Upsilon_n)^{-1}\hat\Gamma_m   = n\, \hat\Gamma_m' C_m^{-1}\hat\Gamma_m + n\, \hat\Gamma_m'\Upsilon_n^\ast \hat\Gamma_m \ .
\end{align*}
Thanks to the arguments developed in the proof of Theorem \ref{sn1}, $n \hat\Gamma_m'\hat\Gamma_m $ converges in distribution. So $n \hat\Gamma_m'\Upsilon_n^\ast \hat\Gamma_m$ tends to zero in distribution, hence in probability.  Then $n \hat\Gamma_m'\hat C_m^{-1}\hat\Gamma_m $ and  $n \hat\Gamma_m' C_m^{-1}\hat\Gamma_m $ have the same limit in distribution and the result is proved.
\subsection{Invertibility of the normalization matrix}\label{cmmm}
We prove in this short subsection why the matrix $C_m$ is invertible.
\begin{lemme}\label{inversible}
Under the assumption of Theorem \ref{sn1}, the matrix $C_m$ is almost surely non singular.
 \end{lemme}
 \begin{proof}
We write the matrix $\Lambda$ as
$$ \Lambda =  \left(
  \begin{array}{ccc}
    (\lambda_{ij})_{1\le i\le m{;}1\le j\le k_0} &|& I_m \\
  \end{array}
\right)  $$
we have for $t=1,...,n$:
$$ S_t=\left(
  \begin{array}{c}
    S_t^1 \\
    \vdots \\
    S_t^m
  \end{array}
\right) =  \left(
  \begin{array}{c}
    \left ( \sum_{j=1}^t  (\mbox{$\sum_{l=1}^{k_0}\lambda_{1l} Y^l_j$})  + e_je_{j-1} \right ) - t\Gamma_e(1)\\
    \vdots \\
     \left (   \sum_{j=1}^t (\mbox{$\sum_{l=1}^{k_0}\lambda_{ml} Y^l_j$})  + e_je_{j-m} \right ) - t\Gamma_e(m)\\
  \end{array}
\right), $$
with $ Y^l_j = e_{j}(\vartheta_0)\frac{\partial e_{j}(\vartheta_0)}{\partial \vartheta_l}$.
We remark that
\begin{align}\label{st}
S_{t+1}^i& =S_t^i + (\mbox{$\sum_{l=1}^{k_0}\lambda_{il} Y^l_{t+1}$}) +  e_{t+1}e_{t+1-i}   -\Gamma_e(i)\ .
\end{align}
If the matrix $C_m$ is not invertible, there exits some real constants $c_1,...,c_m$, not all equal to zero, such that for any $t=1,...,n$ we have
$ \sum_{i=1}^m {c_i}S_t^i=0$. By \eqref{st}, it would imply that
$$ \sum_{i=1}^m c_i \left (  (\mbox{$\sum_{l=1}^{k_0}\lambda_{il} Y^l_{t}$})+e_{t}e_{t-i} \right ) = \sum_{i=1}^m c_i\Gamma_e(i)\ .$$
By the ergodic Theorem, $\sum_{i=1}^m c_i\Gamma_e(i)\to 0$ almost-surely as $n$ goes to infinity. Consequently, for any $t\ge 1$, $ \sum_{i=1}^m c_i  (\mbox{$\sum_{l=1}^{k_0}\lambda_{il} Y^l_{t}$})+c_i e_{t}e_{t-i} =0$. Using \eqref{ut} yields that
$$    e_t \left ( \sum_{k=1}^\infty \big (\mbox{$\sum_{i=1}^m  c_i \sum_{l=1}^{k_0}\lambda_{il}d_{k,l}$}\big )  e_{t-k} + \sum_{k=1}^m  c_k e_{t-k}  \right ) = 0\ ,$$
or equivalently
$$    e_t \left ( \sum_{k=1}^m \big [\mbox{$\sum_{i=1}^m  c_i \sum_{l=1}^{k_0}\lambda_{il}d_{k,l}$} +c_k \big ]e_{t-k} + \sum_{k=m+1}^\infty \big [ \mbox{$\sum_{i=1}^m  c_i \sum_{l=1}^{k_0}\lambda_{il}d_{k,l}$}\big ] e_{t-k}  \right ) = 0\ .$$

Therefore, there exists a sequence $(\alpha_k)_{k\ge 1}$, with $\sum_{k\ge 1} |\alpha_k| <\infty $, such that
$e_t  \sum_{k\ge 1}\alpha_k e_{t-k} = 0$. Thanks to Assumption {\bf (A4)}, $e_t$ has a positive density in some neighborhood of zero and then $e_t\neq 0$ almost-surely. So we would have
$\sum_{k\ge 1}\alpha_k e_{t-k} = 0$. Since the variance of the linear innovation process is not equal to zero, all the coefficients $\alpha_k$ vanish. Then
$$ 0=\alpha_k = \left \{
\begin{array}{ll}
     \mbox{$\sum_{i=1}^m  c_i \sum_{l=1}^{k_0}\lambda_{il}d_{k,l}$} +c_k &\quad \mbox{if $1\le k\le m$}\\
     \mbox{$\sum_{i=1}^m  c_i \sum_{l=1}^{k_0}\lambda_{il}d_{k,l}$}  &\quad \mbox{if $ k> m$}\ .\\
  \end{array}
\right . $$
Then we would have $c_1=...=c_m=0$ which is impossible. Thus we have a contradiction and the matrix $C_m$ is non singular.
\end{proof}
\begin{remi}\label{rem-inv}
The property that $e_t$ has a positive density in some neighborhood of zero is necessary to ensure the invertibility of $C_m$. Indeed, if we choose an independent noise sequence such that the law of $e_t$ is given by
$\frac{1}2\delta_0 +\frac{1}2 \mathcal N(0{,}1)$, then the event $E = \cap_{i=1}^n \{e_i=0\}$ has a positive probability (actually equals to $2^{-n}$). On this event, the matrix $C_m$ is clearly singular because is it equal to the null matrix. It is worth to notice that this counterexample may apply to the work of Lobato (see \cite{lobato}) and it seems that this hypotheses is missing in this work. Of course, the probability of $E$ is very small ($\mathbb P(E) < 10^{-60} $ for $n=200$ which is the lower $n$ chosen in our simulation) and thus it is quite normal that we do obtain any singular matrix $C_m$ in our numerical studies.
\end{remi}

\newpage
\section{Empirical size for small length series}
\begin{table}[h!]
 \caption{\small{Empirical size (in \%) of the modified  and standard  versions
 of the LB and BP tests in the case of ARMA$(1,1)$. 
The nominal asymptotic level of the tests is $\alpha=5\%$.
The number of replications is $N=1000$. }}
\begin{center}
\begin{tabular}{ccc ccc ccc}
\hline\hline \\
Model& Length $n$ & Lag $m$ & $\mathrm{{LB}}_{\textsc{sn}}$&$\mathrm{BP}_{\textsc{sn}}$&$\mathrm{{LB}}_{\textsc{frz}}$&$\mathrm{BP}_{\textsc{frz}}$&
$\mathrm{{LB}}_{\textsc{s}}$&$\mathrm{BP}_{\textsc{s}}$
\vspace*{0.2cm}\\\hline
&& $1$&5.9 &5.9 &4.9 &4.5&n.a. &n.a.\\
&& $2$&4.6  & 4.4  & 4.5 & 4.3&n.a. &n.a.\\
&& $3$&4.9 & 4.5 & 4.3 & 4.1& 11.8& 11.3\\
&& $4$&4.0& 3.8& 5.1& 4.6& 9.4 &8.9\\
I &$n=200$& $5$&4.3& 4.2& 4.5& 4.1 &8.5 &8.0\\
 &&$6$&4.2& 3.5 &4.9 &4.6 &7.2 &6.5\\
 &&$8$&5.0& 4.4& 4.4 &3.9& 6.2& 5.2\\
 &&$10$&4.1& 3.7& 4.8& 3.7& 6.1& 5.1\\
 &&$12$&3.5& 2.8& 4.2& 3.1 &5.9 &4.7\\
\hline
&& $1$&5.6& 5.6& 5.8& 5.6&n.a. &n.a.\\
&& $2$&4.6  & 4.2  & 4.4  & 4.2&n.a. &n.a.\\
&& $3$&2.0 & 2.0 & 3.6 & 3.2 &19.4& 19.0\\
II &$n=200$& $4$&2.0 & 2.0 & 3.4 & 3.2& 15.2& 14.6\\
 && $5$&1.8 & 1.8 & 2.4 & 2.0 &16.0& 14.8\\
 &&$6$&1.6 & 1.6 & 2.8 & 2.8& 15.0& 14.2\\
 &&$8$&2.8 & 2.2&  3.0&  2.4& 12.0& 10.4\\
\hline
&& $1$&3.4& 3.4& 3.1& 2.9&n.a. &n.a.\\
&& $2$&3.2  & 2.9 &  2.5  & 2.5&n.a. &n.a.\\
&& $3$&1.7 & 1.5 & 2.2 & 2.1& 18.5& 17.6\\
III &$n=200$& $4$&1.5 & 1.2 & 2.4 & 2.2 &13.5 &12.6\\
&& $5$&1.4 & 1.3 & 1.9 & 1.4& 10.5 & 9.4\\
 &&$6$&1.2 &1.0& 1.6& 1.3 &8.7& 8.0\\
 &&$8$&0.9& 0.8& 1.5& 1.0& 7.7& 7.2\\

\hline
&& $1$&3.0& 3.0 &3.3& 3.0&n.a. &n.a.\\
&& $2$& 3.7  & 3.7  & 2.7  &2.7&n.a. &n.a.\\
&& $3$&1.7& 1.3  &1.7  &1.7 &17.7& 17.0\\
IV &$n=200$& $4$&0.7& 0.6& 1.0& 1.0& 12.0& 11.3\\
&& $5$&0.4& 0.3 & 0.9 & 0.8& 11.4& 11.2\\
 &&$6$&0.3& 0.3 & 1.6  &1.3& 9.2 & 8.6\\
\hline
&& $1$&5.3& 5.3& 4.3& 4.2&n.a. &n.a.\\
&& $2$&4.0  & 4.0  & 5.1  & 5.0&n.a. &n.a.\\
&& $3$&3.6 & 3.3 & 3.8 & 3.5& 13.3 &12.8\\
V &$n=200$& $4$&4.9& 4.7& 4.0& 4.0& 9.0& 8.5\\
&& $5$&4.0& 3.8 &4.4 &4.0& 8.1& 7.3\\
 &&$6$&4.6& 4.4& 4.7& 4.3& 7.7& 7.2\\
 &&$8$&3.4& 2.9& 4.7& 4.1& 6.7& 5.6\\
 &&$10$&4.9 &3.9 &4.7 &3.9 &6.5& 5.2\\
 &&$12$&4.1& 3.2& 5.5 &3.8 &7.2 &5.5\\
\hline\hline
\\
\multicolumn{9}{l}{I: Strong ARMA(1,1) model
(\ref{ARMA01MonteCarlo})-(\ref{bruitARCH}) with $(\alpha_1,\beta_1)=(0,0)$}\\
\multicolumn{9}{l}{II: Weak ARMA(1,1) model
(\ref{ARMA01MonteCarlo})-(\ref{bruitARCH}) with $(\alpha_1,\beta_1)=(0.1,0.85)$.}\\
\multicolumn{9}{l}{III: Weak ARMA(1,1) model
(\ref{ARMA01MonteCarlo})-(\ref{PT}).}\\
\multicolumn{9}{l}{IV: Weak ARMA(1,1) model
(\ref{ARMA01MonteCarlo})-(\ref{PTcarre}).}\\
\multicolumn{9}{l}{V: Weak ARMA(1,1) model
(\ref{ARMA01MonteCarlo})-(\ref{RT}).}
\\
\end{tabular}
\end{center}
\label{tabARMAreferee}
\end{table}

\ \newpage
\begin{table}[h!]
 \caption{\small{Empirical size (in \%) of the modified  and standard  versions
 of the LB and BP tests in the case of VARMA$(1,1)$. 
The nominal asymptotic level of the tests is $\alpha=5\%$.
The number of replications is $N=1000$. }}
\begin{center}
\begin{tabular}{ccc ccc ccc}
\hline\hline \\
Model& Length $n$ & Lag $m$ & $\mathrm{{{LB}}}_{\textsc{sn}}$&$\mathrm{{{BP}}}_{\textsc{sn}}$&$\mathrm{{{LB}}}_{\textsc{bm}}$&
$\mathrm{{{BP}}}_{\textsc{bm}}$&$\mathrm{{{LB}}}_{\textsc{s}}$&$\mathrm{{{BP}}}_{\textsc{s}}$
\vspace*{0.2cm}\\\hline
&& $1$&3.8& 3.8& 3.6& 3.4&n.a.&n.a.\\
&& $2$&3.7  & 3.9  & 3.6  & 3.5&n.a.&n.a.\\
I &$n=200$& $3$&3.5 & 3.7 & 2.8 & 2.7& 16.7& 15.8\\
 &&$4$&3.4& 3.6& 2.7& 2.4& 9.3& 7.6\\
 &&$5$&3.2& 3.5& 2.4 &2.1 &6.9 &5.9
\\
\hline
&& $1$&2.9& 2.9& 5.9 &5.7&n.a.&n.a.\\
&& $2$&1.9   &1.9  & 4.7  & 4.6&n.a.&n.a.\\
II &$n=200$& $3$&0.6 & 0.7 & 2.9&  2.9& 32.5& 30.8\\
 &&$4$&0.7 & 0.9&  4.1&  4.1& 22.3& 21.2\\
 &&$5$&1.2 & 1.2 & 5.1&  5.0& 17.4& 16.1
\\
\hline
&& $1$&1.9 &1.9& 8.3& 8.2&n.a.&n.a.\\
&& $2$&0.4  & 0.5&  12.5 & 12.2&n.a.&n.a.\\
III &$n=200$& $3$&0.0 & 0.0& 16.2 &15.9& 41.6& 40.7\\
 &&$4$&0.0&  0.0& 22.3& 22.2& 31.6& 29.9\\
 &&$5$&0.0 & 0.0 &30.4& 29.9& 24.5 &22.4\\
\hline
&& $1$&0.3 & 0.3& 11.7 &11.5&n.a.&n.a.\\
&& $2$&0.1 & 0.2& 14.8& 14.4&n.a.&n.a.\\
IV &$n=200$& $3$&0.2 & 0.2& 18.6& 18.5& 28.4& 27.9\\
 &&$4$&0.4 & 0.4& 23.5& 23.2& 21.7& 21.3\\
 &&$5$&0.5 & 0.5& 33.7& 33.3& 18.7& 18.0\\
\hline
&& $1$&3.8& 3.8& 5.1 &4.7&n.a.&n.a.\\
&& $2$&3.8  & 4.1  & 4.3 &  3.7&n.a.&n.a.\\
V &$n=200$& $3$&3.4&  3.6 & 3.6 & 3.3 &16.5& 15.3
\\
 &&$4$&3.4& 3.7& 0.9& 0.6& 8.9& 8.0\\
 &&$5$&3.1& 3.2& 0.6& 0.5& 6.0& 5.2\\

\hline\hline \\
\multicolumn{9}{l}{I: Strong VARMA$(1,1)$ model
(\ref{VARMA11MonteCarlo})-(\ref{bruitfort}).}\\
\multicolumn{9}{l}{II: Weak VARMA$(1,1)$ model
(\ref{VARMA11MonteCarlo})-(\ref{ARCH1}).}\\
\multicolumn{9}{l}{III: Weak VARMA$(1,1)$ model
(\ref{VARMA11MonteCarlo})-(\ref{multiPT}).}\\
 \multicolumn{9}{l}{IV: Weak VARMA$(1,1)$ model
(\ref{VARMA11MonteCarlo})-(\ref{multiPTcarre}).}\\
\multicolumn{9}{l}{V: Weak VARMA$(1,1)$ model
(\ref{VARMA11MonteCarlo})-(\ref{multiRT}).}\\
\end{tabular}
\end{center}
\label{tab1VARMAreferee}
\end{table}
\newpage
\section{The Standard \& Poor's 500 index as an illustrative example}\label{otherill}
As we did for the CAC40, we also consider an application to the daily $\log$ returns (also simply called the returns) of the S\&P500\footnote{
The Standard \& Poor's 500 index. }.
The observations cover the period from March 1, 1979 to December 31, 2001.
The length of the series is $n=5808$. The data can be downloaded
from the website Yahoo Finance: http://fr.finance.yahoo.com/.
Figure~\ref{figsp500}
shows that the S\&P500 index series are generally close to a random
walk without intercept and that the returns are generally compatible with the second-order stationarity
assumption.
\begin{figure}[h!]
\centering
\includegraphics[width=0.9\textwidth]{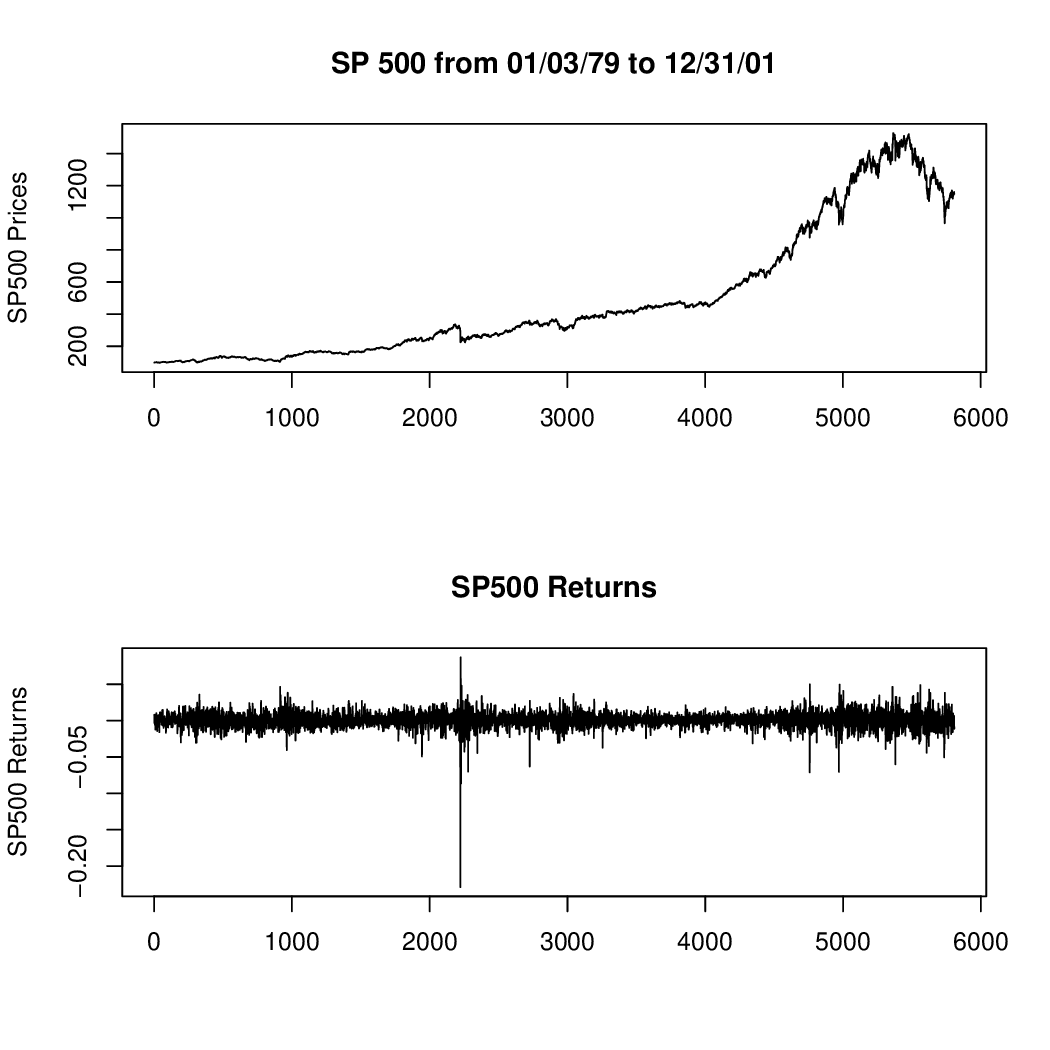}
\caption{\label{figsp500} {\footnotesize Closing prices and returns of the S\&P 500 index from January 3, 1979 to
December 31, 2001 (5807 observations). }}
\end{figure}


First, we apply portmanteau tests
for checking the hypothesis that the S\&P500 returns constitute a white noise. Table \ref{sp500BB} displays the statistics of the standard and modified BP tests.
Since the $p$-values of the
standard test are very small, the white noise hypothesis is rejected at the nominal level $\alpha=1\%$.
This is not surprising because the standard tests required the iid assumption and, in particular in view of
the so-called volatility clustering, it is well known that the strong white noise model is not adequate for these series.
By contrast, the  white noise hypothesis is not rejected by the modified tests,
since for the modified tests, the statistic  is not larger than the critical values (see Table 1 in \cite{lobato}). This is also in accordance with other works devoted to the analysis of stock-market returns \cite{lobatoNS2001}.
To summarize, the outputs of Table~\ref{sp500BB} are in accordance with the
common belief that these series are not strong white noises, but could be weak white noises.

\begin{table}[h!]
 \caption{\small{Modified and standard versions
 of portmanteau tests to check the null hypothesis that the S\&P 500 returns is a white noise. }}
\begin{center}
{\small
\begin{tabular}{c ccc cccc c}
\hline\hline
Lag $m$& &2&3&4 &5&10&18& 24\\
\hline$\hat\rho(m)$&& -0.03894& -0.03278& -0.01951& 0.00802&0.00408&-0.01833&0.01885
\\
$\mathrm{LB}_{\textsc{sn}}$&&62.8145&280.930&326.492&529.533&1425.52&1723.81&2323.05
\\
$\mathrm{BP}_{\textsc{sn}}$&&62.7759&280.712&326.257&529.140&1424.20&1722.54&2321.68
\\
\cline{3-9}  $\mathrm{LB}_{\textsc{frz}}$&&16.7123& 22.9587& 25.1716&25.5460&28.3375&38.4913&46.6186
\\
$\mathrm{BP}_{\textsc{frz}}$&&16.7021&22.9432&25.1538&25.5278&28.3145&38.4420&46.5352
\\
\cline{3-9}
$\mbox{p}_{\textsc{sn}}^{\textsc{lb}}$&&0.12053&0.01302&0.02585&0.01319&0.01139&0.26024&0.48544
\\
$\mbox{p}_{\textsc{sn}}^{\textsc{bp}}$&&0.12064&0.01306&0.02593&0.01323&0.01142&0.26077&0.48610
\\
\cline{3-9}  $\mbox{p}_{\textsc{frz}}^{\textsc{lb}}$&&0.03527&0.01917&0.02594&0.16036&0.25624&0.30714&0.34018
\\
$\mbox{p}_{\textsc{frz}}^{\textsc{bp}}$&&0.03533&0.01922&0.02601&0.16054&0.25657&0.30794&0.34149
\\
\cline{3-9}  $\mbox{p}_{\textsc{S}}^{\textsc{lb}}$&&0.00023&0.00004&0.00004&0.00010&0.00159&0.00333&0.00373
\\
$\mbox{p}_{\textsc{S}}^{\textsc{bp}}$&& 0.00024&0.00004&0.00004&0.00011&0.00160&0.00338&0.00381
\\
\hline\hline
\end{tabular}
}
\end{center}
\label{sp500BB}
\end{table}

Next, turning to the dynamics of the squared returns, we fit an ARMA$(1,1)$ model to the squares of the S\&P500 returns.
Denoting by $(X_t)$ the mean corrected
series of the squared returns, we obtain the model
\begin{eqnarray*}
X_{t}=0.83123X_{t-1}+\epsilon_{t}-0.72676\epsilon_{t-1},
\text{ where }\mbox{Var}(\epsilon_{t})=52.62582\times10^{-8}.
\end{eqnarray*}
Table \ref{sp500ARMA11} displays the statistics of the standard and modified BP tests.
From Table \ref{sp500ARMA11}, we draw the same conclusion, on the squares of the previous daily returns,
that the strong ARMA$(1,1)$ model is
rejected, but a weak ARMA$(1,1)$ model is not rejected. Note that the first and second-order structures we found for the S\&P500 returns, namely
a weak white noise for the returns and a weak ARMA$(1,1)$ model for the squares of the returns, are
compatible with a GARCH$(1,1)$ model.

%
%
\begin{table}[h!]
 \caption{\small{Modified and standard versions
 of portmanteau tests to check the null hypothesis that the S\&P 500 squared returns follow an ARMA$(1,1)$ model. }}
\begin{center}
{\small
\begin{tabular}{c ccc cccc c}
\hline\hline
Lag $m$& &1&2&3 &4&5&6& 7\\
\hline$\hat\rho(m)$&& -0.01978&  0.04776&-0.01510& -0.05975&  0.08872&-0.02572& -0.03186
\\
$\mathrm{LB}_{\textsc{sn}}$&&7.27253&11.5733&94.7095&117.522&120.184&175.809&195.675
\\
$\mathrm{BP}_{\textsc{sn}}$&& 7.26878&11.5731&94.6567&117.370&120.028&175.496&195.237
\\
\cline{3-9}
$\mathrm{LB}_{\textsc{frz}}$&&2.27281&15.5306&16.8568&37.6101&83.3760&87.2225&93.1247
\\
$\mathrm{BP}_{\textsc{frz}}$&&2.27164&15.5203&16.8454&37.5772&83.2880&87.1292&93.0222
\\
\cline{3-9}
$\mbox{p}_{\textsc{sn}}^{\textsc{lb}}$&&0.34481&0.57434&0.16482&0.24529&0.40148&0.38794&0.48841
\\
$\mbox{p}_{\textsc{sn}}^{\textsc{bp}}$&&0.34495&0.57434&0.16495&0.24573&0.40205&0.38879&0.48967
\\
\cline{3-9}  $\mbox{p}_{\textsc{frz}}^{\textsc{lb}}$&&0.36613&0.34701&0.32849&0.32611&0.32051&0.31997&0.74110
\\
$\mbox{p}_{\textsc{frz}}^{\textsc{bp}}$&&0.36608&0.34719&0.32865&0.32632&0.32076&0.32023&0.74123
\\
\cline{3-9}  $\mbox{p}_{\textsc{S}}^{\textsc{lb}}$&&n.a.&n.a.&0.00004&0.00000&0.00000&0.00000&0.00000
\\
$\mbox{p}_{\textsc{S}}^{\textsc{bp}}$&&n.a.&n.a.&0.00004&0.00000&0.00000&0.00000&0.00000
\\
\hline
Lag $m$& &8&9&10 &12&18&20& 24\\
\hline$\hat\rho(m)$&&0.01686&  0.00422& -0.01346&-0.00686&0.00946&-0.00175&0.01174
\\
$\mathrm{LB}_{\textsc{sn}}$&&204.217&271.635&298.802&372.828&1072.04&1403.02&1451.54
\\
$\mathrm{BP}_{\textsc{sn}}$&&203.764&270.793&297.731&371.484&1069.14&1397.58&1445.80
\\
\cline{3-9}
$\mathrm{LB}_{\textsc{frz}}$&&94.7779&94.8815&95.9359&96.3334&97.2783&98.2907&99.4676
\\
$\mathrm{BP}_{\textsc{frz}}$&& 94.6726&94.7760&95.8282&96.2249&97.1667&98.1754&99.3472
\\
\cline{3-9}
$\mbox{p}_{\textsc{sn}}^{\textsc{lb}}$&&0.61390&0.59294&0.66719&0.76105&0.64801&0.60942&0.85991
\\
$\mbox{p}_{\textsc{sn}}^{\textsc{bp}}$&&0.61512&0.59480&0.66914&0.76285&0.65004&0.61268&0.86192
\\
\cline{3-9}  $\mbox{p}_{\textsc{frz}}^{\textsc{lb}}$&&0.57299&0.30637&0.51343&0.57000&0.40044&0.44177&0.32510
\\
$\mbox{p}_{\textsc{frz}}^{\textsc{bp}}$&& 0.57321&0.30664&0.51367&0.57022&0.40072&0.44205&0.32541
\\
\cline{3-9}  $\mbox{p}_{\textsc{S}}^{\textsc{lb}}$&&0.00000&0.00000&0.00000&0.00000&0.00000&0.00000&0.00000
\\
$\mbox{p}_{\textsc{s}}^{\textsc{bp}}$&&0.00000&0.00000&0.00000&0.00000&0.00000&0.00000&0.00000
\\
\hline\hline
\end{tabular}
}
\end{center}
\label{sp500ARMA11}
\end{table}
\newpage
\section{Empirical power using asymptotic critical values}
In Section \ref{emp-power}, we investigate the size  adjusted power which is the rate at which the null hypothesis is rejected, when it is false, and when the critical value that is used is the one that will actually ensure that there is no size distortion. This new critical value requires some Monte Carlo simulation. In terms of a day-to-day application, we usually wouldn't do this. In that case, what's of interest is the "raw power" of the test,  namely the rate at which it rejects false null hypotheses when the asymptotic critical value is used. This is what we do below in the context described in Section \ref{emp-power}.

With this notion of power, we remark that the powers our tests become comparable to the standard ones, even when $n=500$.
\newpage
\begin{table}[h!]
 \caption{\small{Empirical power (in \%) of the modified and standard versions
 of the LB and BP tests in the case of ARMA$(2,1)$ model.
The number of replications is $N=1000$. }}
\begin{center}
\begin{tabular}{ccc ccc ccc}
\hline\hline \\
Model& Length $n$ & Lag $m$ & $\mathrm{{LB}}_{\textsc{sn}}$&$\mathrm{BP}_{\textsc{sn}}$&$\mathrm{{LB}}_{\textsc{frz}}$&$\mathrm{BP}_{\textsc{frz}}$&$\mathrm{{LB}}_{\textsc{s}}$&$\mathrm{BP}_{\textsc{s}}$
\vspace*{0.2cm}\\\hline
&& $1$&77.6 &77.6 &91.4& 91.3&n.a. &n.a.\\
&& $2$&64.6 & 64.6 & 86.3 & 86.2&n.a. &n.a.\\
I &$n=500$& $3$&59.8& 59.5& 85.2& 85.0& 93.9& 93.8\\
 &&$6$&50.2& 49.5& 77.5& 76.7& 83.3& 83.0\\
 &&$12$&36.6& 35.9& 65.1& 63.8& 73.0& 72.3\\
 \cline{2-9}
&& $1$&99.7 & 99.7 &100.0 &100.0&n.a. &n.a.\\
&& $2$&97.5 & 97.5&100.0& 100.0&n.a. &n.a.\\
I &$n=2,000$& $3$&96.8 & 96.8& 100.0 &100.0& 100.0& 100.0\\
 &&$6$&96.9 & 96.9 &100.0& 100.0& 100.0& 100.0\\
 &&$12$&95.4&  95.3& 100.0& 100.0& 100.0& 100.0\\
  \cline{2-9}
&& $1$&100.0& 100.0 &99.9 &100.0&n.a. &n.a.\\
&& $2$&100.0& 100.0 &100.0& 100.0 &n.a. &n.a.\\
I &$n=10,000$& $3$&100.0& 100.0 &100.0 &100.0& 100.0& 100.0\\
 &&$6$&100.0& 100.0 &100.0 &100.0& 100.0& 100.0\\
 &&$12$&100.0& 100.0 &100.0 &100.0& 100.0& 100.0\\

\hline
&& $1$&62.7& 62.4& 87.5& 87.3&n.a. &n.a.\\
&& $2$&47.6 & 47.2 & 79.1 & 78.9&n.a. &n.a.\\
II &$n=500$& $3$&40.4& 40.0& 75.5& 75.0& 92.6& 92.6\\
 &&$6$&28.4& 27.8& 60.2& 59.8& 84.4& 83.9\\
 &&$12$&12.8& 12.2& 41.0& 39.5& 77.0& 76.3
\\
 \cline{2-9}
&& $1$&97.9 & 97.9 &100.0& 100.0 &n.a. &n.a.\\
&& $2$&92.0 & 92.0& 100.0& 100.0&n.a. &n.a.\\
II &$n=2,000$& $3$&90.7 & 90.7& 100.0& 100.0 &100.0 &100.0\\
 &&$6$&89.3 & 89.2&  99.8 & 99.8& 100.0 &100.0\\
 &&$12$&83.2& 83.2& 99.2& 99.2& 99.9& 99.9\\
  \cline{2-9}
&& $1$&100.0 &100.0&100.0 &100.0&n.a. &n.a.\\
&& $2$&100.0 &100.0& 100.0& 100.0&n.a. &n.a.\\
II &$n=10,000$& $3$&99.9& 99.9& 100.0& 100.0& 100.0& 100.0\\
 &&$6$&99.9 & 99.9& 100.0 &100.0 &100.0& 100.0\\
 &&$12$&100.0& 100.0& 100.0& 100.0& 100.0& 100.0\\
  \hline
&& $1$&51.0& 50.5& 75.6& 75.3&n.a. &n.a.\\
&& $2$&38.1 & 37.7&  69.5 & 69.3&n.a. &n.a.\\
III &$n=500$& $3$&31.0& 30.8& 69.3& 68.9& 88.5& 88.4\\
 &&$6$&20.0& 19.3& 58.3& 57.8& 78.3& 77.8\\
 &&$12$&8.8&  8.4& 45.3& 44.1& 68.5& 67.7\\
 \cline{2-9}
&& $1$&95.2& 95.2& 99.8& 99.8&n.a. &n.a.\\
&& $2$&90.5 & 90.5 & 99.8 & 99.8&n.a. &n.a.\\
III &$n=2,000$& $3$&85.7 & 85.7 & 99.9 & 99.9& 100.0& 100.0
\\
 &&$6$&80.0& 80.0& 99.7& 99.7& 99.9 &99.9\\
 &&$12$&71.9& 71.9& 99.6& 99.6& 99.8& 99.8\\
  \cline{2-9}
&& $1$&100.0& 100.0 & 100.0 &100.0&n.a. &n.a.\\
&& $2$&100.0& 100.0 &100.0& 100.0&n.a. &n.a.\\
III &$n=10,000$& $3$&99.5 & 99.5& 100.0 &100.0 &100.0 &100.0\\
 &&$6$&99.7 & 99.7& 100.0 &100.0 &100.0& 100.0\\
 &&$12$&100.0& 100.0 & 100.0 &100.0 &100.0 &100.0\\
\hline\hline \\
\multicolumn{9}{l}{I: Strong ARMA$(2,1)$ model
(\ref{lutkepohl})-(\ref{bruitARCH}) with $(\alpha_1,\beta_1)=(0,0)$}\\
\multicolumn{9}{l}{II: Weak ARMA$(2,1)$ model
(\ref{lutkepohl})-(\ref{bruitARCH}) with $(\alpha_1,\beta_1)=(0.1,0.85)$}\\
\multicolumn{9}{l}{III: Weak ARMA$(2,1)$ model
(\ref{lutkepohl})-(\ref{PT}).}\\
\end{tabular}
\end{center}
\label{tab2ARMAbis}
\end{table}

\begin{table}[hbt]
 \caption{\small{Empirical power (in \%) of the modified and standard  versions
 of the LB and BP tests in the case of ARMA$(2,1)$ model.
The number of replications is $N=1000$. }}
\begin{center}
\begin{tabular}{ccc ccc ccc}
\hline\hline \\
Model& Length $n$ & Lag $m$ & $\mathrm{{LB}}_{\textsc{sn}}$&$\mathrm{BP}_{\textsc{sn}}$&$\mathrm{{LB}}_{\textsc{frz}}$&$\mathrm{BP}_{\textsc{frz}}$&$\mathrm{{LB}}_{\textsc{s}}$&$\mathrm{BP}_{\textsc{s}}$
\vspace*{0.2cm}\\\hline
&& $1$&39.8& 39.4& 62.0 &61.4&n.a. &n.a.\\
&& $2$&29.7 & 29.6 & 56.5 & 56.5&n.a. &n.a.\\
IV &$n=500$& $3$&23.5& 23.2& 56.5& 56.4& 86.9& 86.8\\
 &&$6$&13.0 &12.4 &45.2& 44.7 &75.8& 75.7\\
 &&$12$&1.9 & 1.4& 27.9& 26.6& 65.8& 64.8\\
 \cline{2-9}
&& $1$&84.9& 84.9& 98.6 &98.6 &n.a. &n.a.\\
&& $2$&76.8 & 76.7 & 98.4 & 98.4&n.a. &n.a.\\
IV &$n=2,000$& $3$&72.4&  72.4 & 98.4 & 98.4& 100.0 &100.0
\\
 &&$6$&1.2& 61.1& 97.8& 97.7& 99.7& 99.7\\
 &&$12$&48.0& 47.6& 96.4& 96.4& 99.5& 99.5\\
  \cline{2-9}
&& $1$&100.0& 100.0 &100.0 &100.0&n.a. &n.a.\\
&& $2$&99.5& 99.5 & 100.0 &100.0 &n.a. &n.a.\\
IV &$n=10,000$& $3$&99.2  &99.2& 100.0& 100.0& 100.0 &100.0\\
 &&$6$&98.6 & 98.6&100.0 &100.0 &100.0 &100.0\\
 &&$12$&99.3 & 99.3 &100.0 &100.0 &100.0 &100.0\\

\hline
&& $1$&84.6& 84.5& 99.0 &99.0&n.a. &n.a.\\
&& $2$&74.8 & 74.5 & 94.7 & 94.6&n.a. &n.a.\\
V &$n=500$& $3$&70.5& 70.3& 94.3& 94.3& 97.0& 96.9\\
 &&$6$&60.4& 59.8& 84.6& 84.2& 85.3& 84.9\\
 &&$12$&46.5& 45.9& 69.3& 68.0& 73.7& 72.7\\
 \cline{2-9}
&& $1$&100.0& 100.0& 100.0& 100.0&n.a. &n.a.\\
&& $2$& 99.4&  99.4& 100.0& 100.0&n.a. &n.a.\\
V &$n=2,000$& $3$&99.2 & 99.2& 100.0& 100.0& 100.0& 100.0\\
 &&$6$&98.7 & 98.7& 100.0& 100.0&100.0 &100.0\\
 &&$12$&99.2 & 99.1 & 99.9 & 99.9&100.0 &100.0\\
  \cline{2-9}
&& $1$&100.0& 100.0& 100.0& 100.0&n.a. &n.a.\\
&& $2$&100.0& 100.0& 100.0& 100.0&n.a. &n.a.\\
V &$n=10,000$& $3$&100.0& 100.0& 100.0& 100.0& 100.0& 100.0\\
 &&$6$&100.0& 100.0& 100.0& 100.0&100.0 &100.0\\
 &&$12$&100.0& 100.0& 100.0& 100.0&100.0 &100.0\\
\hline\hline \\
\multicolumn{9}{l}{IV: Weak ARMA$(2,1)$ model
(\ref{lutkepohl})-(\ref{PTcarre}).}\\
\multicolumn{9}{l}{V: Weak ARMA$(2,1)$ model
(\ref{lutkepohl})-(\ref{RT}).}\\
\end{tabular}
\end{center}
\label{tab2ARMAsuitebis}
\end{table}

\
\newpage

\section{Empirical power for VARMA models}
Here we uses both empirical critical values (leading to the so called self-adjusted power) and asymptotic critical values (leading to the so-called raw-power) to compare the power of our test when one deals with VARMA models.

For that sake, we simulate $N=1,000$ independent trajectories of size $n=10,000$ of the following bivariate VARMA$(2,1)$ defined by
\begin{align}
\label{lutkepohlVARMA}
\left(\begin{array}{c}X_{1,t}\\X_{2,t}\end{array}\right)
&=\left(\begin{array}{cc}1.2&0.6\\-0.5&0.3\end{array}\right)
\left(\begin{array}{c}X_{1,t-1}\\X_{2,t-1}\end{array}\right)+\left(\begin{array}{cc}0.1&0.0\\0.0&0.1\end{array}\right)
\left(\begin{array}{c}X_{1,t-2}\\X_{2,t-2}\end{array}\right)+\left(\begin{array}{c}\epsilon_{1,t}\\\epsilon_{2,t}\end{array}\right)\nonumber\\
&\qquad -\left(\begin{array}{cc}-0.6&0.3\\0.3&0.6\end{array}\right)
\left(\begin{array}{c}\epsilon_{1,t-1}\\\epsilon_{2,t-1}\end{array}\right),
\end{align}
where the multivariate, strong and weak versions of the innovation process ${\epsilon}$ are defined in  Section \ref{multiARMA}.

For each of these $N$ replications we fit a VARMA$(1,1)$ model and we perform standard and modified test based on $m=1,\dots,5$  residual autocorrelations.

Tables \ref{size  adjusted-VARMA1} and \ref{size  adjusted-VARMA2} display the relative
rejection frequencies of over the $N$ independent replications for the VARMA models when one uses the empirical critical value (size adjusted power of the test).

Tables \ref{raw-test1} and \ref{raw-test2} display the relative
rejection frequencies of over the $N$ independent replications for the VARMA models when one uses the asymptotic critical value (raw power of the test).

In these examples, the standard and modified versions of the tests have very similar powers when $n\ge 2000$.
When $n=500$, the empirical size of the tests for the models II, III and IV  are very far from the $5\%$ nominal level (see Tables \ref{tab1VARMA} and \ref{tab1VARMAsuite}) so the discussion on the power is not relevant. However, the results for the model V (for which the empirical level is close to the nominal)  are quite satisfactory whereas it is less the case for the strong model I.

\newpage
\begin{table}[h!]
 \caption{\small{Empirical size  adjusted power (in \%) of the  modified and standard versions
 of the LB and BP tests at the 5\% nominal level in the case of VARMA$(2,1)$ model.
The number of replications is $N=1000$. }}
\begin{center}
\begin{tabular}{ccc ccc ccc}
\hline\hline \\
Model& Length $n$ & Lag $m$ & $\mathrm{{{LB}}}_{\textsc{sn}}$&$\mathrm{{{BP}}}_{\textsc{sn}}$&$\mathrm{{{LB}}}_{\textsc{bm}}$&
$\mathrm{{{BP}}}_{\textsc{bm}}$&$\mathrm{{{LB}}}_{\textsc{s}}$&$\mathrm{{{BP}}}_{\textsc{s}}$
\vspace*{0.2cm}\\\hline
&& $1$&29.0 &29.0 &99.9 &99.9&  n.a. &n.a.\\
&& $2$&21.4 & 21.4 & 99.8 & 99.8&   n.a. &n.a.\\
I &$n=500$& $3$&20.7& 20.7& 99.5& 99.5 &37.4& 37.5\\
 &&$4$&15.6& 15.6& 99.2& 99.2& 31.8& 31.9\\
 &&$5$&15.5& 15.5& 99.1& 99.1& 29.4& 29.6\\
 \cline{2-9}
&& $1$&85.0 & 85.0& 100.0 &100.0&  n.a. &n.a.\\
&& $2$&78.5 & 78.5 & 100.0 &100.0&  n.a. &n.a.\\
I &$n=2,000$& $3$&76.5&  76.5& 100.0& 100.0 & 98.7 & 98.7\\
 &&$4$&69.2 & 69.2& 100.0& 100.0 & 97.5 & 97.5\\
 &&$5$&65.4 & 65.1& 100.0& 100.0 & 96.0 & 96.0\\
  \cline{2-9}
&& $1$&99.8 & 99.8 &100.0 &100.0&  n.a. &n.a.\\
&& $2$&99.9 & 99.9&100.0& 100.0 &  n.a. &n.a.\\
I &$n=10,000$& $3$&99.9 & 99.9&100.0 &100.0& 100.0& 100.0\\
 &&$4$&100.0& 100.0 &100.0 &100.0& 100.0& 100.0\\
 &&$5$&100.0& 100.0 &100.0 &100.0& 100.0& 100.0\\

\hline
&& $1$&24.0& 24.0& 99.1& 99.1&  n.a. &n.a.\\
&& $2$&19.8 & 19.7  &98.6 & 98.7&  n.a. &n.a.\\
II &$n=500$& $3$&18.9& 18.7& 98.7& 98.7& 28.8& 28.8\\
 &&$4$&14.0& 13.7& 98.0& 98.0& 26.0& 26.0\\
 &&$5$&11.8& 11.6& 97.3& 97.4& 19.7 &19.7\\
 \cline{2-9}
&& $1$&63.7& 63.7& 99.9& 99.9 &  n.a. &n.a.\\
&& $2$&53.5 & 53.5 & 99.8 & 99.8&  n.a. &n.a.\\
II &$n=2,000$& $3$&53.1& 52.9& 99.8& 99.8& 82.2& 82.2\\
 &&$4$&46.8& 46.8& 99.7& 99.7& 77.1& 77.1\\
 &&$5$&42.2& 42.1& 99.9& 99.9& 76.9& 76.9
\\
  \cline{2-9}
&& $1$&98.8 & 98.8&100.0& 100.0&  n.a. &n.a.\\
&& $2$&98.9 & 98.9 & 100.0& 100.0&  n.a. &n.a.\\
II &$n=10,000$& $3$&99.0 & 99.0&100.0& 100.0&100.0& 100.0\\
 &&$4$&98.8 & 98.8 & 100.0& 100.0&100.0& 100.0\\
 &&$5$&98.4 & 98.4&100.0& 100.0&100.0& 100.0\\
  \hline
&& $1$&21.0& 21.0& 98.1 &98.1&   n.a. &n.a.\\
&& $2$&18.4 & 18.4 & 98.2 & 98.0&  n.a. &n.a.\\
III &$n=500$& $3$&16.3& 16.3& 98.3& 98.0& 22.2& 22.2\\
 &&$4$&12.5& 12.2& 97.0& 96.9& 20.1& 20.1\\
 &&$5$&12.0& 12.0& 96.5& 96.6& 20.0& 20.2\\
 \cline{2-9}
&& $1$&63.2 & 63.2 &100.0 &100.0&   n.a. &n.a.\\
&& $2$&64.2 & 64.2  &100.0 &100.0 &  n.a. &n.a.\\
III &$n=2,000$& $3$&59.2& 59.2& 99.8& 99.8& 73.6& 73.6\\
 &&$4$&53.0 &52.9 &99.9& 99.9& 73.1& 73.1\\
 &&$5$&51.2& 51.0 &99.7 &99.7& 71.1& 71.1\\
  \cline{2-9}
&& $1$&99.0 & 99.0 &100.0& 100.0&  n.a. &n.a.\\
&& $2$&99.6 & 99.6 &100.0 &100.0 &  n.a. &n.a.\\
III &$n=10,000$& $3$&99.4 & 99.4 & 100.0 &100.0 &100.0 &100.0\\
 &&$4$&99.7 & 99.7& 100.0 &100.0 &100.0& 100.0\\
 &&$5$&99.5 & 99.5& 100.0 &100.0 &100.0 &100.0\\
\hline\hline \\
\multicolumn{9}{l}{I: Strong VARMA$(2,1)$ model
(\ref{lutkepohlVARMA})-(\ref{bruitfort}).}\\
\multicolumn{9}{l}{II: Weak VARMA$(2,1)$ model
(\ref{lutkepohlVARMA})-(\ref{ARCH1}).}\\
\multicolumn{9}{l}{III: Weak VARMA$(2,1)$ model
(\ref{lutkepohlVARMA})-(\ref{multiPT}).}\\
\end{tabular}
\end{center}
\label{size  adjusted-VARMA1}
\end{table}

\begin{table}[h!]
 \caption{\small{Empirical size  adjusted power (in \%) of the modified and  standard versions
 of the LB and BP tests at the 5\% nominal level in the case of VARMA$(2,1)$ model.
The number of replications is $N=1000$. }}
\begin{center}
\begin{tabular}{ccc ccc ccc}
\hline\hline \\
Model& Length $n$ & Lag $m$ & $\mathrm{{{LB}}}_{\textsc{sn}}$&$\mathrm{{{BP}}}_{\textsc{sn}}$&$\mathrm{{{LB}}}_{\textsc{bm}}$&
$\mathrm{{{BP}}}_{\textsc{bm}}$&$\mathrm{{{LB}}}_{\textsc{s}}$&$\mathrm{{{BP}}}_{\textsc{s}}$
\vspace*{0.2cm}\\\hline
&& $1$&28.6& 28.6& 98.2& 98.3 & n.a.& n.a.\\
&& $2$&21.3 & 21.3 & 96.9 & 96.9& n.a.& n.a.\\
IV &$n=500$& $3$&13.6& 13.8& 97.5& 97.5& 15.4& 15.5
\\
 &&$4$&11.9& 12.1& 97.3& 97.5 &13.6& 13.6\\
 &&$5$&13.6& 13.6& 97.2& 97.1& 13.2& 13.2\\
 \cline{2-9}
&& $1$&65.6& 65.6& 99.7& 99.7 & n.a.& n.a.\\
&& $2$&63.1 & 63.1 & 99.5 & 99.5&  n.a.& n.a.\\
IV &$n=2,000$& $3$&56.7& 56.7& 98.5& 98.6& 48.4& 48.5\\
 &&$4$&50.9& 50.9 &99.0 &99.0 &49.8& 49.8\\
 &&$5$&47.7& 47.7 &98.4& 98.4& 47.5& 47.4\\
  \cline{2-9}
&& $1$&97.3&  97.3& 100.0 &100.0&  n.a.& n.a.\\
&& $2$&98.6 & 98.6& 100.0& 100.0&  n.a.& n.a.\\
IV &$n=10,000$& $3$&98.4 & 98.4& 100.0& 100.0& 100.0& 100.0\\
 &&$4$&98.9 & 98.9& 100.0& 100.0& 100.0& 100.0\\
 &&$5$&98.8 & 98.9&100.0 &100.0 &100.0 &100.0\\

\hline
&& $1$&38.8& 38.8& 99.7& 99.7 & n.a.& n.a.\\
&& $2$&28.7 & 28.8 & 99.6 & 99.6  &  n.a.& n.a.\\
V &$n=500$& $3$&25.9& 25.9& 99.7& 99.7& 41.5& 41.6\\
 &&$4$&22.9& 23.2 &99.2& 99.2& 35.2& 35.2\\
 &&$5$&20.3& 20.4& 99.3& 99.3& 30.0& 30.3\\
 \cline{2-9}
&& $1$&88.8 & 88.8 &100.0 &100.0 & n.a.& n.a.\\
&& $2$&87.1 & 87.1 &100.0 &100.0 & n.a.& n.a.\\
V &$n=2,000$& $3$&85.4 & 85.4& 100.0& 100.0 & 99.2 & 99.2\\
 &&$4$& 80.1 & 80.1& 100.0& 100.0 & 98.3 & 98.3\\
 &&$5$&79.1&  79.1& 100.0& 100.0 & 97.0 & 97.0\\
  \cline{2-9}
&& $1$&100.0& 100.0& 100.0& 100.0& n.a.& n.a.\\
&& $2$&100.0& 100.0& 100.0& 100.0& n.a.& n.a.\\
V &$n=10,000$& $3$&99.9 & 99.9& 100.0& 100.0& 100.0& 100.0\\
 &&$4$&100.0& 100.0& 100.0& 100.0&100.0 &100.0\\
 &&$5$&100.0& 100.0& 100.0& 100.0&100.0 &100.0\\
\hline\hline \\
\multicolumn{9}{l}{IV: Weak VARMA$(2,1)$ model
(\ref{lutkepohlVARMA})-(\ref{multiPTcarre}).}\\
\multicolumn{9}{l}{V: Weak VARMA$(2,1)$ model
(\ref{lutkepohlVARMA})-(\ref{multiRT}).}\\
\end{tabular}
\end{center}
\label{size  adjusted-VARMA2}
\end{table}

\begin{table}[h!]
 \caption{\small{Empirical power (in \%) of the modified  and standard  versions
 of the LB and BP tests in the case of VARMA$(2,1)$ model.
The number of replications is $N=1000$. }}
\begin{center}
\begin{tabular}{ccc ccc ccc}
\hline\hline \\
Model& Length $n$ & Lag $m$ & $\mathrm{{{LB}}}_{\textsc{sn}}$&$\mathrm{{{BP}}}_{\textsc{sn}}$&$\mathrm{{{LB}}}_{\textsc{bm}}$&
$\mathrm{{{BP}}}_{\textsc{bm}}$&$\mathrm{{{LB}}}_{\textsc{s}}$&$\mathrm{{{BP}}}_{\textsc{s}}$
\vspace*{0.2cm}\\\hline
&& $1$&28.4 &28.6 &41.6& 41.4 &n.a. &n.a.\\
&& $2$&21.7&  21.8 & 31.3 & 30.3 &n.a. &n.a.\\
I &$n=500$& $3$&16.6& 16.8& 27.4& 26.9 &66.3& 66.0\\
 &&$4$&11.8& 11.9& 23.8& 23.4& 48.9 &48.3\\
 &&$5$&10.9 &11.1 &20.7 &20.0& 40.0 &38.8\\
 \cline{2-9}
&& $1$&83.6& 83.7& 94.2 &94.2&n.a. &n.a.\\
&& $2$&77.4 & 77.4 & 87.4 & 87.2&n.a. &n.a.\\
I &$n=2,000$& $3$&73.5& 73.6& 88.0& 88.0& 99.8& 99.8\\
 &&$4$&69.2& 69.3& 86.4& 86.4& 98.7& 98.7\\
 &&$5$& 64.5& 64.6& 85.6& 85.6& 97.7& 97.7\\
  \cline{2-9}
&& $1$&99.8 & 99.8&100.0 &100.0&n.a. &n.a.\\
&& $2$&99.9 & 99.9 &100.0& 100.0 &n.a. &n.a.\\
I &$n=10,000$& $3$&99.9 & 99.9 &100.0 &100.0& 100.0& 100.0\\
 &&$4$&100.0& 100.0 &99.9 & 99.9 & 100.0& 100.0\\
 &&$5$&100.0& 100.0 &100.0 &100.0& 100.0& 100.0\\

\hline
&& $1$&20.5 &20.6& 37.4& 36.9 &n.a. &n.a.\\
&& $2$&14.4 & 14.4 & 27.4 & 26.7&n.a. &n.a.\\
II &$n=500$& $3$&9.2&  9.2& 21.0& 20.6& 74.6& 74.6\\
 &&$4$&6.0&  6.4& 15.7 &15.3& 62.6 &61.7\\
 &&$5$&3.8 & 4.0& 11.0& 10.8& 51.5& 50.3\\
 \cline{2-9}
&& $1$&60.8& 60.9& 89.7& 89.7&n.a. &n.a.\\
&& $2$&53.1 & 53.3 & 81.9 & 81.9 &n.a. &n.a.\\
II &$n=2,000$& $3$&45.3& 45.3& 79.4 &79.3& 99.0& 99.0\\
 &&$4$&41.3& 41.3& 77.2& 77.0& 97.3& 97.3\\
 &&$5$&35.4& 35.4 &73.4& 73.0 &96.2 &96.2
\\
  \cline{2-9}
&& $1$&98.8& 98.8 &99.5& 99.5 &n.a. &n.a.\\
&& $2$&98.9 & 98.9 & 99.4 & 99.4&n.a. &n.a.\\
II &$n=10,000$& $3$&98.8 & 98.8&  99.4 & 99.4& 100.0& 100.0\\
 &&$4$&98.6 & 98.6 & 99.3&  99.3& 100.0& 100.0\\
 &&$5$&98.1 & 98.1 & 99.4&  99.4& 100.0& 100.0\\
  \hline
&& $1$&17.6& 17.6& 22.9 &22.9 &n.a. &n.a.\\
&& $2$&8.3  & 8.3 & 15.2 & 14.9&n.a. &n.a.\\
III &$n=500$& $3$&2.4 & 2.4 & 7.4&  7.3& 83.2& 82.8\\
 &&$4$&0.9 & 0.9 & 3.7&  3.5& 71.9& 71.2\\
 &&$5$&0.3 & 0.3 & 3.3 & 3.2 &60.4 &59.4\\
 \cline{2-9}
&& $1$&62.8 &62.8& 78.8& 78.8&n.a. &n.a.\\
&& $2$&57.9  &58.0  &63.1 & 63.1&n.a. &n.a.\\
III &$n=2,000$& $3$&49.1& 49.1 &59.0& 58.9& 99.7& 99.7\\
 &&$4$&40.6& 40.7& 49.7& 49.6& 98.5 &98.5\\
 &&$5$&35.3& 35.2& 41.4& 41.3 &96.5 &96.5\\
  \cline{2-9}
&& $1$&99.0 & 99.0 &100.0 &100.0&n.a. &n.a.\\
&& $2$&99.7 & 99.7 & 99.9 & 99.9  &n.a. &n.a.\\
III &$n=10,000$& $3$&99.3 & 99.3 & 99.9 & 99.9 &100.0& 100.0\\
 &&$4$&99.6 & 99.6 & 99.9&  99.9& 100.0& 100.0\\
 &&$5$&99.1 & 99.1 & 99.9 & 99.9& 100.0 &100.0\\
\hline\hline \\
\multicolumn{9}{l}{I: Strong VARMA$(2,1)$ model
(\ref{lutkepohlVARMA})-(\ref{bruitfort}).}\\
\multicolumn{9}{l}{II: Weak VARMA$(2,1)$ model
(\ref{lutkepohlVARMA})-(\ref{ARCH1}).}\\
\multicolumn{9}{l}{III: Weak VARMA$(2,1)$ model
(\ref{lutkepohlVARMA})-(\ref{multiPT}).}\\
\end{tabular}
\end{center}
\label{raw-test1}
\end{table}

\begin{table}[h!]
 \caption{\small{Empirical power (in \%) of the modified  and standard  versions
 of the LB and BP tests in the case of VARMA$(2,1)$ model.
The number of replications is $N=1000$. }}
\begin{center}
\begin{tabular}{ccc ccc ccc}
\hline\hline \\
Model& Length $n$ & Lag $m$ & $\mathrm{{{LB}}}_{\textsc{sn}}$&$\mathrm{{{BP}}}_{\textsc{sn}}$&$\mathrm{{{LB}}}_{\textsc{bm}}$&
$\mathrm{{{BP}}}_{\textsc{bm}}$&$\mathrm{{{LB}}}_{\textsc{s}}$&$\mathrm{{{BP}}}_{\textsc{s}}$
\vspace*{0.2cm}\\\hline
&& $1$&15.6& 15.7& 22.6& 22.4& n.a.& n.a.\\
&& $2$&5.7  & 6.0 & 14.5 & 14.4& n.a.& n.a.\\
IV &$n=500$& $3$&1.7 & 1.8 & 9.6 & 9.6 &82.7& 82.2
\\
 &&$4$&0.4 & 0.4 &10.3& 10.1& 70.5& 70.0\\
 &&$5$&0.8 & 0.8 & 8.5 & 8.5& 61.8& 61.4\\
 \cline{2-9}
&& $1$&59.0& 59.2& 62.7 &62.7& n.a.& n.a.\\
&& $2$&54.9 & 54.9 & 44.5 & 44.5&  n.a.& n.a.\\
IV &$n=2,000$& $3$&46.1& 46.2& 36.7& 36.6 &99.3& 99.3\\
 &&$4$&35.5& 35.5& 29.3& 29.3& 98.5& 98.5\\
 &&$5$&25.5& 25.8& 23.3 &23.2 &96.5 &96.5\\
  \cline{2-9}
&& $1$&97.3 &97.3 &98.6 &98.6& n.a.& n.a.\\
&& $2$&97.8 & 97.8 & 96.9 & 96.9 & n.a.& n.a.\\
IV &$n=10,000$& $3$&97.9 & 97.9 & 96.5 & 96.5 &100.0 &100.0\\
 &&$4$&98.4 & 98.4 & 96.3 & 96.3 &100.0 &100.0\\
 &&$5$&97.4 & 97.4 & 96.1 & 96.1& 100.0& 100.0\\

\hline
&& $1$&34.1& 34.1 &48.7 &48.6 & n.a.& n.a.\\
&& $2$&27.0 & 27.1 & 33.2 & 32.9& n.a.& n.a.\\
V &$n=500$& $3$&23.2& 23.3 &27.3& 26.5& 67.7 &67.3\\
 &&$4$&18.6& 19.0& 23.3& 23.0& 46.7& 45.9\\
 &&$5$&15.0& 15.3& 20.8& 20.2& 38.5& 37.1\\
 \cline{2-9}
&& $1$&91.2 &91.2& 98.9& 98.9& n.a.& n.a.\\
&& $2$&87.4 & 87.4  &94.7 & 94.7& n.a.& n.a.\\
V &$n=2,000$& $3$&84.1& 84.2& 93.8& 93.8& 99.9& 99.9\\
 &&$4$& 79.8& 79.8& 92.1& 92.1& 99.2& 99.2\\
 &&$5$&76.7& 76.8 &90.8 &90.8& 98.3& 98.3\\
  \cline{2-9}
&& $1$&100.0& 100.0& 100.0& 100.0& n.a.& n.a.\\
&& $2$&100.0& 100.0& 100.0& 100.0& n.a.& n.a.\\
V &$n=10,000$& $3$&99.9 & 99.9 & 100.0& 100.0& 100.0& 100.0\\
 &&$4$&100.0& 100.0& 100.0& 100.0&100.0 &100.0\\
 &&$5$&100.0& 100.0& 100.0& 100.0&100.0 &100.0\\
\hline\hline \\
\multicolumn{9}{l}{IV: Weak VARMA$(2,1)$ model
(\ref{lutkepohlVARMA})-(\ref{multiPTcarre}).}\\
\multicolumn{9}{l}{V: Weak VARMA$(2,1)$ model
(\ref{lutkepohlVARMA})-(\ref{multiRT}).}\\
\end{tabular}
\end{center}
\label{raw-test2}
\end{table}


\begin{thebibliography}{}

\bibitem[Andrews, 1991]{a_econ}
Andrews, D. W.~K. (1991).
\newblock Heteroskedasticity and autocorrelation consistent covariance matrix
  estimation.
\newblock {\em Econometrica}, 59(3):817--858.

\bibitem[Bauwens et~al., 2006]{blr2006}
Bauwens, L., Laurent, S., and Rombouts, J. V.~K. (2006).
\newblock Multivariate {GARCH} models: a survey.
\newblock {\em J. Appl. Econometrics}, 21(1):79--109.

\bibitem[Berk, 1974]{berk}
Berk, K.~N. (1974).
\newblock Consistent autoregressive spectral estimates.
\newblock {\em Ann. Statist.}, 2:489--502.
\newblock Collection of articles dedicated to Jerzy Neyman on his 80th
  birthday.

\bibitem[Boubacar~Mainassara, 2011]{yac}
Boubacar~Mainassara, Y. (2011).
\newblock Multivariate portmanteau test for structural {VARMA} models with
  uncorrelated but non-independent error terms.
\newblock {\em J. Statist. Plann. Inference}, 141(8):2961--2975.

\bibitem[Boubacar~Mainassara and Francq, 2011]{yac2}
Boubacar~Mainassara, Y. and Francq, C. (2011).
\newblock Estimating structural {VARMA} models with uncorrelated but
  non-independent error terms.
\newblock {\em J. Multivariate Anal.}, 102(3):496--505.

\bibitem[Box and Pierce, 1970]{bp70}
Box, G. E.~P. and Pierce, D.~A. (1970).
\newblock Distribution of residual autocorrelations in
  autoregressive-integrated moving average time series models.
\newblock {\em J. Amer. Statist. Assoc.}, 65:1509--1526.

\bibitem[Brockwell and Davis, 1991]{broc-d}
Brockwell, P.~J. and Davis, R.~A. (1991).
\newblock {\em Time series: theory and methods}.
\newblock Springer Series in Statistics. Springer-Verlag, New York, second
  edition.

\bibitem[Chitturi, 1974]{C1974}
Chitturi, R.~V. (1974).
\newblock Distribution of residual autocorrelations in multiple autoregressive
  schemes.
\newblock {\em J. Amer. Statist. Assoc.}, 69:928--934.

\bibitem[Davydov, 1968]{davy}
Davydov, J.~A. (1968).
\newblock The convergence of distributions which are generated by stationary
  random processes.
\newblock {\em Teor. Verojatnost. i Primenen.}, 13:730--737.

\bibitem[den Haan and Levin, 1997]{haan}
den Haan, W.~J. and Levin, A.~T. (1997).
\newblock A practitioner's guide to robust covariance matrix estimation.
\newblock In {\em Robust inference}, volume~15 of {\em Handbook of Statist.},
  pages 299--342. North-Holland, Amsterdam.

\bibitem[Dufour and Jouini, 2014]{DufourJouini2014}
Dufour, J.-M. and Jouini, T. (2014).
\newblock Asymptotic distributions for quasi-efficient estimators in echelon
  {VARMA} models.
\newblock {\em Comput. Statist. Data Anal.}, 73:69--86.

\bibitem[Francq et~al., 2005]{frz}
Francq, C., Roy, R., and Zako{\"{\i}}an, J.-M. (2005).
\newblock Diagnostic checking in {ARMA} models with uncorrelated errors.
\newblock {\em J. Amer. Statist. Assoc.}, 100(470):532--544.

\bibitem[Francq and Zako{\"{\i}}an, 1998]{fz98}
Francq, C. and Zako{\"{\i}}an, J.-M. (1998).
\newblock Estimating linear representations of nonlinear processes.
\newblock {\em J. Statist. Plann. Inference}, 68(1):145--165.

\bibitem[Francq and Zako{\"{\i}}an, 2000]{fz00}
Francq, C. and Zako{\"{\i}}an, J.-M. (2000).
\newblock Covariance matrix estimation for estimators of mixing weak {ARMA}
  models.
\newblock {\em J. Statist. Plann. Inference}, 83(2):369--394.

\bibitem[Francq and Zako{\"{\i}}an, 2005]{fz05}
Francq, C. and Zako{\"{\i}}an, J.-M. (2005).
\newblock Recent results for linear time series models with non independent
  innovations.
\newblock In {\em Statistical modeling and analysis for complex data problems},
  volume~1 of {\em GERAD 25th Anniv. Ser.}, pages 241--265. Springer, New York.

\bibitem[Francq and Zako\"{\i}an, 2010]{FZ2010}
Francq, C. and Zako\"{\i}an, J.-M. (2010).
\newblock {\em GARCH Models: Structure, Statistical Inference and Financial
  Applications}.
\newblock Wiley.

\bibitem[Hannan, 1976]{H1976}
Hannan, E.~J. (1976).
\newblock The identification and parametrization of {ARMAX} and state space
  forms.
\newblock {\em Econometrica}, 44(4):713--723.

\bibitem[Herrndorf, 1984]{herr}
Herrndorf, N. (1984).
\newblock A functional central limit theorem for weakly dependent sequences of
  random variables.
\newblock {\em Ann. Probab.}, 12(1):141--153.

\bibitem[Hong, 1996]{H1996}
Hong, Y. (1996).
\newblock Consistent testing for serial correlation of unknown form.
\newblock {\em Econometrica}, 64(4):837--864.

\bibitem[Hosking, 1980]{H1980}
Hosking, J. R.~M. (1980).
\newblock The multivariate portmanteau statistic.
\newblock {\em J. Amer. Statist. Assoc.}, 75(371):602--608.

\bibitem[Imhof, 1961]{i1961}
Imhof, J.~P. (1961).
\newblock Computing the distribution of quadratic forms in normal variables.
\newblock {\em Biometrika}, 48:419--426.

\bibitem[Jeantheau, 1998]{j1998}
Jeantheau, T. (1998).
\newblock Strong consistency of estimators for multivariate {ARCH} models.
\newblock {\em Econometric Theory}, 14(1):70--86.

\bibitem[Kuan and Lee, 2006]{kl2006}
Kuan, C.-M. and Lee, W.-M. (2006).
\newblock Robust {$M$} tests without consistent estimation of the asymptotic
  covariance matrix.
\newblock {\em J. Amer. Statist. Assoc.}, 101(475):1264--1275.

\bibitem[Ljung and Box, 1978]{lb}
Ljung, G.~M. and Box, G. E.~P. (1978).
\newblock On a measure of lack of fit in time series models.
\newblock {\em Biometrika}, 65(2):pp. 297--303.

\bibitem[Lobato, 2001]{lobato}
Lobato, I.~N. (2001).
\newblock Testing that a dependent process is uncorrelated.
\newblock {\em J. Amer. Statist. Assoc.}, 96(455):1066--1076.

\bibitem[Lobato et~al., 2001]{lobatoNS2001}
Lobato, I.~N., Nankervis, J.~C., and Savin, N.~E. (2001).
\newblock Testing for autocorrelation using a modified {B}ox-{P}ierce {Q} test.
\newblock {\em Inter. Econ. Review}, 42(1):187--205.

\bibitem[Lobato et~al., 2002]{lobatoNS2002}
Lobato, I.~N., Nankervis, J.~C., and Savin, N.~E. (2002).
\newblock Testing for zero autocorrelation in the presence of statistical
  dependence.
\newblock {\em Econ. Theory}, 18(3):730--743.

\bibitem[L{\"u}tkepohl, 2005]{L2005}
L{\"u}tkepohl, H. (2005).
\newblock {\em New introduction to multiple time series analysis}.
\newblock Springer-Verlag, Berlin.

\bibitem[M\'elard et~al., 2006]{MelardRoySaidi2006}
M\'elard, G., Roy, R., and Saidi, A. (2006).
\newblock Exact maximum likelihood estimation of structured or unit root
  multivariate time series models.
\newblock {\em Comput. Statist. Data Anal.}, 50(11):2958--2986.

\bibitem[Newey and West, 1987]{newey}
Newey, W.~K. and West, K.~D. (1987).
\newblock A simple, positive semidefinite, heteroskedasticity and
  autocorrelation consistent covariance matrix.
\newblock {\em Econometrica}, 55(3):703--708.

\bibitem[Reinsel, 1997]{reinsel97}
Reinsel, G.~C. (1997).
\newblock {\em Elements of multivariate time series analysis}.
\newblock Springer Series in Statistics. Springer-Verlag, New York, second
  edition.

\bibitem[Romano and Thombs, 1996]{rt1996}
Romano, J.~P. and Thombs, L.~A. (1996).
\newblock Inference for autocorrelations under weak assumptions.
\newblock {\em J. Amer. Statist. Assoc.}, 91(434):590--600.

\bibitem[Shao, 010a]{s2010JRSSBa}
Shao, X. (2010a).
\newblock A self-normalized approach to confidence interval construction in
  time series.
\newblock {\em J. R. Stat. Soc. Ser. B Stat. Methodol.}, 72(3):343--366.

\bibitem[Shao, 010b]{s2010JRSSBb}
Shao, X. (2010b).
\newblock Corrigendum: {A} self-normalized approach to confidence interval
  construction in time series.
\newblock {\em J. R. Stat. Soc. Ser. B Stat. Methodol.}, 72(5):695--696.

\bibitem[Shao, 2011]{s2011ET}
Shao, X. (2011).
\newblock Testing for white noise under unknown dependence and its applications
  to diagnostic checking for time series models.
\newblock {\em Econometric Theory}, 27(2):312--343.

\bibitem[Shao, 2012]{shaox}
Shao, X. (2012).
\newblock Parametric inference in stationary time series models with dependent
  errors.
\newblock {\em Scand. J. Stat.}, 39(4):772--783.

\bibitem[Shao, 2015]{s2016}
Shao, X. (2015).
\newblock Self-normalization for time series: a review of recent developments.
\newblock {\em J. Amer. Statist. Assoc.}, 110(512):1797--1817.

\bibitem[Zhu and Li, 2015]{zl2015JE}
Zhu, K. and Li, W.~K. (2015).
\newblock A bootstrapped spectral test for adequacy in weak {ARMA} models.
\newblock {\em J. Econometrics}, 187(1):113--130.

\end{thebibliography}
\end{document}